\documentclass[a4paper]{amsart}

\usepackage[english]{babel}
\usepackage[utf8x]{inputenc}
\usepackage{amsmath,amssymb,amsthm}
\usepackage{graphicx,enumerate,url}
\usepackage[colorinlistoftodos]{todonotes}
\usepackage[line,matrix,frame,curve,poly]{xy}
\usepackage{color}
\usepackage{tikz}
\usetikzlibrary{calc,through,backgrounds}
\usepackage{hyperref}

\newtheorem{theorem}{Theorem}[section]
\newtheorem{proposition}[theorem]{Proposition}
\newtheorem{conjecture}[theorem]{Conjecture}
\newtheorem{corollary}[theorem]{Corollary}
\newtheorem{lemma}[theorem]{Lemma}

\newtheorem{openproblem}[theorem]{Open Problem}

\theoremstyle{definition}
\newtheorem{definition}[theorem]{Definition}
\newtheorem{example}[theorem]{Example}

\newtheorem{remark}[theorem]{Remark}  % wenn man dem Leser etwas mitteilen moechte

\def\oO{{\mathcal{O}}}

\def\cC{{\mathcal{C}}}

\def\fF{{\mathcal{F}}}

\def\RR{{\mathbb R}}
\def\CC{{\mathbb C}}
\def\CC{{\mathbb C}}
\def\ZZ{{\mathbb Z}}

\def\xx{{\mathbf x}}
\def\vv{{\mathbf v}}
\def\mm{{\mathbf m}}

\def\lessdot{<\hskip-6pt\cdot\hskip3pt}
\def\conv{{\operatorname{conv}}}

\def\Cat{{\operatorname{Cat}}}
\def\NN{\Delta^{N\!N}}
\def\NC{\Delta^{N\!C}}
\def\NCred{\widetilde{\Delta}^{N\!C}}
\def\WS{{\Delta}^{Sep}}
\def\WSred{\widetilde{\Delta}^{Sep}}

\def\Tab{\mathcal{T}}
\def\Tab{\mathcal{T}}
\newcommand{\defn}[1]{\emph{\color{red} #1}} % emphasis of a definition

\def\bijNN{\varphi_{N\!N}}
\def\bijNC{\varphi_{N\!C}}

\def\Tamari{\mathcal{T}}

\def\Hilb{\mathcal{H}}
\def\bend{{\mathbf{b}}}

\newcommand\figref{Figure~\ref}

\makeatletter
\def\saveenum{\xdef\@savedenum{\the\c@enumi\relax}}
\def\resetenum{\global\c@enumi\@savedenum}
\makeatother

\title[Noncrossing sets and a Gra\ss mannian associahedron]{Noncrossing sets and a\\ Gra\ss mann associahedron}

\author[F.~Santos]{Francisco Santos$^*$}
\address[F.~Santos]{Departamento de Matem\'aticas, Estad\'istica y Computaci\'on Universidad de Cantabria, Santander, Spain}
\email{francisco.santos@unican.es}
\thanks{$^\star$Supported by the Spanish Ministry of Science (MICINN) through grant MTM2011-22792, and by a Humboldt Research Award of the Alexander von Humboldt Foundation}

\author[C.~Stump]{Christian Stump$^\dagger$}
\address[C.~Stump]{Institut f\"ur Mathematik, Freie Universit\"at Berlin, Germany}
\email{christian.stump@fu-berlin.de}
\thanks{$^\dagger$Supported by the German Research Foundation DFG, grant STU 563/2-1 ``Coxeter-Catalan combinatorics".}

\author[V.~Welker]{Volkmar Welker}
\address[V.~Welker]{Fachbereich Mathematik und Informatik, Philipps-Universit\"at Marburg, Germany}
\email{welker@mathematik.uni-marburg.de}

\thanks{{\bf Acknowledgement:} The authors would like to thank Christian Haase and Raman Sanyal for valuable comments and discussions.
After the first version of this paper was submitted to the arXiv we became aware of the facts that our results substantially overlap with~\cite{PPS2010} 
and that Conjecture~\ref{conj:ws-sphere} was already posed in the preliminary version~\cite{HessHirsch2011} of~\cite{HessHirsch2013}. 
We thank David Speyer and Vic Reiner, respectively, for pointing these two facts to us.
We finally remark that the current version of this paper is only preliminary.
}

\date{\today}

\begin{document}

\begin{abstract}
  We study a natural generalization of the noncrossing relation between pairs of elements in $[n]$ to $k$-tuples in $[n]$ that was first considered by Petersen, Pylyavskyy, Speyer (2010). 
  We give an alternative approach to their result that the flag simplicial complex on $\binom{[n]}{k}$ induced by this relation is a regular, unimodular and flag triangulation of the order polytope of the poset given by the product $[k]\times[n-k]$ of two chains (also called Gelfand-Tsetlin polytope), and that it is the join of a simplex and a sphere (that is, it is a Gorenstein triangulation).
  We then observe that this already implies the existence of a flag simplicial polytope generalizing the dual associahedron, whose Stanley-Reisner ideal is an initial ideal of the Gra\ss mann-Pl\"ucker ideal, while previous constructions of such a polytope did not guarantee flagness nor reduced to the dual associahedron for $k=2$.
  On our way we provide general results about order polytopes and their triangulations.
  We call the simplicial complex the \emph{noncrossing complex}, and the polytope derived from it the dual \emph{Gra\ss mann associahedron}.
  We extend results of Petersen, Pylyavskyy, Speyer (2010) showing that the non-crossing complex and the Gra\ss mann associahedron naturally reflect the relations between Gra\ss mannians with different parameters, in particular the isomorphism $G_{k,n} \cong G_{n-k,n}$.
  Moreover, our approach allows us to show that the adjacency graph of the noncrossing complex admits a natural acyclic orientation that allows us to define a \emph{Gra\ss mann-Tamari order} on maximal noncrossing families.
  Finally, we look at the precise relation of the noncrossing complex and the weak separability complex of Leclerc, Zelevinsky (1998), see also Scott (2005) among others.
  We show that the weak separability complex is not only a subcomplex of the noncrossibg complex as noted by Petersen, Pylyavskyy, Speyer (2010) but actually the cyclically invariant part of it.
\end{abstract}

\maketitle

\setcounter{tocdepth}{1}

\vspace*{-15pt}
\tableofcontents

\vspace*{-15pt}

\section{Introduction and main results}

Let $[n]$ denote the (ordered) set $\{ 1,\ldots,n \}$ of the first $n$ positive integers.
%Pictorially, we consider these numbers being drawn on a line.
Two pairs $(i<i')$ and $(j<j')$ with $i \leq j$ are said to \defn{nest} if $i<j<j'<i'$ and \defn{cross} if $i<j<i'<j'$.
In other words, they nest and cross if the two arcs nest and, respectively, cross in the following picture,

\begin{center}
  \setlength{\unitlength}{13pt}
  \begin{picture}(12,2.2)
    \put(0,0){\hbox{$1$}}
    \put(1,0){\hbox{$\cdots$}}
    \put(2.5,0){\hbox{$i$}}
    \put(3.5,0){\hbox{$<$}}
    \put(4.75,0){\hbox{$j$}}
    \put(5.75,0){\hbox{$<$}}
    \put(7,0){\hbox{$j'$}}
    \put(8,0){\hbox{$<$}}
    \put(9.25,0){\hbox{$i'$}}
    \put(10.25,0){\hbox{$\cdots$}}
    \put(11.75,0){\hbox{$n$}}
    \qbezier(2.9,0.7)(6.0,3.2)(9.1,0.7)
    \qbezier(5.1,0.7)(6.0,2.2)(6.9,0.7)
  \end{picture}
  \quad\quad
  \begin{picture}(12,2.2)
    \put(0,0){\hbox{$1$}}
    \put(1,0){\hbox{$\cdots$}}
    \put(2.5,0){\hbox{$i$}}
    \put(3.5,0){\hbox{$<$}}
    \put(4.75,0){\hbox{$j$}}
    \put(5.75,0){\hbox{$<$}}
    \put(7,0){\hbox{$i'$}}
    \put(8,0){\hbox{$<$}}
    \put(9.25,0){\hbox{$j'$}}
    \put(10.25,0){\hbox{$\cdots$}}
    \put(11.75,0){\hbox{$n$}}
    \qbezier(2.9,0.7)(4.9,3.2)(6.9,0.7)
    \qbezier(5.1,0.7)(7.1,3.2)(9.1,0.7)
  \end{picture} \ .
\end{center}

Nestings and crossings have been intensively studied and generalized in the literature, see e.g.~\cite{Ath1998,PPS2010,Pyl2009,RubeyStump}.
One important context in which they appear are two pure and flag simplicial complexes $\NN_n$ and $\NC_n$. 
Recall that a \defn{flag} simplicial complex is the complex of all vertex sets of cliques of some graph. 
$\NN_n$ is the flag complex having the arcs $1 \leq i < j \leq n$ as vertices and pairs of nonnesting arcs as edges, 
while $\NC_n$ is the flag complex with the same vertices and pairs of noncrossing arcs as edges. 

It is not hard to see that the maximal faces of $\NN_n$ are parametrized by \defn{Dyck paths} of length $2(n-2)$, while the maximal faces of $\NC_n$ are parametrized by \defn{triangulations} of a convex $n$-gon.
Thus both complexes have the same number of maximal faces, the $(n-2)$\textsuperscript{nd}
Catalan number $\frac{1}{n-1}\binom{2n-4}{n-2}$.
Moreover, it can be shown that their face vectors coincide and that both are 
balls of dimension $2n-4$. In addition, the complex $\NC_n$ is the join of an
$(n-1)$-dimensional simplex and an ubiquitous $(n-4)$-dimensional polytopal
sphere~$\NCred_n$, the (dual of the) \defn{associahedron}.

\subsection{The nonnesting complex}
The following generalization of the nonnesting complex is well known. 
Let $V_{k,n}$ denote the set of all vectors $(i_1,\ldots,i_k)$, $1 \leq i_1 < \cdots < i_k \leq n$, of length $k$ with entries in~$[n]$.

\begin{definition}
  \label{def:nonnesting}
  Two vectors $I = (i_1,\ldots,i_k)$ and $J = (j_1,\ldots,j_k)$ in $V_{k,n}$ are \defn{nonnesting} if for all indices $a<b$
  %with $i_\ell = j_\ell$ for $a < \ell < b$,
  the arcs $(i_a<i_b)$ and $(j_a<j_b)$ are nonnesting.
  The \defn{(multidimensional) nonnesting complex} $\NN_{k,n}$ is  the flag simplicial complex with vertices
  $V_{k,n}$ and with edges being the nonnesting pairs of
  vertices.
\end{definition}

By definition, we have $\NN_{2,n}=\NN_n$.
Equipped with the component-wise order,~$V_{k,n}$ becomes a distributive lattice.
Moreover, $I,J\in V_{k,n}$ are nonnesting if and only if $I\ge J$ or $J\ge I$ component-wise.
That is, $\NN_{k,n}$ is the order complex of the distributive lattice~$V_{k,n}$.

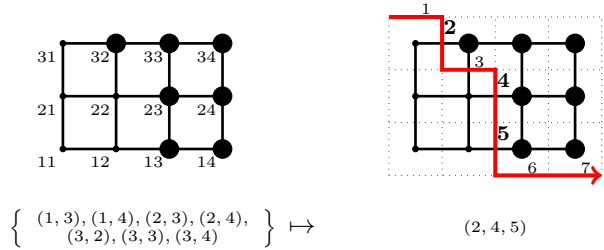
\begin{figure}[ht]
  \begin{tikzpicture}[scale=.7]
    \coordinate (11) at (0,0);
    \coordinate (12) at (1,0);
    \coordinate (13) at (2,0);
    \coordinate (14) at (3,0);
    \coordinate (21) at (0,1);
    \coordinate (22) at (1,1);
    \coordinate (23) at (2,1);
    \coordinate (24) at (3,1);
    \coordinate (31) at (0,2);
    \coordinate (32) at (1,2);
    \coordinate (33) at (2,2);
    \coordinate (34) at (3,2);

    \node at (-.3,-.25) {\tiny$11$};
    \node at (0.7,-.25) {\tiny$12$};
    \node at (1.7,-.25) {\tiny$13$};
    \node at (2.7,-.25) {\tiny$14$};

    \node at (-.3,0.75) {\tiny$21$};
    \node at (0.7,0.75) {\tiny$22$};
    \node at (1.7,0.75) {\tiny$23$};
    \node at (2.7,0.75) {\tiny$24$};

    \node at (-.3,1.75) {\tiny$31$};
    \node at (0.7,1.75) {\tiny$32$};
    \node at (1.7,1.75) {\tiny$33$};
    \node at (2.7,1.75) {\tiny$34$};

    \draw[line width=1pt,black] {(11)--(14)};
    \draw[line width=1pt,black] {(21)--(24)};
    \draw[line width=1pt,black] {(31)--(34)};

    \draw[line width=1pt,black] {(11)--(31)};
    \draw[line width=1pt,black] {(12)--(32)};
    \draw[line width=1pt,black] {(13)--(33)};
    \draw[line width=1pt,black] {(14)--(34)};

    \draw[fill=black] (11) circle (1.5pt);
    \draw[fill=black] (12) circle (1.5pt);
    \draw[fill=black] (13) circle (5pt);
    \draw[fill=black] (14) circle (5pt);
    \draw[fill=black] (21) circle (1.5pt);
    \draw[fill=black] (22) circle (1.5pt);
    \draw[fill=black] (23) circle (5pt);
    \draw[fill=black] (24) circle (5pt);
    \draw[fill=black] (31) circle (1.5pt);
    \draw[fill=black] (32) circle (5pt);
    \draw[fill=black] (33) circle (5pt);
    \draw[fill=black] (34) circle (5pt);

    \node at (1.5,-1.5) {\tiny $\left\{\begin{array}{cc} (1,3),(1,4),(2,3),(2,4), \\ (3,2),(3,3),(3,4) \end{array} \right\}$};
    \node at (4.5,-1.5) {$\mapsto$};
    \node at (1.5,-2) {};
    \end{tikzpicture}
  \qquad
  \begin{tikzpicture}[scale=.7]
    \coordinate (11) at (0,0);
    \coordinate (12) at (1,0);
    \coordinate (13) at (2,0);
    \coordinate (14) at (3,0);
    \coordinate (21) at (0,1);
    \coordinate (22) at (1,1);
    \coordinate (23) at (2,1);
    \coordinate (24) at (3,1);
    \coordinate (31) at (0,2);
    \coordinate (32) at (1,2);
    \coordinate (33) at (2,2);
    \coordinate (34) at (3,2);

    \coordinate (11dual) at (-.5,-.5);
    \coordinate (12dual) at (0.5,-.5);
    \coordinate (13dual) at (1.5,-.5);
    \coordinate (14dual) at (2.5,-.5);
    \coordinate (15dual) at (3.5,-.5);
    \coordinate (21dual) at (-.5,0.5);
    \coordinate (22dual) at (0.5,0.5);
    \coordinate (23dual) at (1.5,0.5);
    \coordinate (24dual) at (2.5,0.5);
    \coordinate (25dual) at (3.5,0.5);
    \coordinate (31dual) at (-.5,1.5);
    \coordinate (32dual) at (0.5,1.5);
    \coordinate (33dual) at (1.5,1.5);
    \coordinate (34dual) at (2.5,1.5);
    \coordinate (35dual) at (3.5,1.5);
    \coordinate (41dual) at (-.5,2.5);
    \coordinate (42dual) at (0.5,2.5);
    \coordinate (43dual) at (1.5,2.5);
    \coordinate (44dual) at (2.5,2.5);
    \coordinate (45dual) at (3.5,2.5);

    \draw[line width=1pt,black] {(11)--(14)};
    \draw[line width=1pt,black] {(21)--(24)};
    \draw[line width=1pt,black] {(31)--(34)};

    \draw[line width=1pt,black] {(11)--(31)};
    \draw[line width=1pt,black] {(12)--(32)};
    \draw[line width=1pt,black] {(13)--(33)};
    \draw[line width=1pt,black] {(14)--(34)};

    \draw[line width=0.25pt,dotted] {(11dual)--(15dual)};
    \draw[line width=0.25pt,dotted] {(21dual)--(25dual)};
    \draw[line width=0.25pt,dotted] {(31dual)--(35dual)};
    \draw[line width=0.25pt,dotted] {(41dual)--(45dual)};

    \draw[line width=0.25pt,dotted] {(11dual)--(41dual)};
    \draw[line width=0.25pt,dotted] {(12dual)--(42dual)};
    \draw[line width=0.25pt,dotted] {(13dual)--(43dual)};
    \draw[line width=0.25pt,dotted] {(14dual)--(44dual)};
    \draw[line width=0.25pt,dotted] {(15dual)--(45dual)};

    \draw[fill=black] (11) circle (1.5pt);
    \draw[fill=black] (12) circle (1.5pt);
    \draw[fill=black] (13) circle (5pt);
    \draw[fill=black] (14) circle (5pt);
    \draw[fill=black] (21) circle (1.5pt);
    \draw[fill=black] (22) circle (1.5pt);
    \draw[fill=black] (23) circle (5pt);
    \draw[fill=black] (24) circle (5pt);
    \draw[fill=black] (31) circle (1.5pt);
    \draw[fill=black] (32) circle (5pt);
    \draw[fill=black] (33) circle (5pt);
    \draw[fill=black] (34) circle (5pt);

    \draw[line width=1.5pt, red, ->] (41dual) -- (42dual) -- (32dual) --(33dual) -- (13dual) -- (15dual);

    \node at (0.2,2.65) {\tiny$1$};
    \node at (1.2,1.65) {\tiny$3$};
    \node at (2.2,-.35) {\tiny$6$};
    \node at (3.2,-.35) {\tiny$7$};

    \node at (0.65,2.3) {\footnotesize$\bf 2$};
    \node at (1.65,1.3) {\footnotesize$\bf 4$};
    \node at (1.65,0.3) {\footnotesize$\bf 5$};

    \node at (1.5,-1.5) {\tiny $(2,4,5)$};
    \node at (1.5,-2) {};
\end{tikzpicture}

\caption{
  An order filter in the poset~$P_{3,7}$, and the corresponding monotone path in the dual grid.
  Since the second, forth and fifth steps on the path are going south, the corresponding element of~$V_{3,7}$ is $(2,4,5)$.
}
\label{fig:poset}
\end{figure}

By Birkhoff's representation theorem for distributive lattices~\cite{Birkhoff}, there is a poset~$P$ such that the distributive lattice~$V_{k,n}$ is isomorphic to the lattice 
of (order) filters of~$P$.
Remember that an \defn{order filter} in~$P$ is a subset~$F \subset P$ satisfying
$
a\in F, a<_P b \Rightarrow b\in F.
$
It is easy to see that, indeed,~~$V_{k,n}$ is the lattice of filters in the product poset~$P_{k,n}$ of a $k$-chain and an $(n-k)$-chain.
A graphical way to set up this bijection between vectors in~$V_{k,n}$ and order filters in~$P_{k,n}$ is illustrated in Figure~\ref{fig:poset}.
To each filter in $P_{k,n}$ associate a monotone lattice path from $(0,0)$ to $(k,n-k)$ in a grid ``dual'' to the Hasse diagram of~$P_{k,n}$.
The path is defined by separating the elements in the filter from those not in the filter.
Such paths biject to~$V_{k,n}$ in the usual way by selecting the indices of steps in the direction of the first coordinate (the south direction in the picture).
As long as there is no ambiguity, we will thus consider elements of~$V_{k,n}$ as increasing $k$-tuples, as $k$-subsets, or as order filters in~$P_{k,n}$.

By a result of R.~Stanley~\cite[Sec, 5]{Sta1986}, $\NN_{k,n}$ is the standard triangulation of the \emph{order polytope} $\oO_{k,n} \subseteq [0,1]^{k\times(n-k)}$ of~$P_{k,n}$, where a vector~$I \in V_{k,n}$ is mapped to the characteristic vector~$\chi_I \in \mathbb{N}^{P_{k,n}}$ of the corresponding order filter.
We refer to Section~\ref{sec:orderpolytope} for basic facts about order polytopes and their triangulations.
It follows that~$\NN_{k,n}$ is a simplicial ball of dimension $k(n-k)$. 
Through this connection, its $h$-vector is linked to the Hilbert series of the coordinate ring of the Gra{\ss}mannian~$G_{k,n}$ of $k$-planes in $\CC^n$.
For details on this connection we refer to Section~\ref{sec:motivation}.

Linear extensions of $P_{k,n}$, i.e., maximal faces of $\NN_{k,n}$, are in bijection with standard tableaux of shape $k \times (n-k)$.
Here, a \defn{tableau} of shape $k \times (n-k)$ is a matrix in $\mathbb{N}^{k \times (n-k)}$ that is weakly increasing along rows from left to right and along columns from bottom to top.
Equivalently, it is a \emph{weakly} order preserving map $P_{k,n}\to \mathbb{N}$.
We denote the set of all tableaux of this shape by~$\Tab_{k,n}$.
A tableau is called \defn{standard} if it contains every integer $1$ through $k(n-k)$ exactly once.
An application of the hook length formula implies that maximal faces of $\NN_{k,n}$ are counted by the $(n-k,k)$\textsuperscript{th} \defn{multidimensional Catalan number}
\[
\Cat_{n-k,k} := \frac{0!\ 1!\ \cdots (k-1)!}{(n-1)!\ (n-2)!\ \cdots (n-k)!}\big(k(n-k)\big)! \ .
\]
These numbers were studied e.g. in~\cite{GorskaPenson,Sulanke}, see as well~\cite[Seq.~A060854]{OEIS}.
Denote the $h$-vector of $\NN_{k,n}$ by $(h_0^{(k,n)},\ldots,h_{n(n-k)}^{(k,n)})$. 
It follows from the connection of $\NN_{k,n}$ to the Hilbert series of the Gra{\ss}mannian, and it was also observed 
in~\cite{Sulanke} going back to P.~A.~MacMahon's study of plane partitions, that its entries are the \defn{multidimensional Narayana numbers}.
We refer to~\cite{Sulanke} for an explicit formula of these numbers, which can be combinatorially defined in terms of 
standard tableaux of shape $k\times (n-k)$ as follows.
Call an integer $a\in[k(n-k)-1]$ a \defn{peak} of a standard tableau $T$ if $a+1$ is placed in a lower row than~$a$. 
Then, $h^{(k,n)}_i$ equals the number of standard tableaux with exactly~$i$ peaks.
This combinatorial interpretation implies in particular that
\begin{align}
h_{i}^{(k,n)} =
  \begin{cases}
    1 &\text{ if } i = k(n-k)-n+1 \\
    0 &\text{ if } i > k(n-k)-n+1
  \end{cases} \label{eq:zeroesattheend}.
\end{align}

\begin{figure}
  \centering
  \begin{tikzpicture}[scale=1]
    \coordinate (1) at (-1,1);
    \coordinate (2) at (1,1);
    \coordinate (3) at (1,-1);
    \coordinate (4) at (-1,-1);
    \coordinate (5) at (-2,0);
    \coordinate (0) at (0,0);

    \fill[fill opacity=0.4] (1) -- (4) -- (5);
    \fill[fill opacity=0.4] (1) -- (4) -- (0);
    \fill[fill opacity=0.4] (3) -- (4) -- (0);
    \fill[fill opacity=0.4] (3) -- (2) -- (0);
    \fill[fill opacity=0.4] (1) -- (2) -- (0);

    \draw[line width=5pt,white] {(1)--(2)--(3)--(4)--(5)--(1)};
    \draw[line width=1pt,black] {(1)--(2)--(3)--(4)--(5)--(1)};
    \draw[line width=5pt,white] {(0)--(1)--(0)--(2)--(0)--(3)--(0)--(4)--(1)};
    \draw[line width=1pt,black] {(0)--(1)--(0)--(2)--(0)--(3)--(0)--(4)--(1)};

    \node[circle,inner sep=1pt,fill=white] at (1) {$14$};
    \node[circle,inner sep=1pt,fill=white] at (2) {$34$};    
    \node[circle,inner sep=1pt,fill=white] at (3) {$23$};    
    \node[circle,inner sep=1pt,fill=white] at (4) {$25$};    
    \node[circle,inner sep=1pt,fill=white] at (5) {$15$};    
    \node[circle,inner sep=1pt,fill=white] at (0) {$24$};    
  \end{tikzpicture}
  \qquad\qquad
  \begin{tikzpicture}[scale=1]
    \coordinate (1) at (-1,1);
    \coordinate (2) at (1,1);
    \coordinate (3) at (1,-1);
    \coordinate (4) at (-1,-1);
    \coordinate (5) at (-2,0);
    \coordinate (0) at (0,0);
    
    \fill[fill opacity=0.4] (1) -- (5) -- (0);
    \fill[fill opacity=0.4] (4) -- (5) -- (0);
    \fill[fill opacity=0.4] (3) -- (4) -- (0);
    \fill[fill opacity=0.4] (3) -- (2) -- (0);
    \fill[fill opacity=0.4] (1) -- (2) -- (0);

    \draw[line width=5pt,white] {(1)--(2)--(3)--(4)--(5)--(1)};
    \draw[line width=1pt,black] {(1)--(2)--(3)--(4)--(5)--(1)};
    \draw[line width=5pt,white] {(0)--(1)--(0)--(2)--(0)--(3)--(0)--(4)--(0)--(5)};
    \draw[line width=1pt,black] {(0)--(1)--(0)--(2)--(0)--(3)--(0)--(4)--(0)--(5)};

    \node[circle,inner sep=1pt,fill=white] at (1) {$14$};
    \node[circle,inner sep=1pt,fill=white] at (2) {$13$};    
    \node[circle,inner sep=1pt,fill=white] at (3) {$35$};    
    \node[circle,inner sep=1pt,fill=white] at (4) {$25$};    
    \node[circle,inner sep=1pt,fill=white] at (5) {$24$};    
    \node[circle,inner sep=3pt,fill=white] at (0) {$\bullet$};    
  \end{tikzpicture}
  \caption{Parts of the nonnesting complex~$\NN_{2,5}$ of the noncrossing complex~$\NC_{2,5}$.}
  \label{fig:52}
\end{figure}
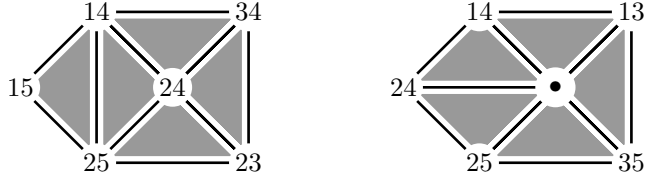

\begin{example}
  For $n=5$ and $k=2$, the vertices of the nonnesting complex $\NN_{2,5}$ are given by $V_{2,5} = \{ 12,13,14,15,23,24,25,34,35,45\}$, and the~$5$ maximal faces are given by the join of the simplex spanned by $\{12,13,35,45\}$ and the~$5$ faces
  $$\big\{ \{14,15,25\}, \{14,24,25\}, \{23,24,25\}, \{23,24,34\}, \{14,24,34\} \big\}.$$
  This subcomplex is shown in Figure~\ref{fig:52} on the left.
\end{example}

\subsection{The noncrossing complex}

The reformulation of the nonnesting complex as the standard triangulation of $\oO_{k,n}$
raises the question whether an analogous construction of a multidimensional noncrossing complex has interesting properties as well.
The main object of study in this paper is  the following slight modification of Definition~\ref{def:nonnesting}, introduced in~\cite{PPS2010}. 

\begin{definition} \label{def:noncrossing}
  Two vectors $I = (i_1,\ldots,i_k)$ and $J = (j_1,\ldots,j_k)$ in $V_{k,n}$ are \defn{noncrossing} if for all 
  indices $a<b$ with $i_\ell = j_\ell$ for $a < \ell < b$, the arcs $(i_a<i_b)$ and $(j_a<j_b)$ do not cross.
  The \defn{(multidimensional) noncrossing complex} $\NC_{k,n}$ is  the flag simplicial complex with vertices 
  $V_{k,n}$ and with edges being the noncrossing pairs of vertices.
\end{definition}

\begin{remark}
The definition in~\cite{PPS2010} allows for the vectors $I$ and $J$ to have different lengths, and restricts to our definition in the case of equal lengths.
We discuss this further in Section~\ref{sec:PPS}.
\end{remark}

\begin{example}
  For $n=5$ and $k=2$, the vertices of the noncrossing complex $\NC_{2,5}$ are again given by $V_{2,5} = \{ 12,13,14,15,23,24,25,34,35,45\}$, and the~$5$ maximal faces are given by the join of the simplex spanned by $\{12,23,34,45,15\}$ and the~$5$ faces
  $$\big\{ \{14,24\}, \{24,25\}, \{13,25\}, \{13,35\}, \{14,35\} \big\}.$$
  The noncrossing complex is shown in Figure~\ref{fig:52} on the right, where the circle indicates the simplex spanned by $\{12,23,34,45,15\}$.
\end{example}

\begin{remark}
  The reader may wonder why in the noncrossing world one requires the noncrossing property only for \emph{some} pairs of coordinates 
  $a<b$, while in the nonnesting world the nonnesting property is required for \emph{all} pairs.
  One answer is that the direct noncrossing analogue of Definition~\ref{def:nonnesting} does not even yield a pure complex.
  But another answer is that it would not make a difference in Definition~\ref{def:nonnesting} to require the condition 
  only for pairs with $i_\ell = j_\ell$ for $a < \ell < b$. All other pairs would automatically be nonnesting, thanks to the 
  following transitivity of nonnestingness: 
  let $a<b<c$ and suppose that the arcs $(i_a<i_b)$ and $(j_a<j_b)$ are nonnesting, and the arcs $(i_b<i_c)$ and $(j_b<j_c)$ are nonnesting as well. 
  Then the arcs $(i_a<i_c)$ and $(j_a<j_c)$ are also nonnesting.
\end{remark}

The main properties of $\NC_{k,n}$  are summarized in the following statement.

\begin{theorem}
  \label{th:main}
  The noncrossing complex $\NC_{k,n}$ is a flag, regular, unimodular and Gorenstein triangulation of the order polytope $\oO_{k,n}$.
  In particular, $\NC_{k,n}$ and $\NN_{k,n}$ have the same $f$- and $h$-vectors.
\end{theorem}

This statement, of which we give an independent proof, is already contained in~\cite{PPS2010} in the following way. There, the order polytope  $\oO_{k,n}$ appears as the Gelfand-Tsetlin polytope of a particular shape (a rectangle). Theorem 8.1 from~\cite{PPS2010} says that $\NC_{k,n}$ is a regular triangulation of it, and Theorem 8.7 that it is Gorenstein. Unimodularity is mentioned in the proof of Corollary 8.2.

The claim that $\NC_{k,n}$ is ``in some respects nicer'' than $\NN_{k,n}$ is justified by the word ``Gorenstein'' in the statement, which fails for $\NN_{k,n}$. Recall that a \emph{Gorenstein triangulation} of a polytope is one that decomposes as the join of a simplex and a sphere (see Section \ref{sec:orderpolytope} for details on such triangulations). This property is related to the last of the following list
of purely combinatorial properties of $\NC_{k,n}$, which 
generalize to higher~$k$ well known properties of the dual associahedron.

\begin{proposition} \label{prop:intro}
     The complex~$\NC_{k,n}$ has the following properties.
    \begin{enumerate}[(i)] 

        \item The map $a \mapsto n+1-a$ induces an automorphism on $\NC_{k,n}$. \label{eq:reflectionsymmetry}

        \item The map $I \mapsto [n] \setminus I$ induces an isomorphism $\NC_{k,n}\ \tilde\longrightarrow\ \NC_{n-k,n}$.  
          \label{eq:complementarity}

        \item $I,J \in V_{k,n}$ are noncrossing if and only if they are noncrossing when restricting to the symmetric difference 
              $I \triangle J = ( I \cup J ) \setminus ( I \cap J)$. \label{eq:symmetricdifference}

        \item For $b \in [n]$ The restriction of $\NC_{k,n}$ to vertices with $b \in I$ yields $\NC_{k-1,n-1}$. 
              The restriction of $\NC_{k,n}$ to vertices with $b \not\in I$  yields $\NC_{k,n-1}$.
              \label{eq:restriction}

        \item The $n$ vertices in $V_{k,n}$ obtained by cyclic rotations of the vertex $(1,2,\ldots,k) \in V_{k,n}$ do not 
              cross any other vertex in $V_{k,n}$ and
              hence are contained in every maximal face of $\NN_{k,n}$.  
              \label{eq:join}
    \end{enumerate}
\end{proposition}

Parts~\eqref{eq:complementarity}  and~\eqref{eq:join} are mentioned in~\cite[Remark 2.7]{PPS2010} and~\cite[Lemma 8.6]{PPS2010}, where the $n$  vertices in~\eqref{eq:join} are called \emph{solid elements}.

\begin{proof}
  Property~\eqref{eq:reflectionsymmetry} is clear from the definition.
  Property~\eqref{eq:complementarity} can be derived from the observation that a crossing between two vertices $I,J \in V_{k,n}$ induces a 
  crossing between $[n]\setminus I$ and $[n]\setminus J$ in $V_{n-k,n}$. Applying this argument twice, we obtain that~$I$ and~$J$ are noncrossing if and 
  only if $[n] \setminus I$ and $[n]\setminus J$ are noncrossing.
  To obtain Property~\eqref{eq:symmetricdifference}, observe that it is clear from the definition that one can always restrict the attention to the 
  situation where the set $[n]$ is replaced by $I \cup J$.
  It then follows with Property~\eqref{eq:complementarity} that one can as well remove $I \cap J$.
  Property~\eqref{eq:restriction} is a consequence of Property~\eqref{eq:symmetricdifference}.
  For Property~\eqref{eq:join},
  let $J = (j_1,\ldots,j_k)$ and let $I = (i_1,\ldots,i_k) = (c,\ldots,c+k)$ for some $1 \leq c \leq n$, with elements considered 
  modulo $n$.
  Since the entries in~$I$ are as `close together' as possible, it is not possible to have 
  two crossing arcs $(i_a < i_b)$ and $(j_a < j_b)$ such that $i_\ell = j_\ell$ for $a < \ell < b$.
\end{proof}

Observe that Properties \eqref{eq:complementarity} and~\eqref{eq:restriction} are natural when considering the relation between $\oO_{k,n}$ and the Gra{\ss}mannian 
(see Section \ref{sec:motivation}) as they reflect the isomorphism $G_{k,n}\cong G_{n-k,n}$ and the embeddings $G_{k-1,n-1}\hookrightarrow G_{k,n}$ and $G_{k,n-1}\hookrightarrow G_{k,n}$.

Properties~\eqref{eq:reflectionsymmetry},~\eqref{eq:complementarity},~\eqref{eq:symmetricdifference}, and~\eqref{eq:restriction} also hold for the 
nonnesting complex $\NN_{k,n}$, but Property~\eqref{eq:join} fails.
This property implies
that $\NC_{k,n}$ is the join of an $(n-1)$-dimensional simplex and a complex~$\NCred_{k,n}$ of dimension $k(n-k)-n$ which has the same $h$-vector as $\NC_{k,n}$.
Note that, for $k=2$,~$\NCred_{k,n}$ reduces to the (dual) associahedron~$\NCred_n$.
Thus the following corollary (which implies that $\NCred_{k,n}$ is a Gorenstein triangulation of $\oO_{k,n}$)
together with the discussion in Section \ref{sec:motivation} justifies that we call
the dual complex of~$\NCred_{k,n}$ the \defn{Gra{\ss}mann associahedron}.

\begin{corollary} \label{cor:reducedsphere}
  %The dual Gra{\ss}mann  associahedron 
  $\NCred_{k,n}$ is a flag polytopal sphere of dimension $k(n-k)-n$.
  Moreover, Properties~\eqref{eq:reflectionsymmetry},~\eqref{eq:complementarity}, and~\eqref{eq:restriction} in 
  Proposition~\ref{prop:intro} also hold for $\NCred_{k,n}$.
\end{corollary}

Observe that although~\cite{PPS2010} show that $\NCred_{k,n}$ is a sphere (Lemma 8.6 and Theorem 8.7), polytopality of this sphere is a special case of  their Conjecture 8.10. The following arguments, applied to their Theorems 8.1 and 8.7 instead of our Theorem~\ref{th:main} and Proposition~\ref{prop:intro}\eqref{eq:join}, prove that conjecture in full generality.

\begin{proof}
  $\NCred_{k,n}$ is clearly a sphere or a ball of dimension $k(n-k)-n$, since it is the link of an $(n-1)$-simplex in the 
  triangulation $\NC_{k,n}$ of the $k(n-k)$-dimensional polytope $\oO_{k,n}$. 
  Since $h^{k,n}_{k(n-k)-n+1}=1$ it must be a sphere.
  Polytopality follows from regularity of $\NC_{k,n}$ as a triangulation of $\oO_{k,n}$ and flagness is preserved under 
  taking links. The three operations in the proposition are also preserved 
  since they leave the set of vertices described in~\eqref{eq:join} invariant.
\end{proof}

In particular, the Gra{\ss}mann associahedron can be realized as a simple polytope of dimension $k(n-k)-n+1 = (k-1)(n-k-1)$.

\begin{remark}
\label{rem:cyclicsymmetry}
  Proposition~\ref{prop:intro}\eqref{eq:reflectionsymmetry} says that $\NC_{k,n}$ possesses
  the reflection symmetry present in the associahedron $\NCred_{n}$. 
  Of course, another symmetry of $\NCred_{n}$ comes from the cyclic rotation $i \mapsto i+1$ 
  (considered as remainders $1, \ldots, n$ modulo $n$). That symmetry \emph{does not} carry over to $\NC_{k,n}$ for $k\ge 3$.
  In fact, such a cyclic symmetry \emph{cannot} carry over to the general situation since no flag complex 
  on the set of vertices $V_{3,6}$
  that has the $h$-vector of $\NN_{3,6}$ can be invariant under cyclic rotation.
  To see this, observe that such a complex would have~$155$ edges and~$35$ nonedges. In particular, there should be in 
  $\binom{V_{3,6}}{2}$ at least two rotational orbits of size \emph{not} a multiple of three 
  (one orbit of edges and one orbit of nonedges).
  But an orbit whose size is not divisible by three must have all its elements fixed by the order three rotation $i\mapsto i+2$, and the only 
  element of $\binom{V_{3,6}}{2}$ fixed by this rotation turns out to be
  $\{135, 246\}$.
\end{remark}

\subsection{Motivation: the Hilbert series of the Pl\"ucker embedding}
\label{sec:motivation}

Besides its well behaved combinatorial properties, our main motivation for studying the noncrossing complex comes from the 
connection between the order polytope $\oO_{k,n}$, initial ideals of the ideal of Pl\"ucker relations, and Hilbert series of Gra\ss mannians.
We refer to~\cite{Sturmfels1996},~\cite{GonciuleaLakshmibai} and~\cite{Hibi} for more details of this connection.

Let $G_{k,n}$ denote the \defn{Gra\ss mannian} of $k$-dimensional linear subspaces in $\CC^{n}$, and let~$L_{k,n}$ be the defining ideal of $G_{k,n}$  in its Pl\"ucker embedding.
Thus,~$L_{k,n}$ is the homogeneous ideal in the polynomial ring
$$T_{k,n} = \CC[x_{i_1,\ldots, i_{k}}~:~1 \leq i_1 < \cdots < i_{k} \leq n]$$
generated by the Pl\"ucker relations.
It follows from work of B.~Sturmfels~\cite{Sturmfels1996} that the Stanley-Reisner ideals of all regular unimodular triangulations of $\oO_{k,n}$ are squarefree initial ideals of $L_{k,n}$.
Indeed, let~$M_{k,n}$ be the ideal in the polynomial ring with variables $\{x_F: F$ filter of $P_{k,n}\}$ generated by the binomials $x_{E}x_{F} - x_{E \cap F}x_{E \cup F}$ for all choices of order ideals~$E$ and~$F$ in~$P_{k,n}$.
The ideal~$M_{k,n}$ is known as the Hibi ideal of the poset~$P_{k,n}$, or the Ehrhart ideal of the polytope $\oO_{k,n}$.
In~\cite[Prop. 11.10, Cor. 8.9]{Sturmfels1996} it is shown that~$M_{k,n}$ appears as an initial ideal of~$L_{k,n}$.
In turn, it follows from~\cite[Ch. 8]{Sturmfels1996} that there is a one to one correspondence between regular unimodular triangulations of $\oO_{k,n}$ and squarefree monomial initial ideals of~$M_{k,n}$.
This correspondence sends a particular regular unimodular triangulation to its Stanley-Reisner ideal.

\begin{itemize}
\item The regular, unimodular, flag triangulation $\NN_{k,n}$ of $\oO_{k,n}$ leads to a squarefree monomial initial ideal of $M_{k,n}$ 
  studied by T.~Hibi~\cite{Hibi}.
\item The regular, unimodular, flag triangulation $\NC_{k,n}$ provides a new initial ideal with particularly nice properties and leads to 
  new insight in the Hilbert series of the coordinate ring $A_{k,n} = T_{k,n} \big/ L_{k,n}$.
\end{itemize}

From the relation between initial ideals and unimodular triangulations stated above it follows that this Hilbert series is given by 
$$
  \Hilb_{A_{k,n}}(t) = H(t) \big/(1-t)^{k(n-k)+1},
$$
where 
$
  H(t) = h_{0}^{(k,n)} + h_{1}^{(k,n)} t + \cdots + h_{k(n-k)}^{(k,n)} t^{k(n-k)}
$
is the $h$-polynomial of any regular unimodular triangulation corresponding of $\oO_{k,n}$. In particular, its coefficients are the multidimensional Narayana numbers.

\medskip

In the following, let~$\Delta$ be a simplicial complex whose Stanley-Reisner ideal $I_\Delta$ appears as an initial ideal of~$L_{k,n}$.
Then the following properties are desirable for~$\Delta$:
\begin{itemize}

  \item It follows from~\eqref{eq:zeroesattheend} that there are at most~$n$ variables that do not appear in the set of generators of~$I_\Delta$. 
      Equivalently, if~$\Delta$ decomposes into $\Delta = 2^V * \Delta'$ where $2^V$ is the full simplex spanned by~$V$, then $\# V \leq n$.
      Thus the `most factorizable' complex~$\Delta$ should be a join over a simplex spanned by~$n$ vertices.

  \item The fact that $A_{k,n}$ is Gorenstein should be reflected in $\Delta$. Thus we desire that $\Delta$ is the join of a simplex with a 
      triangulation of a (homology) sphere of the appropriate dimension or, even better, 
      the boundary complex of a simplicial polytope, which would then deserve the name \emph{(dual) Gra\ss mann associahedron}.

  \item Since $A_{k,n}$ has a quadratic Gr\"obner basis, it is Koszul. Hence, one could hope that $I_{\Delta}$ is generated by quadratic monomials, 
      or, equivalently, that~$\Delta$ is flag.

  \item One could hope that~$\Delta$ reflects the duality between $G_{k,n}$ and $G_{n-k,n}$,
      as well as the embeddings $G_{k-1,n-1}\hookrightarrow G_{k,n}$ and $G_{k,n-1}\hookrightarrow G_{k,n}$.

\end{itemize}

Theorem~\ref{th:main}, Proposition~\ref{prop:intro}, and Corollary~\ref{cor:reducedsphere} say that the noncrossing complex $\Delta~=~\NC_{k,n}$ fulfills all these properties. 

\medskip

Note that in this algebraic framework the result of Theorem~\ref{thm:uniquemultiset} translates into a statement about standard monomials 
in $T_{k,n}/M_{k,n}$ (the Hibi ring of~$P_{k,n}$, or the Ehrhart ring of~$\oO_{k,n}$).
Let $\preceq$ be a term order for $T_{k,n}$ and suppose the corresponding initial ideal of $M_{k,n}$ is squarefree (and monomial). Equivalently, by Sturmfels' results, 
the initial ideal comes from a unimodular triangulation of~$\oO_{k,n}$.
Tableaux of shape $k \times (n-k)$ are nothing but the integer points in dilations of $\oO_{k,n}$ (see Lemma~\ref{lemma:tableaux-as-lattice-points}) and hence they index standard monomials with respect to $\preceq$.
More precisely, tableaux in the $r$\textsuperscript{th} dilation of~$\oO_{k,n}$ correspond to standard monomials of degree $r$.%
\footnote{Observe that there is a certain ambiguity here. Since~$\oO_{k,n}$ contains the origin, a point in its $r$\textsuperscript{th} dilation lies also in the $s$\textsuperscript{th} dilation, 
for any $s\ge r$. This reflects the fact that multiplying by the generator corresponding to the vertex $(n-k+1,\ldots, n)$ of $\oO_{k,n}$ has no effect in the tableaux and is the reason 
why $V_{k,n}^*= V_{k,n}\setminus\{(n-k+1,\ldots, n)\}$ appears in Theorem~\ref{thm:uniquemultiset} instead of $V_{k,n}$.}
By assumption the initial ideal of $M_{k,n}$ with respect to $\preceq$ is squarefree. Hence there is a simplicial complex $\Delta$ such that the initial ideal of $M_{k,n}$ is the Stanley-Reisner ideal of $\Delta$ 
and consequently standard monomials for $\preceq$ are the monomials whose support is a face in $\Delta$. 
Since each standard monomial is in a unique way a product of the variables, which are in bijection to the vertices of~$\oO_{k,n}$ or to $V_{k,n}$, 
standard monomials of degree $r$ are identified with multisets of $r$ elements from $V_{k,n}$ whose support lies $\Delta$. Thus combinatorially we get an identification of tableaux and multisets. 
In this perspective Theorem~\ref{thm:uniquemultiset} provides this identification for $\Delta = \NN_{k,n}$ and $\Delta = \NC_{k,n}$ and the corresponding term orders.

Since $M_{k,n}$ is an initial ideal of $L_{k,n}$, standard monomials for $\preceq$ are also
standard monomials of a Gr\"obner basis of $L_{k,n}$, which links our results to standard monomial theory (see~\cite{LakshmibaiRaghavan}) for Schubert 
varieties. Among other aspects, this theory deals with straightening rules for products of standard monomials in the coordinate rings. 
For $\NN_{k,n}$ we are in the classical standard monomial theory of the Gra\ss mann variety. 
It would be interesting to develop straightening laws for our new set of standard monomials corresponding to $\NC_{k,n}$.

\subsection{Relation to previous work}

\subsubsection{Petersen-Pylyavskyy-Speyer's noncrossing complex} 
\label{sec:PPS}
  Some of the main results from this paper were previously proved by Petersen, Pylyavskyy and Speyer in~\cite{PPS2010} in a more general context. Let $V_n$ denote the set of all subsets of $[n]$, which can be thought of as the disjoint union of $V_{k,n}$ for all $k\in [0,n]$. Petersen et al.~then define a noncrossing relation among elements of $V_n$ and consider, for each subset $L\subset [n]$, the flag complex 
  $\mm^{(nc)}_L$
  of non-crossing vectors whose length belongs to $L$. In particular, $\mm^{(nc)}_{\{k\}}$ is exactly equal to the noncrossing complex $\NCred_{k,n}$ considered in this paper.
  
  The main result of~\cite{PPS2010}, as was already mentioned, is the generalization of Theorem~\ref{th:main} to arbitrary $L$, by changing the order polytope $\oO_{k,n}$ to the more general Gelfand-Tsetlin polytopes of shape $L$. 
  
  Our methods of proof, however, are different, and, as we think, of independent interest. 
  Our main innovation is the explicit relation of facets of the non-crossing complex and tableaux that we set up in Section~\ref{sec:combinatorics}.
  This has several algorithmic applications, analogous to the \emph{driving rules} in~\cite{PPS2010}:
  \begin{enumerate}
  \item Theorem~\ref{thm:uniquemultiset} (or, rather, its proof) contains a fast algorithm for \emph{point location} in $\NC_{k,n}$: 
  given a point $T$ in the order polytope $\oO_{k,n}$, the algorithm outputs the minimal face of $\NC_{k,n}$ containing $T$ (the \emph{carrier} of $T$).
  \item The ``pushing of bars'' procedure described in ?? gives an efficient algorithm to construct the non-crossing complex $\NC_{k,n}$ or the star of any individual face in it. Efficient here means ``polynomial in the output size''.
  \end{enumerate}
   As a by-product of the second item above, we have a natural way to give directions to the  edges in the dual graph of the non-crossing complex. In Section~\ref{sec:GrassmannTamari} we show that these directions make the graph acyclic which, in particular, allows us to define a poset structure on the facets of $\NC_{k,n}$. We call this the Gra\ss{}mann Tamari poset since it generalizes the classical Tamari poset, and conjecture it to be a lattice.

  \subsubsection{Pylyavskyy's noncrossing tableaux} 
  In~\cite{Pyl2009}, P.~Pylyavskyy introduces and studies what he calls \emph{noncrossing tableaux}, showing that they are equinumerous with standard tableaux, hence with facets of $\NC_{k,n}$.
  The construction therein does not seem to be directly linked to the multidimensional noncrossing complex, as already noted in~\cite{PPS2010}.
  For example, Pylyavskyy's noncrossing tableau are not in general monotone along columns, while the tableaux that we biject to maximal faces of $\NC_{k,n}$ in 
  Section~\ref{sec:decomposition} are strictly monotone along rows and columns.

\subsubsection{Weakly separable sets} 
  Closely related to our complex is the notion of \emph{weakly separable} subsets of~$[n]$, introduced by B.~Leclerc and A.~Zelevinsky in~\cite{LZ1998} in the context of quasi-commuting families of quantum {P}l\"ucker coordinates.
  Restricted to subsets of the same size~$k$, which is the case of interest to us, the definition is that two $k$-subsets~$X,Y\subset [n]$ are \defn{weakly separable} if, when considered as subsets of vertices in an $n$-gon,  the convex hulls of $X\setminus Y$ and $Y\setminus X$ are disjoint. 
  The flag complex~$\WS_{k,n}$ of weakly separable $k$-subsets of~$[n]$ was studied by J.~S.~Scott in~\cite{Sco2005,Sco2006},
  who conjectured that~$\WS_{k,n}$ is pure of dimension $k(n-k)$, and that it is strongly connected (that is, its dual graph is connected).
  Both conjectures were shown to hold by S.~Oh, A.~Postnikov and D.~Speyer~\cite{OPS2011}, for the first see also V.~I.~Danilov, A.~V.~Karzanov, and G.~A.~Koshevoy~\cite[Prop.~5.9]{DKK2010}.

  It is not hard to see that~$\WS_{k,n}$ is a subcomplex of $\NC_{k,n}$ and it is trivial to observe that $\WS_{k,n}$ is invariant under cyclic (or, more strongly, dihedral) symmetry.
  As we will see in  Section~\ref{sec:separated}, it turns out that the weak separation graph is the intersection of all cyclic shifts of the noncrossingness graph.
  Since flagness is preserved by intersection, the same happens for the complexes.
  We expect our approach to the noncrossing complex~$\NC_{k,n}$ to also shed further light on the weak separability complex~$\WS_{k,n}$.
% \paco{Mention the conjecture about the topology of $\WS_{k,n}$, and cite~\cite{HessHirsch2013}}
% \volkmar{done}
  In particular, we hope to better understand the intriguing conjecture about the topology of $\WS_{k,n}$ and its generalizations that can be found a 
  preliminary version~\cite{HessHirsch2011} of~\cite{HessHirsch2013}.

\subsubsection{Triangulations of order polytopes}
  \label{sec:orderpolytope}
  
  Essential for most of our main conclusions is the fact that $\NN_{k,n}$ and $\NC_{k,n}$ are triangulations of an order polytope. 
  We recall some basic facts about order polytopes and then give relations to known results about triangulations of order
  polytopes or more general integer polytopes.
  Let~$P$ be a finite poset. The \defn{order polytope} of $P$, introduced by R.~Stanley~\cite{Sta1986}, is given by
  \begin{align*}
    \oO(P) &:= \big\{(x_a)_{a\in P} \in [0,1]^P : x_a\le x_b\ \mbox{~for all~} a<_{P} b \big\}.\\[3pt]
           &\phantom{:}= \conv\big\{ \chi_F ~:~ F \text{ order filter of } P \big\},
  \end{align*}
  where~$\chi_F \in \mathbb{N}^{P}$ is the characteristic vector of the order filter~$F$ of~$P$.
  The order polytope is a $0/1$-polytope of dimension~$|P|$.
  It has a somehow canonical triangulation $\Delta(P)$, see again~\cite[Sec. 5]{Sta1986}, that we call the 
  \defn{standard triangulation}. It is also sometimes called the \defn{staircase triangulation} of $\oO(P)$. It can be described in the 
  following equivalent ways.
  \begin{itemize}
    \item Each of the $|P|!$ monotone paths from $(0,\ldots,0)$ to $(1,\ldots,1)$ in the unit cube $[0,1]^P$ defines a full-dimensional 
     simplex. These simplices triangulate the cube, and the subset of them whose vertices lie in $\oO(P)$ triangulate $\oO(P)$.

   \item Each such monotone path is the Hasse diagram of a \defn{linear extension} of $P$. Thus, $\Delta(P)$ is the subdivision of $\oO(P)$ 
     into the order polytopes of the linear extensions of $P$.

   \item Under the correspondence between vertices of $\oO(P)$ and filters of $P$, linear extensions correspond to maximal containment chains of filters. 
     That is, $\Delta(P)$ is the \emph{order complex} of the lattice of filters of $P$, where the order complex 
     of a poset is the flag simplicial complex obtained from the comparability graph of $P$.

   \item Last but not least, the complex $\Delta(P)$ can be realized as a the partition of $\oO(P)$ obtained by slicing it by all the 
     hyperplanes of the form $\{x_a=x_b\}$, $a,b\in P$. Of course, these hyperplanes only slice $\oO(P)$ if $a $ and $b$ were incomparable, 
     in which case the two sides of the hyperplane correspond to the two possible relative orders of $a$ and $b$ in a linear extension of $P$.
  \end{itemize}

  The third (and also the fourth) description of $\Delta(P)$ shows that it is a flag complex. Any of the first three shows that it is 
  unimodular (all simplices have euclidean volume $1/|P|)$, the minimal possible volume of a full-dimensional lattice simplex in $\RR^P$). 
  Finally, the last description implies it to be regular. 

  In ~\cite{RW2005} V.~Reiner and V.~Welker construct, for every graded poset~$P$ of rank~$n$, a regular unimodular triangulation $\Gamma(P)$ of $\oO(P)$
  that decomposes as $2^{W} * \widetilde{\Gamma}(P)$ for a simplex $2^{W}$ with $n$ vertices and a polytopal sphere $\widetilde{\Gamma}(P)$. Since this is
  a Gorenstein simplicial complex we call it a Gorenstein triangulation.
  The existence of Gorenstein triangulations was later verified by C.~A.~Athanasiadis~\cite{Ath2005} for a larger geometrically
  defined class of polytopes and then by W.~Bruns and T.~R\"omer~\cite{BrunsRomer} for the even larger class of all 
  \emph{Gorenstein polytopes} admitting a regular unimodular triangulation. A Gorenstein polytope, here, is one whose unimodular triangulations have a symmetric $h$-vector, and it was first shown in~\cite{Hibi}
  that an order polytope $\oO(P)$ is Gorenstein if and only if $P$ is graded. The survey article by~\cite{ConcaHostenThomas} puts the
  existence of Gorenstein triangulations in an algebraic perspective. 

  In particular, any of~\cite{RW2005},~\cite{Ath2005},~\cite{BrunsRomer} shows the existence of a regular, unimodular, Gorenstein triangulation of $\oO(P_{k,n})$. 
  This implies that the multidimensional Narayana numbers are the face numbers of a 
  simplicial polytope, and thus satisfy all conditions of the $g$-theorem.
  It can be checked that the triangulation of~\cite{RW2005} is not flag for $P_{k,n}$ and neither the results from~\cite{Ath2005} nor from~\cite{BrunsRomer} can
  guarantee flagness of the triangulation. 
  The construction in~\cite{PPS2010} and the present paper does. In particular, the 
  multidimensional Narayana numbers satisfy all inequalities valid for $h$-vectors of flag simplicial polytopes. This includes the positivity of 
  the $\gamma$-vector and as a special case the Charney-Davis inequalities. Note, that the latter implication are know to hold by~\cite{Branden2004}, where 
  they are shown to hold for all triangulations of order polytopes of graded posets.
  Also, it was pointed out by C.~A.~Athanasiadis to the authors of~\cite{RW2005} that the Gorenstein triangulation of~$\oO_{2,n}$ obtained from their construction is 
  not isomorphic to a dual associahedron
  To our best knowledge, neither the construction from~\cite{Ath2005} nor from~\cite{BrunsRomer} can be used to obtain such a triangulation.
  Thus, the $\NC_{k,n}$ from~\cite{PPS2010} studies in this paper appears to be more suited for a combinatorial analysis, and more closely related to Gra\ss mannians, 
  than these previous constructions.

  \section{Combinatorics of the noncrossing complex}
\label{sec:combinatorics}

This section is devoted to the combinatorics of the noncrossing complex and its close relationship with the combinatorics of the nonnesting complex.
We study nonnesting and noncrossing decompositions of tableaux, which will later be the main tool in Section~\ref{sec:geometry} to understand the geometry of these complexes.
We then deduce several further combinatorial properties of these complexes directly from the tableau decompositions.
In the final part of this section, we define and study the Gra{\ss}mann-Tamari order on maximal faces of the noncrossing complex.

\subsection{The nonnesting and noncrossing decompositions of a tableau}
\label{sec:decomposition}

As defined in the introduction, a \emph{tableau} of shape $k \times (n-k)$ is a matrix $T\in \mathbb{N}^{k\times(n-k)}$
that is weakly increasing along rows from left to right and along columns from bottom to top.
Recall also that we denote the set of all tableaux of shape $k \times (n-k)$ by~$\Tab_{k,n}$.
We still consider rows as labeled from top to bottom (i.e., the top row is the first row).
This unusual choice makes tableaux of zeros and ones correspond to vectors in $V_{k,n}$.
For each weakly increasing vector $(b_1,\ldots,b_k)\in [0,n-k]^k$, the tableau having as its $a$\textsuperscript{th} row $b_a$ zeroes followed by $n-k-b_a$ ones corresponds to the increasing vector $I = (b_1+1,\ldots,b_k+k)\in V_{k,n}$ via the bijection sending $I \in V_{k,n}$ to its characteristic vector~$\chi_I \in \mathbb{N}^{P_{k,n}}$.

\medskip

We now show how to go from a \emph{multiset} of vectors in $V_{k,n}$ to a tableau, and vice versa.
The geometric interpretation of tableaux as integer points in the cone spanned by the order polytope $\oO_{k,n}$ as discussed in Section~\ref{sec:geometry} (see in particular Lemma~\ref{lemma:tableaux-as-lattice-points}) will then lead to a proof that the nonnesting and the noncrossing complexes triangulate~$\oO_{k,n}$.

\medskip

Let~$L$ be a multiset of~$\ell$ vectors $(i_{1j},\ldots,i_{kj}) \in V_{k,n}$ ($1 \leq j \leq \ell$).
The \defn{summing tableau}~$T = (t_{ab})$ of the multiset~$L$ is the $k \times (n-k)$-matrix
\[
  t_{ab} = \# \big\{ j \in [\ell]\ :\ i_{aj} \leq b+a-1 \big\}.
\]

Note that if $L=\{I\}$ is a single vector, then the summing tableau has only zeroes and ones and coincides
with $\chi_I$, the characteristic vector of a filter in $P_{k,n}$.
The following lemma can be seen as a motivation for the definition of the summing tableau, and is a direct consequence thereof.

\begin{lemma}
\label{lem:summingtableau}
  The summing tableau~$T$ of a multiset~$L$ of vectors in $V_{k,n}$ equals
  $$
  T = \sum_{I \in L} \chi_I \quad\in \mathbb{N}^{P_{k,n}}.
  $$
  In particular,~$T$ is a weakly order preserving map from $P_{k,n}$ to the nonnegative integers and thus a tableau in~$\Tab_{k,n}$.
\end{lemma}

It follows directly from this description of the summing tableau that the two maps described in Proposition~\ref{prop:intro}\eqref{eq:reflectionsymmetry} and~\eqref{eq:complementarity} translate to natural actions on summing tableaux, see also~Figure~\ref{fig:poset}.
\begin{corollary}
  \label{cor:mapsontableaux}
  The action on $\NC_{k,n}$ induced by $a \mapsto n+1-a$ corresponds to a $180^\circ$ rotation of the summing tableau.
  The map from $\NC_{k,n}$ to $\NC_{n-k,n}$ induced by $I \mapsto [n] \setminus I$ corresponds to transposing the 
  summing tableau along the north-west-to-south-east diagonal.
\end{corollary}

It will be convenient in the following to represent an $\ell$-multiset $L$ of vectors in~$V_{k,n}$ as the $(k \times \ell)$-table 
containing the vectors in~$L$ as columns, in lexicographic order.
For example, let $n=7$ and $k=3$, and consider the multiset given by the $(3 \times 9)$-table

\smallskip
\begin{center}
\begin{tikzpicture}[node distance=0 cm,outer sep = 0pt]
  \tikzstyle{bsq}=[minimum width=.6cm, minimum height=.6cm]
  \tikzstyle{bsq2}=[rectangle,draw,opacity=.25,fill opacity=1]
  \tikzstyle{cor}=[anchor=north west,inner sep=1pt]

  \node[bsq] (r1) at (-.3, .3)    {\scriptsize 3};
  \node[bsq] (r2) [above = of r1] {\scriptsize 2};
  \node[bsq] (r3) [above = of r2] {\scriptsize 1};

  \node[bsq] (c1) at (.3, -.3)    {\scriptsize 1};
  \node[bsq] (c2) [right = of c1] {\scriptsize 2};
  \node[bsq] (c3) [right = of c2] {\scriptsize 3};
  \node[bsq] (c4) [right = of c3] {\scriptsize 4};
  \node[bsq] (c5) [right = of c4] {\scriptsize 5};
  \node[bsq] (c6) [right = of c5] {\scriptsize 6};
  \node[bsq] (c7) [right = of c6] {\scriptsize 7};
  \node[bsq] (c8) [right = of c7] {\scriptsize 8};
  \node[bsq] (c9) [right = of c8] {\scriptsize 9};

  \node[bsq,bsq2] (11) at (.3, 1.5)     {1};
  \node[bsq,bsq2] (12) [right = of 11]  {1};
  \node[bsq,bsq2] (13) [right = of 12]  {1};
  \node[bsq,bsq2] (14) [right = of 13]  {2};
  \node[bsq,bsq2] (15) [right = of 14]  {2};
  \node[bsq,bsq2] (16) [right = of 15]  {2};
  \node[bsq,bsq2] (17) [right = of 16]  {2};
  \node[bsq,bsq2] (18) [right = of 17]  {3};
  \node[bsq,bsq2] (19) [right = of 18]  {5};

  \node[bsq,bsq2] (21) at (.3, .9)      {2};
  \node[bsq,bsq2] (22) [right = of 21]  {2};
  \node[bsq,bsq2] (23) [right = of 22]  {3};
  \node[bsq,bsq2] (24) [right = of 23]  {3};
  \node[bsq,bsq2] (25) [right = of 24]  {4};
  \node[bsq,bsq2] (26) [right = of 25]  {4};
  \node[bsq,bsq2] (27) [right = of 26]  {5};
  \node[bsq,bsq2] (28) [right = of 27]  {5};
  \node[bsq,bsq2] (29) [right = of 28]  {6};

  \node[bsq,bsq2] (31) at (.3,.3)       {3};
  \node[bsq,bsq2] (32) [right = of 31]  {4};
  \node[bsq,bsq2] (33) [right = of 32]  {5};
  \node[bsq,bsq2] (34) [right = of 33]  {5};
  \node[bsq,bsq2] (35) [right = of 34]  {5};
  \node[bsq,bsq2] (36) [right = of 35]  {5};
  \node[bsq,bsq2] (37) [right = of 36]  {7};
  \node[bsq,bsq2] (38) [right = of 37]  {7};
  \node[bsq,bsq2] (39) [right = of 38]  {7};

  \node      (00) [left = of r2]  {$L = $};
\end{tikzpicture}
\end{center}
\smallskip
Its summing tableau is 

\smallskip
\begin{center}
\begin{tikzpicture}[node distance=0 cm,outer sep = 0pt]
  \tikzstyle{bsq}=[rectangle, draw,opacity=1,fill opacity=1, minimum width=.6cm, minimum height=.6cm]
  \node[bsq] (11) at (1, 1)       {3};
  \node[bsq] (12) [right = of 11] {7};
  \node[bsq] (13) [right = of 12] {8};
  \node[bsq] (14) [right = of 13] {8};
  \node[bsq] (21) [below = of 11] {2};
  \node[bsq] (22) [right = of 21] {4};
  \node[bsq] (23) [right = of 22] {6};
  \node[bsq] (24) [right = of 23] {8};
  \node[bsq] (31) [below = of 21] {1};
  \node[bsq] (32) [right = of 31] {2};
  \node[bsq] (33) [right = of 32] {6};
  \node[bsq] (34) [right = of 33] {6};
  \node      (00) [left = of 21]  {$T =\ $};
\end{tikzpicture}
\end{center}
\smallskip

For example the first row of~$T$ says that the first row of~$L$ contains three~$1$'s, four~$2$'s, and one~$3$, while the last vector $(5,6,7)$ does not contribute to~$T$ as $\chi_{(5,6,7)} = 0 \in \mathbb{N}^{P_{3,7}}$.
As in this example, if~$L$ contains the vector $I_{\hat 1}:=(n-k+1,\ldots,n) \in V_{k,n}$, that vector does not contribute to the summing tableau since $\chi_{I_{\hat 1}} = 0 \in \mathbb{N}^{P_{k,n}}$.
We thus set $V_{k,n}^* = V_{k,n} \setminus \big\{\ (n-k+1,\ldots,n)\ \big\}$ for later convenience.

\medskip

The following statement is at the basis of our results about the two simplicial complexes $\NN_{k,n}$ and $\NC_{k,n}$.

\begin{theorem}
\label{thm:uniquemultiset}
  Let~$T \in \Tab_{k,n}$.
  Then there is a unique multiset $\bijNN(T)$ and a unique multiset $\bijNC(T)$ of vectors in $V^*_{k,n}$ whose 
  summing tableaux are~$T$, and such that
  \begin{itemize}
    \item the vectors in $\bijNN(T)$ are mutually nonnesting, and
    \item the vectors in $\bijNC(T)$ are mutually noncrossing.
  \end{itemize}
\end{theorem}

In order to prove this, we provide two (almost identical) procedures to construct $\bijNN(T)$ and $\bijNC(T)$.

\medskip

Let~$T=(t_{ab}) \in \Tab_{k,n}$ be a tableau, and let $\ell = \max(T)=t_{1,n-k}$ be its maximal entry.
We are going to fill a $(k\times \ell)$-table whose columns give~$\bijNN(T)$ and~$\bijNN(T)$, respectively.
Since we want each column to be in $V_{k,n}^*$, we have to fill the $a$\textsuperscript{th} row ($a \in \{1,\ldots,k\}$) 
with numbers in $\{a,\ldots,a+n-k\}$.
Moreover, in order to have $T$ as the summing tableau of the multiset of columns, the number $a+b$ must appear in 
the $a$\textsuperscript{th} row exactly
$
  t_{a,b+1}- t_{a,b}
$
times, where we use the convention $t_{a,0}=0$ and $t_{a,n-k+1}=\max(T)$.
That is, we do not have a choice of \emph{which} entries to use in each row, but only on \emph{where} to put them.
Our procedure is to fill the table row by row from top to bottom, inserting the entries $a+1,\ldots,a+n-k$ in increasing order (each of them the prescribed number of times) placing them one after the other into the \lq\lq next\rq\rq\ column in the~$a$\textsuperscript{th} row of the table.
The only difference between $\bijNN$ and $\bijNC$ is how the term \lq\lq next\rq\rq\ is defined.

\smallskip
\begin{itemize}
\item To obtain~$\bijNN$, \lq\lq next\rq\rq\ is simply the next free box from left to right.
In the above example, the table gets filled as follows.
\smallskip
\begin{center}
\begin{tikzpicture}[node distance=0 cm,outer sep = 0pt]
  \tikzstyle{bsq}=[minimum width=.6cm, minimum height=.6cm]
  \tikzstyle{bsq2}=[rectangle,draw,opacity=.25,fill opacity=1]
  \tikzstyle{cor}=[anchor=north west,inner sep=1pt]

  \draw[ thin] (1.5, 1.5) circle(0.20);
  \draw[ thin] (3.9, 1.5) circle(0.20);
  \draw[ thin] (4.5, 1.5) circle(0.20); % removed double
  \draw[ thin] (0.9, 0.9) circle(0.20); 
  \draw[ thin] (2.1, 0.9) circle(0.20);
  \draw[ thin] (3.3, 0.9) circle(0.20);
  \draw[ thin] (4.5, 0.9) circle(0.20);
  \draw[ thin] (0.3, 0.3) circle(0.20);
  \draw[ thin] (0.9, 0.3) circle(0.20);
  \draw[ thin] (3.3, 0.3) circle(0.20); % removed double

  \node[bsq] (r1) at (-.3, .3)    {\scriptsize 3};
  \node[bsq] (r2) [above = of r1] {\scriptsize 2};
  \node[bsq] (r3) [above = of r2] {\scriptsize 1};

  \node[bsq] (c1) at (.3, -.3)    {\scriptsize 1};
  \node[bsq] (c2) [right = of c1] {\scriptsize 2};
  \node[bsq] (c3) [right = of c2] {\scriptsize 3};
  \node[bsq] (c4) [right = of c3] {\scriptsize 4};
  \node[bsq] (c5) [right = of c4] {\scriptsize 5};
  \node[bsq] (c6) [right = of c5] {\scriptsize 6};
  \node[bsq] (c7) [right = of c6] {\scriptsize 7};
  \node[bsq] (c8) [right = of c7] {\scriptsize 8};

  \node[bsq,bsq2] (11) at (.3, 1.5)     {1};
  \node[bsq,bsq2] (12) [right = of 11]  {1};
  \node[bsq,bsq2] (13) [right = of 12]  {1};
  \node[bsq,bsq2] (14) [right = of 13]  {2};
  \node[bsq,bsq2] (15) [right = of 14]  {2};
  \node[bsq,bsq2] (16) [right = of 15]  {2};
  \node[bsq,bsq2] (17) [right = of 16]  {2};
  \node[bsq,bsq2] (18) [right = of 17]  {3};

  \node[bsq,bsq2] (21) at (.3, .9)      {2};
  \node[bsq,bsq2] (22) [right = of 21]  {2};
  \node[bsq,bsq2] (23) [right = of 22]  {3};
  \node[bsq,bsq2] (24) [right = of 23]  {3};
  \node[bsq,bsq2] (25) [right = of 24]  {4};
  \node[bsq,bsq2] (26) [right = of 25]  {4};
  \node[bsq,bsq2] (27) [right = of 26]  {5};
  \node[bsq,bsq2] (28) [right = of 27]  {5};

  \node[bsq,bsq2] (31) at (.3,.3)       {3};
  \node[bsq,bsq2] (32) [right = of 31]  {4};
  \node[bsq,bsq2] (33) [right = of 32]  {5};
  \node[bsq,bsq2] (34) [right = of 33]  {5};
  \node[bsq,bsq2] (35) [right = of 34]  {5};
  \node[bsq,bsq2] (36) [right = of 35]  {5};
  \node[bsq,bsq2] (37) [right = of 36]  {7};
  \node[bsq,bsq2] (38) [right = of 37]  {7};

  \node      (00) [left = of r2]  {$\bijNN(T) = $};

  \node[cor] at (11.north west) {\tiny 1};
  \node[cor] at (12.north west) {\tiny 2};
  \node[cor] at (13.north west) {\tiny 3};
  \node[cor] at (14.north west) {\tiny 4};
  \node[cor] at (15.north west) {\tiny 5};
  \node[cor] at (16.north west) {\tiny 6};
  \node[cor] at (17.north west) {\tiny 7};
  \node[cor] at (18.north west) {\tiny 8};

  \node[cor] at (21.north west) {\tiny 1};
  \node[cor] at (22.north west) {\tiny 2};
  \node[cor] at (23.north west) {\tiny 3};
  \node[cor] at (24.north west) {\tiny 4};
  \node[cor] at (25.north west) {\tiny 5};
  \node[cor] at (26.north west) {\tiny 6};
  \node[cor] at (27.north west) {\tiny 7};
  \node[cor] at (28.north west) {\tiny 8};

  \node[cor] at (31.north west) {\tiny 1};
  \node[cor] at (32.north west) {\tiny 2};
  \node[cor] at (33.north west) {\tiny 3};
  \node[cor] at (34.north west) {\tiny 4};
  \node[cor] at (35.north west) {\tiny 5};
  \node[cor] at (36.north west) {\tiny 6};
  \node[cor] at (37.north west) {\tiny 7};
  \node[cor] at (38.north west) {\tiny 8};

\end{tikzpicture}
\end{center}
To help the reader, in the top-left corner of each box we indicate the order in which a given entry is inserted into its row of the table.
Also, we have marked with a circle the last occurrence of each entry in each row.
These circle-marks are not needed for the proof of Theorem~\ref{thm:uniquemultiset} but will become important later.

\smallskip

\item To obtain $\bijNC$, \lq\lq next\rq\rq\ is slightly more complicated.
For two vectors $v,w \in \mathbb{N}^k$ we say that $v$ precedes $w$ in {\bf revlex order} if
the rightmost entry of $w-v$ different from $0$ is positive.
We chose the {\bf revlex-largest} vector whose $a$\textsuperscript{th} entry has not yet been 
inserted and for which the property of strictly increasing entries in a column is preserved.
In other words, inserting an integer $i$ into row~$a$ is done by looking at the first $a-1$ entries 
$v = (v_1,\ldots,v_{a-1})$ of all vectors that have not been assigned an~$a$\textsuperscript{th} entry yet and such that $v_{a-1} < i$. Among those, we assign~$i$ to the revlex-largest free box, i.e., to that~$v$ for which $v_{a-1}$ is maximal, then $v_{a-2}$ is maximal, and so on.
If this revlex-largest vector is not unique, we fill the box of the left-most of the choices, in order to maintain the table columns in lexicographic order.
In the above example, the table now gets filled as follows.

\smallskip
\begin{center}
\begin{tikzpicture}[node distance=0 cm,outer sep = 0pt]
  \tikzstyle{bsq}=[minimum width=.6cm, minimum height=.6cm]
  \tikzstyle{bsq2}=[rectangle,draw,opacity=.25,fill opacity=1]
  \tikzstyle{cor}=[anchor=north west,inner sep=1pt]

  \node[bsq] (r1) at (-.3, .3)    {\scriptsize 3};
  \node[bsq] (r2) [above = of r1] {\scriptsize 2};
  \node[bsq] (r3) [above = of r2] {\scriptsize 1};

  \node[bsq] (c1) at (.3, -.3)    {\scriptsize 1};
  \node[bsq] (c2) [right = of c1] {\scriptsize 2};
  \node[bsq] (c3) [right = of c2] {\scriptsize 3};
  \node[bsq] (c4) [right = of c3] {\scriptsize 4};
  \node[bsq] (c5) [right = of c4] {\scriptsize 5};
  \node[bsq] (c6) [right = of c5] {\scriptsize 6};
  \node[bsq] (c7) [right = of c6] {\scriptsize 7};
  \node[bsq] (c8) [right = of c7] {\scriptsize 8};

  \node[bsq,bsq2] (11) at (.3, 1.5)     {1};
  \node[bsq,bsq2] (12) [right = of 11]  {1};
  \node[bsq,bsq2] (13) [right = of 12]  {1};
  \node[bsq,bsq2] (14) [right = of 13]  {2};
  \node[bsq,bsq2] (15) [right = of 14]  {2};
  \node[bsq,bsq2] (16) [right = of 15]  {2};
  \node[bsq,bsq2] (17) [right = of 16]  {2};
  \node[bsq,bsq2] (18) [right = of 17]  {3};

  \node[bsq,bsq2] (21) at (.3, .9)      {2};
  \node[bsq,bsq2] (22) [right = of 21]  {2};
  \node[bsq,bsq2] (23) [right = of 22]  {5};
  \node[bsq,bsq2] (24) [right = of 23]  {3};
  \node[bsq,bsq2] (25) [right = of 24]  {3};
  \node[bsq,bsq2] (26) [right = of 25]  {4};
  \node[bsq,bsq2] (27) [right = of 26]  {5};
  \node[bsq,bsq2] (28) [right = of 27]  {4};

  \node[bsq,bsq2] (31) at (.3,.3)       {3};
  \node[bsq,bsq2] (32) [right = of 31]  {5};
  \node[bsq,bsq2] (33) [right = of 32]  {7};
  \node[bsq,bsq2] (34) [right = of 33]  {4};
  \node[bsq,bsq2] (35) [right = of 34]  {5};
  \node[bsq,bsq2] (36) [right = of 35]  {5};
  \node[bsq,bsq2] (37) [right = of 36]  {7};
  \node[bsq,bsq2] (38) [right = of 37]  {5};

  \node      (00) [left = of r2]  {$\bijNC(T) = $};

  \node[cor] at (11.north west) {\tiny 1};
  \node[cor] at (12.north west) {\tiny 2};
  \node[cor] at (13.north west) {\tiny 3};
  \node[cor] at (14.north west) {\tiny 4};
  \node[cor] at (15.north west) {\tiny 5};
  \node[cor] at (16.north west) {\tiny 6};
  \node[cor] at (17.north west) {\tiny 7};
  \node[cor] at (18.north west) {\tiny 8};

  \node[cor] at (21.north west) {\tiny 1};
  \node[cor] at (22.north west) {\tiny 2};
  \node[cor] at (23.north west) {\tiny 8};
  \node[cor] at (24.north west) {\tiny 3};
  \node[cor] at (25.north west) {\tiny 4};
  \node[cor] at (26.north west) {\tiny 6};
  \node[cor] at (27.north west) {\tiny 7};
  \node[cor] at (28.north west) {\tiny 5};

  \node[cor] at (31.north west) {\tiny 1};
  \node[cor] at (32.north west) {\tiny 6};
  \node[cor] at (33.north west) {\tiny 8};
  \node[cor] at (34.north west) {\tiny 2};
  \node[cor] at (35.north west) {\tiny 5};
  \node[cor] at (36.north west) {\tiny 4};
  \node[cor] at (37.north west) {\tiny 7};
  \node[cor] at (38.north west) {\tiny 3};

  \draw[ thin] (1.5, 1.5) circle(0.20);
  \draw[ thin] (3.9, 1.5) circle(0.20);
  \draw[ thin] (4.5, 1.5) circle(0.20); % removed double
  \draw[ thin] (0.9, 0.9) circle(0.20); 
  \draw[ thin] (1.5, 0.9) circle(0.20);
  \draw[ thin] (2.7, 0.9) circle(0.20);
  \draw[ thin] (3.3, 0.9) circle(0.20);
  \draw[ thin] (0.3, 0.3) circle(0.20);
  \draw[ thin] (0.9, 0.3) circle(0.20); % removed double
  \draw[ thin] (2.1, 0.3) circle(0.20);

\end{tikzpicture}
\end{center}
\smallskip

\end{itemize}

Observe that we could have described the choice of \lq\lq next\rq\rq\ position in the procedure~$\bijNN$ by saying that it means the {\bf revlex-smallest} vector (in the same sense as above for $\bijNC$) for which the property of strictly increasing entries in a column is preserved.
This makes  both procedures  almost identical, only interchanging revlex-smallest and revlex-largest in the choice of the box to insert the next integer.

\begin{remark}
  \label{rem:topbottom}
  Observe that these procedures could, by Proposition~\ref{prop:intro}\eqref{eq:reflectionsymmetry}, also be applied \lq\lq from bottom to top\rq\rq\ by first inserting the last row, and then filling the table row by row by the analogous lex- and revlex-insertions.
\end{remark}

\begin{definition}
\label{def:decomposition}
The multisets $\bijNN$ and $\bijNC$ obtained from a tableau $T$ by the above procedures are called the \emph{nonnesting decomposition} and the \emph{noncrossing decomposition} of $T$.
\end{definition}

\begin{proof}[Proof of Theorem~\ref{thm:uniquemultiset}]
  We start with proving that the procedures give what they are supposed to:
  every two columns of $\bijNN(T)$ are nonnesting, and every two columns of $\bijNC(T)$ are noncrossing.
  To this end, let $I = (i_1,\ldots,i_k)$ and $J = (j_1,\ldots,j_k)$ be two such columns, and let $a,b$ be two indices such that $i_\ell = j_\ell$ for all $\ell$ such that $a < \ell < b$.
  To show that if $I$ and $J$ in $\bijNN(T)$ (resp. in $\bijNC(T)$) are nonnesting (resp. noncrossing), we have to show that $i_a < j_a$ implies  $i_b \leq j_b$ (resp. $i_b \geq j_b$).
  When assigning row~$b$ in the table, we see $i_a,i_{a+1},\ldots,i_{b-1}$ in the column containing~$I$, and similarly $j_a,j_{a+1},\ldots,j_{b-1}$ in the column containing~$J$.
  In this situation, the column containing~$I$ is filled before the column containing~$J$ for $\bijNN$ and after the column~$J$ for $\bijNC$.
  Thus, $i_b \leq j_b$ for $\bijNN$ and $i_b \geq j_b$ for $\bijNC$.

  \smallskip

  To show uniqueness, suppose that we would have not chosen the revlex-smallest (resp. revlex-largest) column at some point in the procedure.
  The same argument as before then implies that we then would have created two nesting (resp. crossing) columns.
\end{proof}

\subsection{Further properties of the complexes $\NN_{k,n}$ and $\NC_{k,n}$}
\label{sec:furthercombi}

We emphasize that Theorem~\ref{thm:uniquemultiset} alone, suitably interpreted, implies that the complexes $\NN_{k,n}$ and $\NC_{k,n}$ are unimodular triangulations of the order polytope $\oO_{k,n}$. This interpretation is carried out in Section~\ref{sec:geometry}, after some preliminaries on triangulations and order polytopes that we briefly survey in Section~\ref{sec:background}. 
Before getting there, we prove in the remainder of this section further combinatorial properties of the two complexes $\NN_{k,n}$ and $\NC_{k,n}$.
Some of these properties (e.g. the pureness of the complexes in Corollary~\ref{coro:pure}) follow also from the geometric results in the subsequent sections, but we think that it is interesting to have independent combinatorial proofs.

\medskip

To further understand the combinatorics of $\NN_{k,n}$ and $\NC_{k,n}$, we mark (indicated by a circle in the figures) the 
last occurrence of every \emph{nonmaximal} integer placed in each row of $\bijNN(T)$ and of $\bijNC(T)$ in the above 
construction procedure. We call them the \defn{marked positions}.
In symbols, for each $a \in [k]$ and $b \in [n-k]$ we mark the last occurrence of~$a+b-1$ that is placed in row~$a$.
Here, \lq\lq last occurrence\rq\rq\ is meant in the order the integer is inserted into the given row.
For two examples, see the above instances of~$\bijNN(T)$ and of~$\bijNC(T)$.

\begin{remark}
\label{rem:bars}
  One can ask in which way the marked positions differ if we fill the table from bottom to top according to Remark~\ref{rem:topbottom}.
  If the position of the last occurrence of the maximal integer $a+n-k$ in row~$a$ is marked, it turns out that both procedures provide 
  the same marked positions. In other words, the marked positions do not depend on the procedure, but can be described purely in terms 
  of the table~$L$, except that the rule is different when describing the marked positions of a noncrossing table or a nonnesting table.
  In both cases, we assume that the table ~$L$ has no repeated columns (if it has, only one copy carries marks).
  There is going to be one mark for each row $a\in[k]$ and each $b\in[n-k+1]$.
  The mark will be in one of the vectors in
  \[
  L_{a,b}:=\big\{ (i_1,\ldots,i_k)\in L: i_a=a+b-1 \big\}.
  \]
  The rule to decide which vector carries the mark is:
  \begin{itemize}

    \item For the marks in the nonnesting table, the mark lies in the vector $I=(i_1,\ldots,i_k) \in L_{a,b}$ for which $(i_1,\ldots,i_a)$ is smallest and 
      $(i_{a+1},\ldots,i_{k})$ is smallest. Observe that there is no inconsistency on what of the two rules we look at first, since nonnesting 
      vectors are component-wise comparable. For the same reason, ``smallest'' means both lex-smallest and revlex-smallest.

    \item For the marks in the noncrossing table, the mark lies in the vector $I=(i_1,\ldots,i_k) \in L_{a,b}$ for which 
      $(i_1,\ldots,i_a)$ is revlex-smallest and $(i_{a+1},\ldots,i_{k})$ is lex-largest.

  \end{itemize}
\end{remark}

\smallskip

In the following three lemmas, we collect further properties of the tables $\bijNN(T)$ and $\bijNC(T)$, some of which can be detected using the information 
where the last occurrences of the entries are placed.

\begin{lemma}
\label{le:furtherproperties1}
  Let~$T \in \Tab_{k,n}$, and let~$L$ be either $\bijNN(T)$ or $\bijNC(T)$.
  We then have the following properties for~$L$.
  \begin{enumerate}[(i)]

    \item The columns in~$L$ are ordered lexicographically. \label{eq:lexorder}

    \item $T_{k,1} > 0$ if and only if the vector $(1,\ldots,k)$ is a column of~$L$. \label{eq:notzero}

    \item Let $L_i$ and $L_{i+1}$ be two two consecutive columns of~$L$. Then the first row in which $L_i$ and $L_{i+1}$ differ is equal to the first row in which $L_i$ has a marked position. \label{eq:firstdiff}

    \item Two consecutive columns $L_i$ and $L_{i+1}$ coincide if and only if $L_i$ has no marked position. \label{eq:repetition}
  \end{enumerate}
\end{lemma}
\begin{proof}
  These properties can be directly read off the insertion procedures for~$\bijNN(T)$ and for~$\bijNC(T)$ and the definition of the marked positions.
\end{proof}

\begin{lemma}
\label{le:furtherproperties2}
  Let~$T \in \Tab_{k,n}$.
  We have the following property for $\bijNN(T)$ which does not hold for $\bijNC(T)$.
  \begin{enumerate}[(i)]
    \item two consecutive columns of $\bijNN(T)$ differ in exactly those positions in which the left of the two has marked positions. \label{eq:alldiffNN}
  \end{enumerate}
\end{lemma}
\begin{proof}
  This is also a direct consequence of the procedure and the definition of the marked positions.
\end{proof}

\begin{lemma}
\label{le:furtherproperties3}
  Let~$T \in \Tab_{k,n}$.
  We have the following properties for $\bijNC(T)$ that do not hold for $\bijNN(T)$.
  \begin{enumerate}[(i)]
    \item $T$ is strictly increasing along rows if and only if all vectors of the form $(b+1,\ldots,b+k)$ for $b \in [n-k-1]$ appear as columns in $\bijNC(T)$. \label{eq:alongrows}

    \item $T$ is strictly increasing along columns if and only if all vectors of the form $(1,\ldots,a,n-k+a+1,\ldots,n)$ for $a \in [k-1]$ appear as columns in $\bijNC(T)$. \label{eq:alongcols}
  \end{enumerate}
\end{lemma}
\begin{proof}
  Observe that the revlex-max insertion ensures that the first insertion of a given integer~$a+b-1$ into row~$a$ yields a partial vector (of length~$a$) of the form~$(b,\ldots,a+b-1)$.
  This implies that for fixed $b \in [n-k-1]$, we have that~$T_{a,b} < T_{a,b+1}$ for all $a \in [k]$ if and only if $(b+1,\ldots,b+k)$ is a column of $\bijNC(T)$, thus implying~\eqref{eq:alongrows}.
  A similar observation holds as well for the columns of~$T$.
  For fixed $a \in [k-1]$, we have that~$T_{a,b} < T_{a+1,b}$ for all $b \in [n-k]$ if and only if $(1,\ldots,a,n-k+a+1,\ldots,n)$ is a column of $\bijNC(T)$, thus implying~\eqref{eq:alongcols}.
\end{proof}

We are now ready to prove the following proposition from which we will then derive the pureness and the dimension of the nonnesting and the noncrossing complexes.

\begin{proposition}
\label{prop:furtherproperties}
  Let~$T \in \Tab_{k,n}$, and let~$L$ be either $\bijNN(T)$ or $\bijNC(T)$.
  If a column of~$L$ contains more than one marked position, then there exists a vector in $V^*_{k,n}$ that is not contained in~$L$ and which does not nest or cross, respectively,  with any vector in~$L$, depending on $L$ being $\bijNN(T)$ or $\bijNC(T)$.
\end{proposition}

\begin{proof}
  In the case of $L = \bijNN(T)$, this is a direct consequence of Lemma~\ref{le:furtherproperties2}:
  Given a column with multiple marked positions, we can always insert a new column to the right of this column where only the last marked position is changed.
  This vector is nonnesting with all other vectors by construction.
  E.g., the second column in the above example for $\bijNN(T)$ has the second and the third position marked. We can thus insert a new column between the second and the third where only the last position is changed, thus being the vector $(1,2,5)$.

  The case of $L = \bijNC(T)$ is a little more delicate.
  First, we can assume that~$T_{k,1} > 0$, and that~$T$ is strictly increasing along rows and columns.
  Otherwise, we can, according to Lemma~\ref{le:furtherproperties1}\eqref{eq:notzero} and Lemma~\ref{le:furtherproperties3}, together with Proposition~\ref{prop:intro}\eqref{eq:join}, always insert the missing cyclic intervals as columns into $\bijNC(T)$ and modify~$T$ accordingly.

  Given this situation and a column containing more than one marked position, one can \lq\lq push\rq\rq\ marked position to obtain a new vector in~$V^*_{k,n}$ that is not yet contained in~$L$, and which is noncrossing with every column of~$L$.
  Since it is enough for our purposes here, we describe the procedure of pushing the last marked position, similarly to the situation for $\bijNN(T)$.
  To this end, let~$b$ be the column containing more than a single marked positions, and let the last marked position be in row~$a$.
  Moreover, let the value of this last marked position be~$x$.
  Now, pretend that we had one more~$x$ to be inserted into row~$a$ in this table.
  The condition that~$T$ is strictly increasing along rows and columns implies that it would indeed be possible to insert another~$x$ into row~$a$.
  Let~$b'$ be the column in which this next~$x$ would be inserted.
  Then, it would be possible to add a new column between columns $b'-1$ and~$b'$ containing the vector given by the first $a-1$ entries of the previous column~$b'$ together with the remaining entries of column~$b$.
  Call the resulting table~$L'$.
  Observe that by construction,~$L'$ equals the revlex-max insertion table of its summing tableau.
  This implies that all columns in~$L'$ are noncrossing.
  Moreover, the marked positions of~$L'$ are exactly those of~$L$, except that the last marked position in the previous column~$b$ has moved to the new column~$b'$.
  As an example, we consider the following table (which is the previously considered table with all extra vectors inserted as described earlier in this proof).

  \smallskip
  \begin{center}
  \begin{tikzpicture}[node distance=0 cm,outer sep = 0pt]
    \tikzstyle{bsq}=[minimum width=.6cm, minimum height=.6cm]
    \tikzstyle{bsq2}=[rectangle,draw=black,opacity=.25,fill opacity=1]
    \tikzstyle{cor}=[anchor=north west,inner sep=1pt]

    \node[bsq] (r1) at (-.3, .3)    {\scriptsize 3};
    \node[bsq] (r2) [above = of r1] {\scriptsize 2};
    \node[bsq] (r3) [above = of r2] {\scriptsize 1};

    \node[bsq] (c1) at (.3, -.3)    {\scriptsize 1};
    \node[bsq] (c2) [right = of c1] {\scriptsize 2};
    \node[bsq] (c3) [right = of c2] {\scriptsize 3};
    \node[bsq] (c4) [right = of c3] {\scriptsize 4};
    \node[bsq] (c5) [right = of c4] {\scriptsize 5};
    \node[bsq] (c6) [right = of c5] {\scriptsize 6};
    \node[bsq] (c7) [right = of c6] {\scriptsize 7};
    \node[bsq] (c8) [right = of c7] {\scriptsize 8};
    \node[bsq] (c9) [right = of c8] {\scriptsize 9};
    \node[bsq] (c10) [right = of c9] {\scriptsize 10};
    \node[bsq] (c11) [right = of c10] {\scriptsize 11};

    \node[bsq,bsq2] (11) at (.3, 1.5)     {1};
    \node[bsq,bsq2] (12) [right = of 11]  {1};
    \node[bsq,bsq2] (13) [right = of 12]  {1};
    \node[bsq,bsq2] (14) [right = of 13]  {1};
    \node[bsq,bsq2] (15) [right = of 14]  {1};
    \node[bsq,bsq2] (16) [right = of 15]  {2};
    \node[bsq,bsq2] (17) [right = of 16]  {2};
    \node[bsq,bsq2] (18) [right = of 17]  {2};
    \node[bsq,bsq2] (19) [right = of 18]  {2};
    \node[bsq,bsq2] (110) [right = of 19]  {3};
    \node[bsq,bsq2] (111) [right = of 110]  {4};

    \node[bsq,bsq2] (21) at (.3, .9)      {2};
    \node[bsq,bsq2] (22) [right = of 21]  {2};
    \node[bsq,bsq2] (23) [right = of 22]  {2};
    \node[bsq,bsq2] (24) [right = of 23]  {5};
    \node[bsq,bsq2] (25) [right = of 24]  {6};
    \node[bsq,bsq2] (26) [right = of 25]  {3};
    \node[bsq,bsq2] (27) [right = of 26]  {3};
    \node[bsq,bsq2] (28) [right = of 27]  {4};
    \node[bsq,bsq2] (29) [right = of 28]  {5};
    \node[bsq,bsq2] (210) [right = of 29]  {4};
    \node[bsq,bsq2] (211) [right = of 210]  {5};

    \node[bsq,bsq2] (31) at (.3,.3)       {3};
    \node[bsq,bsq2] (32) [right = of 31]  {5};
    \node[bsq,bsq2] (33) [right = of 32]  {7};
    \node[bsq,bsq2] (34) [right = of 33]  {7};
    \node[bsq,bsq2] (35) [right = of 34]  {7};
    \node[bsq,bsq2] (36) [right = of 35]  {4};
    \node[bsq,bsq2] (37) [right = of 36]  {5};
    \node[bsq,bsq2] (38) [right = of 37]  {5};
    \node[bsq,bsq2] (39) [right = of 38]  {7};
    \node[bsq,bsq2] (310) [right = of 39]  {5};
    \node[bsq,bsq2] (311) [right = of 310]  {6};

    \node[cor] at (11.north west) {\tiny 1};
    \node[cor] at (12.north west) {\tiny 2};
    \node[cor] at (13.north west) {\tiny 3};
    \node[cor] at (14.north west) {\tiny 4};
    \node[cor] at (15.north west) {\tiny 5};
    \node[cor] at (16.north west) {\tiny 6};
    \node[cor] at (17.north west) {\tiny 7};
    \node[cor] at (18.north west) {\tiny 8};
    \node[cor] at (19.north west) {\tiny 9};
    \node[cor] at (110.north west) {\tiny 1\!0};
    \node[cor] at (111.north west) {\tiny 1\!1};

    \node[cor] at (21.north west) {\tiny 1};
    \node[cor] at (22.north west) {\tiny 2};
    \node[cor] at (23.north west) {\tiny 3};
    \node[cor] at (24.north west) {\tiny 1\!0};
    \node[cor] at (25.north west) {\tiny 1\!1};
    \node[cor] at (26.north west) {\tiny 4};
    \node[cor] at (27.north west) {\tiny 5};
    \node[cor] at (28.north west) {\tiny 7};
    \node[cor] at (29.north west) {\tiny 9};
    \node[cor] at (210.north west) {\tiny 6};
    \node[cor] at (211.north west) {\tiny 8};

    \node[cor] at (31.north west) {\tiny 1};
    \node[cor] at (32.north west) {\tiny 6};
    \node[cor] at (33.north west) {\tiny 1\!1};
    \node[cor] at (34.north west) {\tiny 1\!0};
    \node[cor] at (35.north west) {\tiny 8};
    \node[cor] at (36.north west) {\tiny 2};
    \node[cor] at (37.north west) {\tiny 5};
    \node[cor] at (38.north west) {\tiny 4};
    \node[cor] at (39.north west) {\tiny 9};
    \node[cor] at (310.north west) {\tiny 3};
    \node[cor] at (311.north west) {\tiny 7};

  %  In case circles are preferred to bars...
    \draw[ thin] (2.7, 1.5) circle(0.20);
    \draw[ thin] (5.1, 1.5) circle(0.20);
    \draw[ thin] (5.7, 1.5) circle(0.20); % removed double
    \draw[ thin] (6.3, 1.5) circle(0.20); % removed double

    \draw[ thin] (1.5, 0.9) circle(0.20); 
    \draw[ thin] (2.1, 0.9) circle(0.20);
    \draw[ thin] (3.9, 0.9) circle(0.20);
    \draw[ thin] (4.5, 0.9) circle(0.20);
    
    \draw[ thin] (0.3, 0.3) circle(0.20);
    \draw[ thin] (0.9, 0.3) circle(0.20); % removed double
    \draw[ thin] (3.3, 0.3) circle(0.20);
    \draw[ thin] (6.3, 0.3) circle(0.20);

    \node  (00) [left = of r2]  {$L = $};

  \end{tikzpicture}
  \end{center}
  \smallskip

  Now, consider the last column $b=11$, having marked positions in rows~$1$ and~$3$, with values~$4$ and~$6$, respectively.
  So, a new last~$6$ would be inserted into column~$b'=9$.
  The resulting table~$L'$ has a new column between columns~$8$ and~$9$ consisting of the first~$2$ entries of the previous column~$9$ and the last entry of the previous column~$11$, thus being the vector $(2,5,6)$.
  We therefore get the following table, extending~$L$ by one column.

  \smallskip
  \begin{center}
  \begin{tikzpicture}[node distance=0 cm,outer sep = 0pt]
    \tikzstyle{bsq}=[minimum width=.6cm, minimum height=.6cm]
    \tikzstyle{bsq2}=[rectangle,draw=black,opacity=.25,fill opacity=1]
    \tikzstyle{cor}=[anchor=north west,inner sep=1pt]

    \node[bsq] (r1) at (-.3, .3)    {\scriptsize 3};
    \node[bsq] (r2) [above = of r1] {\scriptsize 2};
    \node[bsq] (r3) [above = of r2] {\scriptsize 1};

    \node[bsq] (c1) at (.3, -.3)    {\scriptsize 1};
    \node[bsq] (c2) [right = of c1] {\scriptsize 2};
    \node[bsq] (c3) [right = of c2] {\scriptsize 3};
    \node[bsq] (c4) [right = of c3] {\scriptsize 4};
    \node[bsq] (c5) [right = of c4] {\scriptsize 5};
    \node[bsq] (c6) [right = of c5] {\scriptsize 6};
    \node[bsq] (c7) [right = of c6] {\scriptsize 7};
    \node[bsq] (c8) [right = of c7] {\scriptsize 8};
    \node[bsq] (c9) [right = of c8] {\scriptsize 9};
    \node[bsq] (c10) [right = of c9] {\scriptsize 10};
    \node[bsq] (c11) [right = of c10] {\scriptsize 11};
    \node[bsq] (c11) [right = of c11] {\scriptsize 12};

    \node[bsq,bsq2] (11) at (.3, 1.5)     {1};
    \node[bsq,bsq2] (12) [right = of 11]  {1};
    \node[bsq,bsq2] (13) [right = of 12]  {1};
    \node[bsq,bsq2] (14) [right = of 13]  {1};
    \node[bsq,bsq2] (15) [right = of 14]  {1};
    \node[bsq,bsq2] (16) [right = of 15]  {2};
    \node[bsq,bsq2] (17) [right = of 16]  {2};
    \node[bsq,bsq2] (18) [right = of 17]  {2};
    \node[bsq,bsq2] (19) [right = of 18]  {2};
    \node[bsq,bsq2] (110) [right = of 19]  {2};
    \node[bsq,bsq2] (111) [right = of 110]  {3};
    \node[bsq,bsq2] (112) [right = of 111]  {4};

    \node[bsq,bsq2] (21) at (.3, .9)      {2};
    \node[bsq,bsq2] (22) [right = of 21]  {2};
    \node[bsq,bsq2] (23) [right = of 22]  {2};
    \node[bsq,bsq2] (24) [right = of 23]  {5};
    \node[bsq,bsq2] (25) [right = of 24]  {6};
    \node[bsq,bsq2] (26) [right = of 25]  {3};
    \node[bsq,bsq2] (27) [right = of 26]  {3};
    \node[bsq,bsq2] (28) [right = of 27]  {4};
    \node[bsq,bsq2] (29) [right = of 28]  {5};
    \node[bsq,bsq2] (210) [right = of 29]  {5};
    \node[bsq,bsq2] (211) [right = of 210]  {4};
    \node[bsq,bsq2] (212) [right = of 211]  {5};

    \node[bsq,bsq2] (31) at (.3,.3)       {3};
    \node[bsq,bsq2] (32) [right = of 31]  {5};
    \node[bsq,bsq2] (33) [right = of 32]  {7};
    \node[bsq,bsq2] (34) [right = of 33]  {7};
    \node[bsq,bsq2] (35) [right = of 34]  {7};
    \node[bsq,bsq2] (36) [right = of 35]  {4};
    \node[bsq,bsq2] (37) [right = of 36]  {5};
    \node[bsq,bsq2] (38) [right = of 37]  {5};
    \node[bsq,bsq2] (39) [right = of 38]  {6};
    \node[bsq,bsq2] (310) [right = of 39]  {7};
    \node[bsq,bsq2] (311) [right = of 310]  {5};
    \node[bsq,bsq2] (312) [right = of 311]  {6};

    \node[cor] at (11.north west) {\tiny 1};
    \node[cor] at (12.north west) {\tiny 2};
    \node[cor] at (13.north west) {\tiny 3};
    \node[cor] at (14.north west) {\tiny 4};
    \node[cor] at (15.north west) {\tiny 5};
    \node[cor] at (16.north west) {\tiny 6};
    \node[cor] at (17.north west) {\tiny 7};
    \node[cor] at (18.north west) {\tiny 8};
    \node[cor] at (19.north west) {\tiny 9};
    \node[cor] at (110.north west) {\tiny 1\!0};
    \node[cor] at (111.north west) {\tiny 1\!1};
    \node[cor] at (112.north west) {\tiny 1\!2};

    \node[cor] at (21.north west) {\tiny 1};
    \node[cor] at (22.north west) {\tiny 2};
    \node[cor] at (23.north west) {\tiny 3};
    \node[cor] at (24.north west) {\tiny 1\!1};
    \node[cor] at (25.north west) {\tiny 1\!2};
    \node[cor] at (26.north west) {\tiny 4};
    \node[cor] at (27.north west) {\tiny 5};
    \node[cor] at (28.north west) {\tiny 7};
    \node[cor] at (29.north west) {\tiny 9};
    \node[cor] at (210.north west) {\tiny 10};
    \node[cor] at (211.north west) {\tiny 6};
    \node[cor] at (212.north west) {\tiny 8};

    \node[cor] at (31.north west) {\tiny 1};
    \node[cor] at (32.north west) {\tiny 6};
    \node[cor] at (33.north west) {\tiny 1\!2};
    \node[cor] at (34.north west) {\tiny 1\!1};
    \node[cor] at (35.north west) {\tiny 9};
    \node[cor] at (36.north west) {\tiny 2};
    \node[cor] at (37.north west) {\tiny 5};
    \node[cor] at (38.north west) {\tiny 4};
    \node[cor] at (39.north west) {\tiny 8};
    \node[cor] at (310.north west) {\tiny 10};
    \node[cor] at (311.north west) {\tiny 3};
    \node[cor] at (312.north west) {\tiny 7};

  %  In case circles are preferred to bars...
    \draw[ thin] (2.7, 1.5) circle(0.20);
    \draw[ thin] (5.7, 1.5) circle(0.20);
    \draw[ thin] (6.3, 1.5) circle(0.20); % removed double
    \draw[ thin] (6.9, 1.5) circle(0.20); % removed double

    \draw[ thin] (1.5, 0.9) circle(0.20); 
    \draw[ thin] (2.1, 0.9) circle(0.20);
    \draw[ thin] (3.9, 0.9) circle(0.20);
    \draw[ thin] (4.5, 0.9) circle(0.20);
    
    \draw[ thin] (0.3, 0.3) circle(0.20);
    \draw[ thin] (0.9, 0.3) circle(0.20); % removed double
    \draw[ thin] (3.3, 0.3) circle(0.20);
    \draw[ thin] (5.1, 0.3) circle(0.20);

    \node  (00) [left = of r2]  {$L' = $};

  \end{tikzpicture}
  \end{center}
\end{proof}

\begin{remark}
  \label{rem:pushing}
  Observe that a procedure similar to the pushing procedure of the last marked position can be used to push the first marked position.
  More concretely, we have seen in Remark~\ref{rem:bars} that filling the table top to bottom or bottom to top produces the same marked positions.
  The first marked position in a column of the top to bottom procedure can thus be seen as the last marked position of the same column of the bottom to top procedure.
  Thus, it can be pushed in the analogous way as the last bar is pushed.
\end{remark}

The following corollary is well known for $\NN_{k,n}$. For $\NC_{k,n}$ it follows from~\cite{PPS2010}.

\begin{corollary}
\label{coro:pure}
  The simplicial complexes $\NN_{k,n}$ and $\NC_{k,n}$ are pure of dimension $k(n-k)$.
\end{corollary}
\begin{proof}
  Consider a face~$F$ of $\NN_{k,n}$ not containing the vector~$(n-k+1,\ldots,n)$. This is,~$F$ is a set of mutually nonnesting elements in $V^*_{k,n}$.
  As we have seen in Theorem~\ref{thm:uniquemultiset},~$F$ can be recovered from its summing tableau~$T$, i.e., $F = \bijNN(T)$.
  Thus, we can recover the marked positions in the table as described before.
  Alternatively, one can as well obtain the marked positions using the procedure described in Remark~\ref{rem:bars}.
  Since the number of inserted marked positions equals $k(n-k)$ (assuming without loss of generality that~$T$ is strictly increasing along rows), we have~$|F| \leq k(n-k)$, as otherwise, we would have repeated columns in~$F$ by Lemma~\ref{le:furtherproperties1}\eqref{eq:repetition}.
  Moreover, if $|F| < k(n-k)$ then there is a column containing more than one marked position.
  Thus, Proposition~\ref{prop:furtherproperties} implies that there is a vector $I \in V^*_{k,n}$ that is not contained in~$F$ such that $F \cup \{ I \}$ is again mutually nonnesting and thus a face of $\NN_{k,n}$.
  This implies the corollary for $\NN_{k,n}$.
  The argument for $\NC_{k,n}$ is word by word the same.
\end{proof}

\begin{remark}
  The proof of Corollary~\ref{coro:pure} contains an implicit characterization of the tableaux that correspond to maximal faces of the nonnesting (respectively, noncrossing) complex: 
  they are those for which the nonnesting (respectively, noncrossing) decompositions result in exactly one marked position in each column of the table.
  For the nonnesting complex, in which all rows of the table are filled from left to right, this condition is clearly equivalent to saying that $T$ contains each entry in 
  $[k(n-k)]$ exactly once. That is, $\bijNN$ gives a bijection between Standard Young Tableaux of shape $k\times(n-k)$ with $\max(T)=k(n-k)$ and maximal faces of $\NN_{k,n}$.

  For $\NC_{k,n}$ we do not have a simple combinatorial characterization of the tableaux that arise.
  According to Lemma~\ref{le:furtherproperties1}\eqref{eq:notzero} and Lemma~\ref{le:furtherproperties3}, it is however easy to see that they must be strictly increasing along rows and columns.
\end{remark}

In the approach taken in~\cite{PPS2010} the next two corollaries follow from their Theorem 8.7.

\begin{corollary}
\label{cor:pseudomanifold}
  Every face of the reduced noncrossing complex $\NCred_{k,n}$ of co\-dimen\-sion-$1$ is contained in exactly two maximal faces.
\end{corollary}
\begin{proof}
  Let~$L$ be the table of a maximal face of~$\NC_{k,n}$, let~$b$ be the index of one column of~$L$, and let~$L'$ be the table of the codimension one 
  face of~$\NC_{k,n}$ obtained from~$L$ by deleting a column~$b$ that does not contain a vector that is a cyclic rotation of the vector $(1,\ldots,k)$.
  Observe that it follows from Corollary~\ref{coro:pure} that every pair $L' \subset L$ of a codimension one face contained in a maximal face is obtained this way.
  Now, every marked position in a column of~$L$ different from column~$b$ is as well a marked position in~$L'$.
  Moreover, there is a unique column of~$L'$ that contains a unique second marked position.
  Following the pushing procedure described above, we obtain that pushing this marked position again yields the table~$L$.
  Since the marked positions in~$L'$ are independent of~$L$, and since we can push exactly those two marked positions in the unique column containing them, we conclude that~$L'$ is contained in exactly two maximal faces.
\end{proof}

Corollaries~\ref{coro:pure} and~\ref{cor:pseudomanifold} together can be rephrased as follows.

\begin{corollary}
  \label{cor:pseudomanifold2}
  $\NCred_{k,n}$ is a pseudo-manifold without boundary.
  $\NC_{k,n}$ is a pseudo-manifold with boundary, its boundary consisting of the codimension one faces that do not use all the cyclic intervals.
\end{corollary}

\subsection{The Gra{\ss}mann-Tamari order}
\label{sec:GrassmannTamari}

Based on our approach to the combinatorics and on the geometry (developed in the next section)  of the noncrossing complex~$\NC_{k,n}$, we consider a natural generalization 
of the \defn{Tamari order} on triangulations of a convex polygon. To this end let the \defn{dual graph}~$G(\Delta)$ of a pure simplicial complex~$\Delta$ be the graph 
whose vertices are the maximal faces of~$\Delta$, and where two maximal faces~$F_1$ and~$F_2$ share an edge if they intersect in a face of codimension one.

\medskip

We start with recalling the definition of the Tamari order.
For further background and many more detailed see e.g.~\cite{Rea2012} and the Tamari Festschrift containing that article.
Fix~$n>2$.
The elements of the Tamari poset $\Tamari_{n}$ are triangulations of a convex $n$-gon and the Hasse diagram of~$\Tamari_{n}$ coincides, as a graph, 
with the dual graph of $\NC_{n} = \NC_{2,n}$. To specify~$\Tamari_{n}$ we thus need to orient its dual graph.
Let $T$ and $T'$ be two triangulations which differ in a single diagonal.
We then have that $T \triangle T' = \{ [i_1,i_2], [j_1,j_2] \}$, and observe that $(i_1,i_2)$ and $(j_1,j_2)$ cross, so we can assume without loss of generality that 
$i_1 < j_1 < i_2 < j_2$. We say that~$T \prec_{\Tamari_n} T'$ form a cover relation in $\Tamari_n$ if $[i_1,i_2] \in T$ and $[j_1,j_2] \in T'$.

\medskip

We next extend this ordering to the dual graph of~$\NC_{k,n}$ for general~$k$.
Let~$F$ be a face of~$\NC_{k,n}$ of codimension one that uses all the cyclic intervals. Equivalently, by 
Corollary~\ref{cor:pseudomanifold2}, let~$F$ be a codimension one face that is not in the boundary of~$\NC_{k,n}$.
Given the procedure we used in the proof of Proposition~\ref{prop:furtherproperties} we have that the table of~$F$ 
contains exactly one column with more than a single marked position.
This column contains exactly two marked positions.
By Corollary~\ref{cor:pseudomanifold},~$F$ is contained in exactly two maximal faces, obtained by \lq\lq pushing\rq\rq\ one of those two marked positions.

\begin{definition}
  \label{def:TamariGraph}
  The \defn{Gra{\ss}mann-Tamari digraph} $\vec G(\NC_{k,n})$ is the orientation on the dual graph~$G(\NC_{k,n})$ given by the following rule.
  Let ~$F_1$ and~$F_2$ be two maximal faces sharing a codimension one face~$F$. We orient the edge~$F_1 - F_2$ from~$F_1$ to~$F_2$ if~$F_1$ is obtained by pushing the lower of the two marked positions in the column of~$F$ that has two marks.
  The \defn{Gra{\ss}mann-Tamari order}~$\Tamari_{k,n}$ is the partial order on the maximal faces of $\NC_{k,n}$ obtained as the transitive closure of $\vec G(\NC_{k,n})$.
  That is, we have $F_1 <_{\Tamari_{k,n}} F_2$ for two maximal faces if there is a directed path from~$F_1$ to~$F_2$ in $\vec G(\NC_{k,n})$.
\end{definition} 

Of course, the Gra{\ss}mann-Tamari order will only be well-defined if the Gra{\ss}mann-Tamari digraph is acyclic (meaning that it does not contain directed cycles).

\begin{theorem}
%\label{prop:GrassmannTamari}
\label{thm:welldefinedTamari}
  The Gra{\ss}mann-Tamari digraph is acyclic, hence the Gra{\ss}\-mann-Tamari order $\Tamari_{k,n}$ is well-defined.
  Moreover, $\Tamari_{k,n}$ has a linear extension that is a shelling order of~$\NC_{k,n}$.
\end{theorem}

Our proof of this theorem relies on the geometry of the noncrossing complex.
It is thus postponed to Section~\ref{sec:shelling}.

\begin{example}
  \label{ex:tamari}
  Figure~\ref{fig:tamari} shows the Gra{\ss}mann-Tamari posets for~$n=5, k\in\{2,3\}$.
\end{example}

\begin{figure}
  \centering
  \begin{tikzpicture}[node distance=0pt,outer sep = 0pt,inner sep=0pt,scale=.9]
    \tikzstyle{bsq}=[rectangle,draw=black,opacity=.25,fill opacity=1,minimum width=.4cm, minimum height=.4cm]

    \node[bsq] (11) at (0,0)         {1};
    \node[bsq] (12) [right = of 11]  {\bf 1};
    \node[bsq] (13) [right = of 12]  {\bf 1};
    \node[bsq] (14) [right = of 13]  {1};
    \node[bsq] (15) [right = of 14]  {2};
    \node[bsq] (16) [right = of 15]  {3};
    \node[bsq] (17) [right = of 16]  {4};

    \node[bsq] (21) [below = of 11]  {2};
    \node[bsq] (22) [right = of 21]  {\bf 3};
    \node[bsq] (23) [right = of 22]  {\bf 4};
    \node[bsq] (24) [right = of 23]  {5};
    \node[bsq] (25) [right = of 24]  {3};
    \node[bsq] (26) [right = of 25]  {4};
    \node[bsq] (27) [right = of 26]  {5};

    \draw[ thin] (14) circle(0.20);
    \draw[ thin] (15) circle(0.20);
    \draw[ thin] (16) circle(0.20);
%     \draw[ thin] (17) circle(0.20);

    \draw[ thin] (21) circle(0.20); 
    \draw[ thin] (22) circle(0.20);
    \draw[ thin] (23) circle(0.20);

    \node[bsq] (11) at (1.5,3)       {1};
    \node[bsq] (12) [right = of 11]  {\bf 1};
    \node[bsq] (13) [right = of 12]  {1};
    \node[bsq] (14) [right = of 13]  {2};
    \node[bsq] (15) [right = of 14]  {3};
    \node[bsq] (16) [right = of 15]  {\bf 3};
    \node[bsq] (17) [right = of 16]  {4};

    \node[bsq] (21) [below = of 11]  {2};
    \node[bsq] (22) [right = of 21]  {\bf 3};
    \node[bsq] (23) [right = of 22]  {5};
    \node[bsq] (24) [right = of 23]  {3};
    \node[bsq] (25) [right = of 24]  {4};
    \node[bsq] (26) [right = of 25]  {\bf 5};
    \node[bsq] (27) [right = of 26]  {5};

    \draw[ thin] (13) circle(0.20);
    \draw[ thin] (14) circle(0.20);
    \draw[ thin] (16) circle(0.20);
%     \draw[ thin] (17) circle(0.20);

    \draw[ thin] (21) circle(0.20); 
    \draw[ thin] (22) circle(0.20);
    \draw[ thin] (25) circle(0.20);

    \node[bsq] (11) at (-2,2)        {1};
    \node[bsq] (12) [right = of 11]  {\bf 1};
    \node[bsq] (13) [right = of 12]  {1};
    \node[bsq] (14) [right = of 13]  {2};
    \node[bsq] (15) [right = of 14]  {\bf 2};
    \node[bsq] (16) [right = of 15]  {3};
    \node[bsq] (17) [right = of 16]  {4};

    \node[bsq] (21) [below = of 11]  {2};
    \node[bsq] (22) [right = of 21]  {\bf 4};
    \node[bsq] (23) [right = of 22]  {5};
    \node[bsq] (24) [right = of 23]  {3};
    \node[bsq] (25) [right = of 24]  {\bf 4};
    \node[bsq] (26) [right = of 25]  {4};
    \node[bsq] (27) [right = of 26]  {5};

    \draw[ thin] (13) circle(0.20);
    \draw[ thin] (15) circle(0.20);
    \draw[ thin] (16) circle(0.20);
%     \draw[ thin] (17) circle(0.20);

    \draw[ thin] (21) circle(0.20); 
    \draw[ thin] (22) circle(0.20);
    \draw[ thin] (24) circle(0.20);

    \node[bsq] (11) at (-2,4)        {1};
    \node[bsq] (12) [right = of 11]  {1};
    \node[bsq] (13) [right = of 12]  {2};
    \node[bsq] (14) [right = of 13]  {\bf 2};
    \node[bsq] (15) [right = of 14]  {\bf 2};
    \node[bsq] (16) [right = of 15]  {3};
    \node[bsq] (17) [right = of 16]  {4};

    \node[bsq] (21) [below = of 11]  {2};
    \node[bsq] (22) [right = of 21]  {5};
    \node[bsq] (23) [right = of 22]  {3};
    \node[bsq] (24) [right = of 23]  {\bf 4};
    \node[bsq] (25) [right = of 24]  {\bf 5};
    \node[bsq] (26) [right = of 25]  {4};
    \node[bsq] (27) [right = of 26]  {5};

    \draw[ thin] (12) circle(0.20);
    \draw[ thin] (15) circle(0.20);
    \draw[ thin] (16) circle(0.20);
%     \draw[ thin] (17) circle(0.20);

    \draw[ thin] (21) circle(0.20); 
    \draw[ thin] (23) circle(0.20);
    \draw[ thin] (24) circle(0.20);

    \node[bsq] (11) at (0,6)         {1};
    \node[bsq] (12) [right = of 11]  {1};
    \node[bsq] (13) [right = of 12]  {2};
    \node[bsq] (14) [right = of 13]  {\bf 2};
    \node[bsq] (15) [right = of 14]  {3};
    \node[bsq] (16) [right = of 15]  {\bf 3};
    \node[bsq] (17) [right = of 16]  {4};

    \node[bsq] (21) [below = of 11]  {2};
    \node[bsq] (22) [right = of 21]  {5};
    \node[bsq] (23) [right = of 22]  {3};
    \node[bsq] (24) [right = of 23]  {\bf 5};
    \node[bsq] (25) [right = of 24]  {4};
    \node[bsq] (26) [right = of 25]  {\bf 5};
    \node[bsq] (27) [right = of 26]  {5};

    \draw[ thin] (12) circle(0.20);
    \draw[ thin] (14) circle(0.20);
    \draw[ thin] (16) circle(0.20);
%     \draw[ thin] (17) circle(0.20);

    \draw[ thin] (21) circle(0.20); 
    \draw[ thin] (23) circle(0.20);
    \draw[ thin] (25) circle(0.20);

    \draw[line width=1pt,black,->] {(0.5,0.25)--(-1,1.25)};
    \draw[line width=1pt,black,->] {(1.5,0.25)--(3,2.25)};
    \draw[line width=1pt,black,->] {(-1.25,2.35)--(-1.25,3.35)};
    \draw[line width=1pt,black,->] {(-1,4.25)--(0.5,5.25)};
    \draw[line width=1pt,black,->] {(3,3.35)--(1.5,5.25)};
    
  \end{tikzpicture}
  \quad
  \begin{tikzpicture}[node distance=0pt,outer sep = 0pt,inner sep=0pt,scale=.9]
      \tikzstyle{bsq}=[rectangle,draw=black,opacity=.25,fill opacity=1,minimum width=.4cm, minimum height=.4cm]

    \node[bsq] (11) at (0,0)         {1};
    \node[bsq] (12) [right = of 11]  {\bf 1};
    \node[bsq] (13) [right = of 12]  {1};
    \node[bsq] (14) [right = of 13]  {\bf 1};
    \node[bsq] (15) [right = of 14]  {1};
    \node[bsq] (16) [right = of 15]  {2};
    \node[bsq] (17) [right = of 16]  {3};

    \node[bsq] (21) [below = of 11]  {2};
    \node[bsq] (22) [right = of 21]  {\bf 2};
    \node[bsq] (23) [right = of 22]  {2};
    \node[bsq] (24) [right = of 23]  {\bf 3};
    \node[bsq] (25) [right = of 24]  {4};
    \node[bsq] (26) [right = of 25]  {3};
    \node[bsq] (27) [right = of 26]  {4};

    \node[bsq] (31) [below = of 21]  {3};
    \node[bsq] (32) [right = of 31]  {\bf 4};
    \node[bsq] (33) [right = of 32]  {5};
    \node[bsq] (34) [right = of 33]  {\bf 4};
    \node[bsq] (35) [right = of 34]  {5};
    \node[bsq] (36) [right = of 35]  {4};
    \node[bsq] (37) [right = of 36]  {5};

    \draw[ thin] (15) circle(0.20);
    \draw[ thin] (16) circle(0.20);
%     \draw[ thin] (17) circle(0.20);

    \draw[ thin] (23) circle(0.20); 
    \draw[ thin] (24) circle(0.20);

    \draw[ thin] (31) circle(0.20); 
    \draw[ thin] (32) circle(0.20);

    \node[bsq] (11) at (-2,2)        {1};
    \node[bsq] (12) [right = of 11]  {1};
    \node[bsq] (13) [right = of 12]  {\bf 1};
    \node[bsq] (14) [right = of 13]  {\bf 1};
    \node[bsq] (15) [right = of 14]  {1};
    \node[bsq] (16) [right = of 15]  {2};
    \node[bsq] (17) [right = of 16]  {3};

    \node[bsq] (21) [below = of 11]  {2};
    \node[bsq] (22) [right = of 21]  {2};
    \node[bsq] (23) [right = of 22]  {\bf 3};
    \node[bsq] (24) [right = of 23]  {\bf 3};
    \node[bsq] (25) [right = of 24]  {4};
    \node[bsq] (26) [right = of 25]  {3};
    \node[bsq] (27) [right = of 26]  {4};

    \node[bsq] (31) [below = of 21]  {3};
    \node[bsq] (32) [right = of 31]  {5};
    \node[bsq] (33) [right = of 32]  {\bf 4};
    \node[bsq] (34) [right = of 33]  {\bf 5};
    \node[bsq] (35) [right = of 34]  {5};
    \node[bsq] (36) [right = of 35]  {4};
    \node[bsq] (37) [right = of 36]  {5};

    \draw[ thin] (15) circle(0.20);
    \draw[ thin] (16) circle(0.20);
%     \draw[ thin] (17) circle(0.20);

    \draw[ thin] (22) circle(0.20); 
    \draw[ thin] (24) circle(0.20);

    \draw[ thin] (31) circle(0.20); 
    \draw[ thin] (33) circle(0.20);

    \node[bsq] (11) at (-2,4)        {1};
    \node[bsq] (12) [right = of 11]  {1};
    \node[bsq] (13) [right = of 12]  {\bf 1};
    \node[bsq] (14) [right = of 13]  {1};
    \node[bsq] (15) [right = of 14]  {2};
    \node[bsq] (16) [right = of 15]  {\bf 2};
    \node[bsq] (17) [right = of 16]  {3};

    \node[bsq] (21) [below = of 11]  {2};
    \node[bsq] (22) [right = of 21]  {2};
    \node[bsq] (23) [right = of 22]  {\bf 3};
    \node[bsq] (24) [right = of 23]  {4};
    \node[bsq] (25) [right = of 24]  {3};
    \node[bsq] (26) [right = of 25]  {\bf 3};
    \node[bsq] (27) [right = of 26]  {4};

    \node[bsq] (31) [below = of 21]  {3};
    \node[bsq] (32) [right = of 31]  {5};
    \node[bsq] (33) [right = of 32]  {\bf 5};
    \node[bsq] (34) [right = of 33]  {5};
    \node[bsq] (35) [right = of 34]  {4};
    \node[bsq] (36) [right = of 35]  {\bf 5};
    \node[bsq] (37) [right = of 36]  {5};

    \draw[ thin] (14) circle(0.20);
    \draw[ thin] (16) circle(0.20);
%     \draw[ thin] (17) circle(0.20);

    \draw[ thin] (22) circle(0.20); 
    \draw[ thin] (23) circle(0.20);

    \draw[ thin] (31) circle(0.20); 
    \draw[ thin] (35) circle(0.20);

    \node[bsq] (11) at (0,6)         {1};
    \node[bsq] (12) [right = of 11]  {1};
    \node[bsq] (13) [right = of 12]  {1};
    \node[bsq] (14) [right = of 13]  {2};
    \node[bsq] (15) [right = of 14]  {\bf 2};
    \node[bsq] (16) [right = of 15]  {\bf 2};
    \node[bsq] (17) [right = of 16]  {3};

    \node[bsq] (21) [below = of 11]  {2};
    \node[bsq] (22) [right = of 21]  {2};
    \node[bsq] (23) [right = of 22]  {4};
    \node[bsq] (24) [right = of 23]  {3};
    \node[bsq] (25) [right = of 24]  {\bf 3};
    \node[bsq] (26) [right = of 25]  {\bf 4};
    \node[bsq] (27) [right = of 26]  {4};

    \node[bsq] (31) [below = of 21]  {3};
    \node[bsq] (32) [right = of 31]  {5};
    \node[bsq] (33) [right = of 32]  {5};
    \node[bsq] (34) [right = of 33]  {4};
    \node[bsq] (35) [right = of 34]  {\bf 5};
    \node[bsq] (36) [right = of 35]  {\bf 5};
    \node[bsq] (37) [right = of 36]  {5};

    \draw[ thin] (13) circle(0.20);
    \draw[ thin] (16) circle(0.20);
%     \draw[ thin] (17) circle(0.20);

    \draw[ thin] (22) circle(0.20); 
    \draw[ thin] (25) circle(0.20);

    \draw[ thin] (31) circle(0.20); 
    \draw[ thin] (34) circle(0.20);

    \node[bsq] (11) at (1.5,3)       {1};
    \node[bsq] (12) [right = of 11]  {\bf 1};
    \node[bsq] (13) [right = of 12]  {1};
    \node[bsq] (14) [right = of 13]  {1};
    \node[bsq] (15) [right = of 14]  {2};
    \node[bsq] (16) [right = of 15]  {\bf 2};
    \node[bsq] (17) [right = of 16]  {3};

    \node[bsq] (21) [below = of 11]  {2};
    \node[bsq] (22) [right = of 21]  {\bf 2};
    \node[bsq] (23) [right = of 22]  {2};
    \node[bsq] (24) [right = of 23]  {4};
    \node[bsq] (25) [right = of 24]  {3};
    \node[bsq] (26) [right = of 25]  {\bf 4};
    \node[bsq] (27) [right = of 26]  {4};

    \node[bsq] (31) [below = of 21]  {3};
    \node[bsq] (32) [right = of 31]  {\bf 4};
    \node[bsq] (33) [right = of 32]  {5};
    \node[bsq] (34) [right = of 33]  {5};
    \node[bsq] (35) [right = of 34]  {4};
    \node[bsq] (36) [right = of 35]  {\bf 5};
    \node[bsq] (37) [right = of 36]  {5};

    \draw[ thin] (14) circle(0.20);
    \draw[ thin] (16) circle(0.20);
%     \draw[ thin] (17) circle(0.20);

    \draw[ thin] (23) circle(0.20); 
    \draw[ thin] (25) circle(0.20);

    \draw[ thin] (31) circle(0.20); 
    \draw[ thin] (32) circle(0.20);

    \draw[line width=1pt,black,->] {(0.5,0.25)--(-1,0.85)};
    \draw[line width=1pt,black,->] {(1.5,0.25)--(3,1.85)};
    \draw[line width=1pt,black,->] {(-1.25,2.35)--(-1.25,2.95)};
    \draw[line width=1pt,black,->] {(-1,4.25)--(0.5,4.85)};
    \draw[line width=1pt,black,->] {(3,3.35)--(1.5,4.85)};
    
  \end{tikzpicture}

  \caption{The Gra{\ss}mann-Tamari posets for $\NC_{5,2}$ and for $\NC_{5,3}$.}
  \label{fig:tamari}
\end{figure}
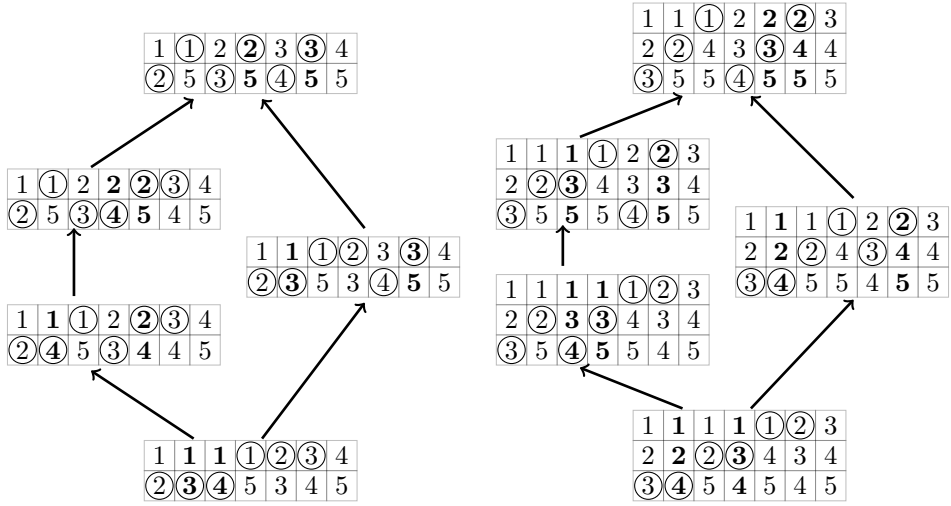

\begin{proposition}
  \label{prop:tamariproperties}
  The Gra{\ss}mann-Tamari order has the following properties:
  \begin{enumerate}
    \item The map induced on $\Tamari_{k,n}$ by $a \mapsto n+1-a$ is order reversing .
    In particular, $\Tamari_{k,n}$ is self-dual. \label{eq:tamari1}
    \item  The map from $\Tamari_{k,n}$ to $\Tamari_{n-k,n}$ induced by $I \mapsto [n] \setminus I$ is order reversing.
    In particular, $\Tamari_{k,n} \cong \Tamari_{n-k,n}$. \label{eq:tamari2}
  \end{enumerate}
\end{proposition}
\begin{proof}
  Both parts of~\eqref{eq:tamari1} follow from the discussion in Remarks~\ref{rem:topbottom},~\ref{rem:bars}, and~\ref{rem:pushing}.

  For~\eqref{eq:tamari2}, let~$F_1$ and~$F_2$ be two maximal faces of~$\NC_{k,n}$ sharing a face~$F=F_1 \cap F_2$ of codimension one, and let 
  $G_i = [n] \setminus F_i$, $i=1,2$, and $G = G_1 \cap G_2$ be the two complementary maximal faces in ~$\NC_{n-k,n}$ and their intersection.
  Proposition~\ref{prop:intro}\eqref{eq:complementarity} implies that~$G$ is a face of codimension one as well.
  It is thus left to show that that if~$F_1$ is obtained from~$F$ by pushing the lower of the two marked 
  positions in the appropriate column of the table for~$F$, then~$G_1$ is  obtained from~$G$ by pushing the higher of the two marked positions again in the appropriate column of~$G$.
  To prove this, recall first that any column in any maximal face of $\NC_{k,n}$ and of $\NC_{n-k,n}$ (except the last) contains 
  a unique marked position, and this column and the column to its right coincide above this marked position by Lemma~\ref{le:furtherproperties1}\eqref{eq:firstdiff}.
  This implies that the values of the two marked positions in the table for~$F$ and in the table for~$G$ coincide.
  Since one of these two values is also the value of the marked position in~$F_1$, the other value must be the value of the marked position in~$G_1$ (since $G_1 = [n] \setminus F_1$).
  This finally yields that if pushing the lower entry in~$F$ yields~$F_1$, then pushing the higher entry in~$G$ yields~$G_1$, as desired.
  As an example, consider the cover relation on the left in the two Gra{\ss}mann-Tamari posets in Figure~\ref{fig:tamari}, and call the maximal 
  faces~$F_1 \prec_{\Tamari_{2,5}} F_2$ in the left poset, and~$G_2  \prec_{\Tamari_{3,5}} G_1$ in the right poset.
  Then~$F = F_1 \cap F_2$ and of $G = G_1 \cap G_2$ are given by

  \begin{figure}[h]
    \centering
    \begin{tikzpicture}[node distance=0pt,outer sep = 0pt,inner sep=0pt,scale=.9]
      \tikzstyle{bsq}=[rectangle,draw=black,opacity=.25,fill opacity=1,minimum width=.4cm, minimum height=.4cm]
  
    \node[bsq] (12) at (0,0)        {1};
    \node[bsq] (13) [right = of 12]  {1};
    \node[bsq] (14) [right = of 13]  {2};
    \node[bsq] (15) [right = of 14]  {\bf 2};
    \node[bsq] (16) [right = of 15]  {3};
    \node[bsq] (17) [right = of 16]  {4};

    \node[bsq] (22) [below = of 12]  {2};
    \node[bsq] (23) [right = of 22]  {5};
    \node[bsq] (24) [right = of 23]  {3};
    \node[bsq] (25) [right = of 24]  {\bf 4};
    \node[bsq] (26) [right = of 25]  {4};
    \node[bsq] (27) [right = of 26]  {5};

    \draw[ thin] (13) circle(0.20);
    \draw[ thin] (15) circle(0.20);
    \draw[ thin] (16) circle(0.20);
%     \draw[ thin] (17) circle(0.20);

    \draw[ thin] (25) circle(0.20); 
    \draw[ thin] (22) circle(0.20);
    \draw[ thin] (24) circle(0.20);

    \node  (00) at (-.8,-.25)  {$F = $};
    \node (asdf) at (0,-.9) {};
  \end{tikzpicture}
  \qquad
    \begin{tikzpicture}[node distance=0pt,outer sep = 0pt,inner sep=0pt,scale=.9]
      \tikzstyle{bsq}=[rectangle,draw=black,opacity=.25,fill opacity=1,minimum width=.4cm, minimum height=.4cm]
  
    \node[bsq] (11) at (0,0)        {1};
    \node[bsq] (12) [right = of 11]  {1};
    \node[bsq] (13) [right = of 12]  {\bf 1};
    \node[bsq] (14) [right = of 13]  {1};
    \node[bsq] (15) [right = of 14]  {2};
    \node[bsq] (17) [right = of 15]  {3};

    \node[bsq] (21) [below = of 11]  {2};
    \node[bsq] (22) [right = of 21]  {2};
    \node[bsq] (23) [right = of 22]  {\bf 3};
    \node[bsq] (24) [right = of 23]  {4};
    \node[bsq] (25) [right = of 24]  {3};
    \node[bsq] (27) [right = of 25]  {4};

    \node[bsq] (31) [below = of 21]  {3};
    \node[bsq] (32) [right = of 31]  {5};
    \node[bsq] (33) [right = of 32]  {\bf 5};
    \node[bsq] (34) [right = of 33]  {5};
    \node[bsq] (35) [right = of 34]  {4};
    \node[bsq] (37) [right = of 35]  {5};

    \draw[ thin] (14) circle(0.20);
    \draw[ thin] (15) circle(0.20);
%     \draw[ thin] (17) circle(0.20);

    \draw[ thin] (22) circle(0.20); 
    \draw[ thin] (23) circle(0.20);

    \draw[ thin] (31) circle(0.20); 
    \draw[ thin] (35) circle(0.20);
    \node  (00) at (-.8,-.45)  {$G = $};

  \end{tikzpicture}
  \end{figure}
  \noindent and the values of the two marked positions in unique columns of~$F$ and~$G$ containing two marked positions are~$2$ and~$4$.
  Pushing the~$4$ in~$F$ yields~$F_1$, and pushing the~$2$ in~$G$ yields~$G_1$, as desired.
\end{proof}

It is straightforward to see that the Gra{\ss}mann-Tamari order restricts to the usual Tamari order for $k=2$.
The latter is well known to be a selfdual lattice.
We conjecture this as well for the Gra{\ss}mann-Tamari order for general~$k$, tested for~$n \in \{ 6,7,8\}$ and~$k=3$.

\begin{conjecture}
  The Gra{\ss}mann-Tamari poset~$\Tamari_{k,n}$ is a lattice.
\end{conjecture}

\begin{remark}
  It would be interesting to extend other properties of the (dual) associahedron or the Tamari lattice to $\NC_{k,n}$.
  An example of such a property concerns the diameter. The diameter of the dual associahedron $\NC_{2,n}$ is known to be 
  bounded by $2n-10$ for every $n$ (D.~D.~Sleator, R.~E.¸Tarjan, and W.~P.~Thurston,~\cite{STT1988}) and equal to 
  $2n-10$ for $n>12$ (L.~Pournin,~\cite{Pou2012}). 

  Although the proof of the $2n-10$ upper bound needs the use of the cyclic symmetry of $\NC_{2,n}$ (and, as we have seen 
  in Remark~\ref{rem:cyclicsymmetry}, this symmetry does not carry over to higher $k$), an almost tight bound of $2n-6$ can be 
  derived from the very simple fact that every maximal face (triangulation of the $n$-gon) is at distance at most 
  $n-3 = (k-1)(n-k-1)$ from the minimal element in $\Tamari_n$.
  Unfortunately, the similar statement does not hold for $\Tamari_{k,n}$:
  For $k=n-k=3$, the above formula would predict that every maximal face can be flipped to the unique minimal element in $\Tamari_{k,n}$ in~$4$ 
  steps, but there are elements that need~$5$ such flips.
  Moreover, there is no maximal face in $\NC_{3,6}$ that is connected to every other maximal face by~$4$ or less flips.
\end{remark}

\medskip

Observe that Theorem~\ref{thm:welldefinedTamari} implies that the $h$-vector of the noncrossing complex~$\NC_{k,n}$ is the 
generating function of out-degrees in the Gra{\ss}mann-Tamari digraph.
In particular, 
\[h_i^{(k,n)} = \big| \{ T \in \NC_{k,n} \ :\ T \text{ has } i \text{ upper covers }\}\big|,\]
are the multidimensional Narayana numbers, and also equal to the number of standard Young tableaux with exactly~$i$ peaks;
see Equation~\eqref{eq:zeroesattheend} in the Introduction and the preceding discussion.
This raises the following problem.

\begin{openproblem}
  Is there an operation on peaks of standard Young tableaux (or, equivalently, on valleys of multidimensional Dyck paths) that
  \begin{itemize}
    \item describes the Gra{\ss}mann-Tamari order directly on standard Young tableaux (or on multidimensional Dyck paths), 
       thus also providing a bijection between maximal faces of $\NN_{k,n}$ and of $\NC_{k,n}$, and
    \item which generalizes the Tamari order as defined on ordinary Dyck paths?
  \end{itemize}
\end{openproblem}

\section{Order polytopes and their triangulations}
\label{sec:background}

\subsection{Cubical faces in order polytopes}
\label{sub:cubicalfaces}

  Let~$E$ and~$F$ be two order filters in a poset~$P$, and let~$\fF(E,F)$ denote the minimal face of the order polytope~$\oO(P)$ containing the corresponding vertices~$\chi_E$ and~$\chi_F$. 
  It turns out that~$\fF(E,F)$ is always (affinely equivalent to) a cube. 
  Although this is not difficult to prove, it was new to us and is useful in some parts of this paper.

  \medskip

  To prove it, we start with the case when~$E$ and~$F$ are comparable filters. 
  In the next statement we use the notation $\vec X$ for the segment going from the origin to $X\in \RR^{P}$.
  \begin{lemma}
    \label{lem:cubes}
    Let  $E\subset F$ be two comparable filters in a finite poset~$P$ and let $P'_1,\ldots,P'_d$ be the connected components of $P|_{F\setminus E}$.
    Then the minimal face~$\fF(E,F)$ of the order polytope~$\oO(P)$ containing the vertices~$\chi_E$ and~$\chi_F$ is the Minkowski sum
    $$
      \chi_E + \vec{\chi}_{P'_1}+ \cdots+ \vec\chi_{P'_d}.
    $$
    In particular, combinatorially $\fF(E,F)$ is a cube of dimension~$d$.
  \end{lemma}

  \begin{proof}
    Observe that $\fF(E,F)$ is contained in the face of $\oO(P)$ obtained by setting the coordinates of elements in $E\cap F=E$ to be $1$ 
    and those of elements not in $E\cup F=F$ to be $0$. That face is just the order polytope of $P|_{F\setminus E}$ (translated by the vector $\chi_E$).
    For the rest of the proof there is thus no loss of generality in assuming that $E=\emptyset$ and $F=P$.
    We claim that the set of vertices of $\fF(\emptyset,P)$ is
    \[
      \{x = (x_a)_{a \in P} \in \{0,1\}^P : x_a=x_b \text{ if $a$ and $b$ are in the same component of $P$}\}.
    \]
    To see that every vertex of $\fF(\emptyset,P)$ must have this form, observe that the equality $x_a=x_b$ for two comparable elements defines a face of the order polytope and it is satisfied both by~$\chi_E$ and~$\chi_F$, so it is satisfied in all of~$\fF(\emptyset,P)$. 
    Moreover, since connected components are the transitive closure of covering relations, the equality $x_a=x_b$ for two  elements of the same connected component is also satisfied in~$\fF(\emptyset,P)$.
%    This proves that every vertex of~$\fF(\emptyset,P)$ is of the described form.
    For the converse, let $x\in \{0,1\}^P$ be such that $x_a=x_b$ when $a$ and $b$ are in the same component of~$P$. 
    Put differently,~$x$ is the sum of the characteristic vectors of some subset $S$ of the components,
    \[
      x=\sum_{b\in S} \chi_{P'_b}, \text{for some $S\subset[d]$}.
    \]
    Clearly,~$x$ is a vertex of~$\oO(P)$, since a union of connected components is a filter.
    Consider the complementary vertex
    \[
      y=\sum_{b\not\in S} \chi_{P'_b}.
    \]
    We have $x+y =\chi_\emptyset+\chi_P$, which implies that the minimal face containing~$\chi_\emptyset$ and~$\chi_P$ contains also~$x$ and~$y$ (and vice versa).

    The description of the vertices of $\fF(\emptyset,P)$ automatically implies the Minkowski sum expression for $\fF(\emptyset,P)$. 
    Moreover, since the different segments $\vec\chi_{P'_1},\ldots,\vec\chi_{P'_d}$ have disjoint supports, their Minkowski sum is a Cartesian product.
  \end{proof}

  The last part of the proof has the following implications.
%   , where $E\triangle F = (E \cup F)\setminus(E \cap F)$ denotes the symmetric difference of the two filters~$E$ and~$F$.

%   \begin{corollary}
%     $\fF(E,F)=\fF(E\cap F, E\cup F)$ for every~$E$ and~$F$. In particular, it is combinatorially a cube of dimension equal to the number of connected 
%     components of $P|_{E\triangle F}$.
%   \end{corollary}

%   \begin{proof}
%     Since $\chi_E+\chi_F = \chi_{E\cap F} + \chi_{E\cup F}$, we have that $\fF(E,F)=\fF(E\cap F, E\cup F)$. Since $E\triangle F = E\cup F\setminus E\cap F$, 
%     the description of $\fF(E,F)$ in the statement follows from the comparable case given in the previous lemma.
%   \end{proof}
% 
  \begin{corollary}
    \label{coro:equivalent_diagonals}
    Let $E$, $F$, $E'$ and $F'$ be four filters. Then the following properties are equivalent:
    \begin{enumerate}
      \item $\fF(E,F)=\fF(E',F')$.
      \item $(\chi_{E'},\chi_{F'})$ is a pair of opposite vertices of the cube $\fF(E,F)$.
      \item $\chi_E+\chi_F=\chi_{E'}+\chi_{F'}$.
      \item $E\cap F=E'\cap F'$ and $E\cup F=E'\cup F'$.
    \end{enumerate}
    In particular, we have that $\fF(E,F)=\fF(E\cap F, E\cup F)$ which, by the previous lemma, is combinatorially a cube.
  \end{corollary}

\subsection{Triangulations. Unimodularity, regularity, and flagness}
\label{sub:triangulations}

Let $Q$ be a polytope with vertex set $V$.
%\paco{Warning: in this section we were jumping from $P$ to $Q$ and back to denote a polytope. Switched to $Q$ since $P$ is a poset in other sections}
%\christian{I changed $P$ to $Q$ also in a few other places.}
A \defn{triangulation} of $Q$ is a simplicial complex $\Delta$ geometrically realized on $V$ (by which we mean that $V$ is the set of vertices of $\Delta$, and that the vertices of every face of $\Delta$ are affinely independent in $Q$) that covers $Q$ without overlaps.
A triangulation $T$ of a polytope $Q$ is called \defn{regular} if there is a weight vector $w:V \to \RR$ such that $T$ coincides with the lower envelope of the lifted point configuration
$$\big\{\big(v,w(v)\big): v \in V \big\}\subset \RR^{|Q|+1}.$$
See~\cite{DeloeraRambauSantos} for a recent monograph on these concepts.
Another way to express the notions of triangulations and regularity, more suited to our context, is as follows.
\begin{itemize}
  \item An abstract simplicial complex $\Delta$ with its vertices identified with those of $Q$ is a triangulation of $Q$ if and only if for every $x\in \conv(Q)$ there is a unique convex combination of vertices of some face $\sigma$ of $\Delta$ that produces~$x$. That is, there is a unique $\sigma\in \Delta$ (not necessarily full-dimensional) such that
  \[
  x=\sum_{v\in \sigma} \alpha_v v
  \]
  for strictly positive $\alpha_v$ with $\sum_v \alpha_v=1$.

  \item $\Delta$ is the regular triangulation of $Q$ for a weight vector $w:V \to \RR$ if, for every $x\in Q$, the expression of the previous statement is also the unique one that minimizes the weighted sum $\sum_{v\in V} \alpha_v w(v)$, among all the convex combinations giving~$x$ in terms of the vertices of~$Q$.
\end{itemize}

Observe that for some choices of $w$ (most dramatically, when $w$ is constant) the weighted sum of coefficients may not always have a unique minimum. In this case $w$ does not define a regular triangulation of $Q$ but rather a \defn{regular subdivision}. It is still true that the support of every minimizing convex combination is \emph{contained} in some cell of this regular subdivision, but it may perhaps not \emph{equal} the set of vertices of that cell.

If the vertices~$V$ of the polytope $Q$ are contained in $\ZZ^d$ (or, more generally, in a point lattice) we call a full-dimensional simplex \defn{unimodular} when it is an affine lattice basis and we call a triangulation \defn{unimodular} when all its full-dimensional faces are. 
%Unimodular simplices have all the same volume, which implies that all unimodular triangulations of a polytope have the same number of full-dimensional simplices. But even more is true: a
All unimodular triangulations of a lattice polytope have the same $f$-vector and, hence, the same $h$-vector. See, for example,~\cite[Sect. 9.3.3]{DeloeraRambauSantos}. This $h$-vector can be easily computed from the Ehrhart polynomial of $Q$ and is usually called the Ehrhart $h^*$-vector of $Q$. 

If we know a triangulation of a lattice polytope~$Q$ to be unimodular and flag, then checking its regularity is easier than in the general case.
The minimality of weighted sum of coefficients in convex combinations of points $x\in Q$ needs to be checked only for very specific choices of~$x$ and very specific combinations.

\begin{lemma}
\label{lemma:regular-easy}
Let~$Q$ be a lattice polytope with vertices~$V$, let~$\Delta$ be a flag unimodular triangulation of~$Q$, and let $w:V \to \RR$ be a weight vector.
Then the following two statements are equivalent:
\begin{enumerate}[(i)]
  \item The complex~$\Delta$ is the regular triangulation corresponding to~$w$,
  \item For every edge $v_1v_2$ of the complex~$\Delta$ and for every pair of vertices $\{v'_1, v'_2\}\ne \{v_1,v_2\}$ with $v_1+v_2=v'_1+v'_2$, we have $w(v_1)+w(v_2) < w(v'_1)+w(v'_2)$.
\end{enumerate}
\end{lemma}
\begin{proof}
Let $\Delta_w$ be the regular triangulation (or subdivision) produced by the weight vector $w$. Consider $\Delta_w$ as a simplicial complex, even if it turns out not to be a triangulation, taking as maximal faces the vertex sets of the full-dimensional cells (simplices or not) of $\Delta_w$. We are going to show that this simplicial complex is contained in $\Delta$. The containment cannot be strict because $\Delta_w$ covers $Q$, so this will imply $\Delta=\Delta_w$.

Since $\Delta$ is flag, to show that $\Delta_w \subset \Delta$ it suffices to show that every edge of $\Delta_w$ is an edge in $\Delta$. So, let $v'_1v'_2$ be an edge of $\Delta_w$.
For the rest of the proof, we consider our polytope $Q$ embedded in $\RR^d\times \{1\}\subset \RR^{d+1}$. We  regard the vertices $\{v_1,\ldots,v_\ell\}$ of each unimodular simplex in $\Delta$ as vectors spanning a cone $C$, where spanning means not only linearly but also integrally (because of unimodularity): every integral point in $C$ is a nonnegative integral combination of the $v_i$'s.

Consider then the point $v'_1+v'_2$. It lies at height two in one of those cones, because $v'_1+v'_2\in \RR^d\times \{2\}$. Thus, $v'_1+v'_2=v_1+v_2$ for some edge $v_1v_2$ of $\Delta$. The hypothesis in the statement is then that either
$\{v'_1, v'_2\}= \{v_1,v_2\}$ (as we claim) or
$w(v_1)+w(v_2) < w(v'_1)+w(v'_2)$. The latter is impossible because then, letting $x=(v_1+v_2)/2 = (v'_1+v'_2)/2$ be the midpoint of the edge $v'_1v'_2$ of $\Delta_w$, the inequality contradicts the fact that $\Delta_w$ is the regular subdivision for $w$.
\end{proof}

\subsection{Central orientations of the dual graph and line shellings}
\label{sec:central_orientations}

Every regular triangulation of a polytope is shellable.
We sketch here a proof, adapted from~\cite[Section~9.5]{DeloeraRambauSantos}.
The ideas in it will  be used in Section~\ref{sec:shelling} for the proof of Theorem~\ref{thm:welldefinedTamari}.

Let~$\Delta$ be a triangulation (regular or not) of a polytope~$Q$, and let~$o$ be a point in the interior of~$Q$.
We moreover assume~$o \in Q$ to be sufficiently generic so that no hyperplane spanned by a codimension one face of~$\Delta$ contains~$o$.
We can then orient the dual graph~$G(\Delta)$ ``away from $o$'', in the following well-defined sense.
Let~$\sigma_1$ and~$\sigma_2$ be two adjacent maximal faces in~$\Delta$ and let~$H$ be the hyperplane containing their common codimension one face.
We orient the edge~$\sigma_1\sigma_2$ of~$G(\Delta)$ from~$\sigma_1$ to~$\sigma_2$  if~$o$ lies on the same side of~$H$ as~$\sigma_1$. 
We call the digraph obtained this way the \defn{central orientation from~$o$} of~$G(\Delta)$ and denote it $\vec G (\Delta,o)$.

\begin{lemma}
\label{lemma:central-shelling}
  If $\Delta$ is regular then $\vec G (\Delta,o)$ is acyclic for every (generic)~$o$. 
  Moreover, the directions in $\vec G (\Delta,o)$ are induced by a shelling order in the maximal faces of~$\Delta$.
\end{lemma}

\begin{proof}
  We are going to prove directly that there is a shelling order in~$\Delta$ that induces the orientations $\vec G (\Delta,o)$. 
  This implies $\vec G (\Delta,o)$ to be acyclic.

  The idea is the concept of a \emph{line shelling}.
  One way to shell the boundary complex of a simplicial polytope~$Q$ is to consider a (generic) line~$\ell$ going through the interior of~$Q$ and taking the maximal faces of~$Q$ in the order that the facet defining hyperplanes intersect~$\ell$. 
  The line~$\ell$ is considered to be closed (its two ends at infinity are glued together) and the maximal faces are numbered starting and ending with the two maximal faces intersecting~$\ell$.
  Put differently, we can think of the process as moving a point $p=p(t)$, $t\in [0,1]$ along the line~$\ell$, starting in the interior of~$Q$, going through infinity, then back to the interior of~$Q$, and recording the facets of~$Q$ in the order their facet-defining hyperplanes are crossed by $p(t)$.
  See Figure~\ref{fig:shelling} (left) for an illustration and~\cite[Ch.~8]{Zie1994} for more details.

  \begin{figure}
    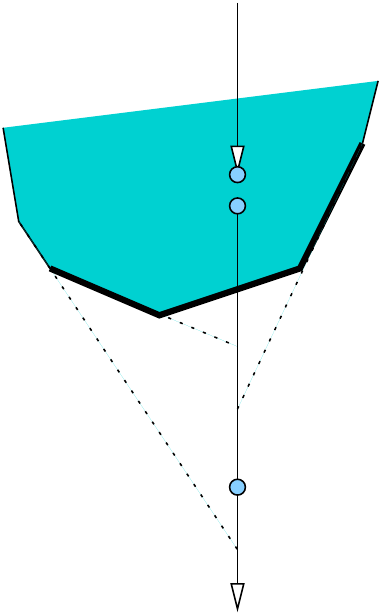
    \hspace*{50pt}
    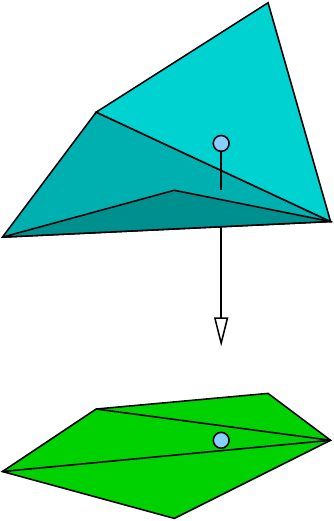
    \caption{The idea behind a line shelling, and the line shelling associated to a central orientation from a generic interior point~$o$.}
    \label{fig:shelling}
  \end{figure}

  For a regular triangulation~$\Delta$, let~$\tilde \Delta$ be the convex hypersurface that projects to~$\Delta$, and let~$\ell$ be the vertical line through~$o$.
  If we order the facets of~$\tilde \Delta$ (hence, the maximal faces of~$\Delta$) in their line shelling order with respect to~$\ell$, this ordering is clearly inducing the central orientation $\vec G (\Delta,o)$ of $G(\Delta)$.
  Moreover, it is a shelling order of the boundary of~$\tilde \Delta$ (and hence of~$\Delta$) because it is an initial segment in the line shelling order of $\conv(\tilde\Delta)$ with respect to the line~$\ell$.
  See Figure~\ref{fig:shelling} (right) for an illustration.
\end{proof}

\begin{remark}
  If the triangulation~$\Delta$ is not regular, then the central orientation $\vec G (\Delta,o)$ may contain cycles.
\end{remark}
  
\section{Geometry of the nonnesting and noncrossing complexes}
\label{sec:geometry}

The goal of this section is to show that the noncrossing complex $\NC_{k,n}$ is a regular, unimodular, 
flag triangulation of the order polytope of the product of two chains.
The method presented also yields the same result for the nonnesting complex $\NN_{k,n}$, which is well known.

\subsection{The order polytope of the product of two chains}

Recall from Section \ref{sec:orderpolytope} that the vertices of~$\oO_{k,n}$ (characteristic vectors of filters in~$P_{k,n}$) are in bijection with the vertices of~$\NN_{k,n}$ (elements of~$V_{k,n}$).
As before, we use the same symbol (typically~$I$ or~$J$) to denote an element of $V_{k,n}$ and its associated filter. 

To show that $\NC_{k,n}$ is a triangulation of $\oO_{k,n}$ let us understand a bit more the combinatorics 
of the maximal faces of $\oO_{k,n}$. These are of the following three types:
\begin{enumerate}
  \item There are two maximal faces corresponding to the unique minimal vector $I_{\hat0}:=(1,\ldots,k)$ and the unique 
    maximal vector $I_{\hat1}:=(n+1-k,\ldots,n)$.
    Each of these maximal faces contains all but one vertex, namely either $\chi_{I_{\hat0}} = (1,\ldots,1)$ or 
    $\chi_{I_{\hat1}} =(0,\ldots,0)$.
    In particular, $\oO_{k,n}$ is an iterated pyramid over these two vertices.

  \item Each of the $k(n-k-1)$ covering relations $(a,b) \lessdot (a,b+1)$ in $P_{k,n}$ produces a maximal face containing 
    all $\chi_{(i_1,\ldots,i_k)}$ with $i_{k+1-a}$ not equal to $k+1-a+b$.

  \item Each of the $(k-1)(n-k)$ covering relations $(a,b) \lessdot (a+1,b)$ produces a maximal face containing 
   all $\chi_{(i_1,\ldots,i_k)}$ with $i_{k+1-a}<k+1-a+b < i_{k+2-a}$. 
   (That is, vectors in $V_{k,n}$ not containing the entry $k+1-a+b$ and having exactly $a-1$ elements greater than $k+1-a+b$).
\end{enumerate}

%%%%%%%%%%%%%%%%%%%%%%%%%%%%%%%%%%%%%%%%%%%%%%%%

\subsection{Tableaux as lattice points in the cone of $\oO_{k,n}$}

Denote by $\oO^*_{k,n}$ the maximal face containing all vertices of $\oO_{k,n}$ except the 
origin $\chi_{I_{\hat1}} =(0,\ldots,0)$. Its vertex set is $V^*_{k,n} = V_{k,n}\setminus \{(n-k+1\cdots n)\}$, 
considered in Section~\ref{sec:combinatorics}. Since $\oO_{k,n}$ is a 
pyramid over $\oO^*_{k,n}$ with apex at the origin, it is natural to study the cone 
$\cC_{k,n}$ over $\oO^*_{k,n}$.
That is,
\[
  \cC_{k,n} :=\RR_{\ge 0} \oO_{k,n} = \{\lambda \vv\in \RR^{P_{k,n}}:
                                   \lambda \in[0,\infty), \vv \in \oO_{k,n}\}.
\]
Equivalently, $\cC_{k,n}$ is the polyhedron obtained from the inequality description of~$\oO_{k,n}$ by removing the inequality $x_{k,n-k}\le 1$.

Let us now look at the set  $\Tab_{k,n}$ of all tableaux of shape $k\times (n-k)$.
It is clear that the inequalities describing weak increase are the same as those defining the maximal faces of the cone $\cC_{k,n}$.
We hence have the following lemma.
\begin{lemma}
  \label{lemma:tableaux-as-lattice-points}
  $\Tab_{k,n}$ is the set of integer points in $\cC_{k,n}$.
\end{lemma}

Moreover, the summing tableau associated to a list of vectors $I_1,\,\dots, I_\ell\in V^*_{k,n}$ is nothing but the sum 
of the characteristic vectors $\chi_{I_1},\ldots, \chi_{I_\ell}$ of the corresponding vertices of $\oO_{k,n}$, see Lemma~\ref{lem:summingtableau}.
With this in mind, Theorem~\ref{thm:uniquemultiset} can be rewritten as follows.

\begin{proposition}
  \label{thmniquemultiset-lattice}
  For every integer point $T\in \cC_{k,n}$ there is a unique nonnegative integer combination of characteristic vectors 
  $\chi_I$ for $I \in V^*_{k,n}$ 
  with noncrossing support that gives~$T$, and another unique combination with nonnesting support that gives~$T$.
\end{proposition}

When translated into geometric terms and identifying faces of $\NN_{k,n}$ and $\NC_{k,n}$ with the convex hulls 
of the corresponding characteristic vectors, this proposition has the following consequence.

\begin{corollary}
  \label{coro:maximal face-triang}
  The restrictions of $\NN_{k,n}$ and $\NC_{k,n}$ to $V^*_{k,n}$ are flag unimodular triangulations of $\oO^*_{k,n}$.
\end{corollary}

\begin{proof}
  We provide the proof for $\NC_{k,n}$. The claim for $\NN_{k,n}$ follows in the same way.
  We use the characterization of triangulations via uniqueness of the convex combination of each $x\in \oO^*_{k,n}$ 
  as a convex combination of vertices of a face in $\NC_{k,n}$ (see Section~\ref{sub:triangulations}).
  Assume, to seek a contradiction, that there is an $x\in \oO^*_{k,n}$ that admits two different combinations whose 
  support is a face in $\NC_{k,n}$. That is, there are faces $S_1$ and $S_2$ in $\NC_{k,n}$ and positive real 
  vectors $\lambda\in \RR^{S_1}$, $\mu\in \RR^{S_2}$, such that
  \[
    \sum_{I\in S_1} \lambda_I \chi_I
        =
    \sum_{I\in S_2} \mu_I \chi_I.
  \]
  Assume further that $S_1$ and $S_2$ are chosen minimizing $|S_1|+|S_2|$ among the faces in $\NC_{k,n}$ with this 
  property. This implies that $S_1$ and $S_2$ are disjoint since common vertices can be eliminated from one side of the 
  equality, and that $\conv\{\chi_I:I\in S_1\}$ and $\conv\{\chi_I:I\in S_2\}$ intersect in a single point. This, in turn, 
  implies that this point, and the vectors $\lambda$ and $\mu$, are rational. Multiplying them by suitable constants 
  we consider them integral. But then the tableau
  \[
    T:=\sum_{I\in S_1} \lambda_I \chi_I = \sum_{I\in S_2} \mu_I \chi_I
  \]
  turns out to have two different noncrossing decompositions, contradicting Proposition~\ref{thmniquemultiset-lattice}.
\end{proof}

Since $\oO_{k,n}$ is a pyramid over $\chi_{(n-k+1,\ldots , n)} = (0,\ldots,0)$ and $(n-k+1,\ldots , n)$ is
noncrossing and nonnesting with every element of $V_{k,n}$, Corollary~\ref{coro:maximal face-triang} implies that 
both are unimodular triangulations of $\oO_{k,n}$.

\begin{theorem}
\label{thm:triangulation}
  $\NN_{k,n}$ and $\NC_{k,n}$ are flag unimodular triangulations of $\oO_{k,n}$.
\end{theorem}

The theorem follows 
from Stanley's work in~\cite{Sta1986} for $\NN_{k,n}$ and from~\cite[Corollary~7.2]{PPS2010} together with a mention of unimodularity 
in the proof of~\cite[Corollary~8.2]{PPS2010} for $\NC_{k,n}$.  

\subsection{$\NC_{k,n}$ as a regular triangulation of $\oO_{k,n}$}

For real parameters $\alpha_1 , \ldots, \alpha_{k-1}$ consider the following weight function on the set of vertices 
$V_{k,n}$ of $\oO_{k,n}$. For each $I=(i_1,\ldots,i_k)\in V_{k,n}$ let
\[
  w(I)=w(i_1,\ldots,i_k):= \sum_{1\le a < b \le k} \alpha_{b-a}i_a i_b.
\]
We assume that the values $\alpha_{1},\ldots,\alpha_{k-1}$ are positive 
and we require that $\alpha_{i+1}\ll\alpha_i$.

\begin{lemma}
\label{lemma:noncrossing-regular}
  Let $I,J$ and $X,Y$ be two different pairs of elements of $V_{k,n}$. If $I$ and $J$ are noncrossing and
  $\chi_I+\chi_J=\chi_X+\chi_Y$
  then $X$ and $Y$ are crossing and
  \[
    w(I)+w(J) < w(X) + w(Y).
  \]
\end{lemma}

\begin{proof}
  We set $I=(i_1,\cdots, i_k)$, $J=(j_1,\cdots, j_k)$, $X=(x_1,\cdots ,x_k)$ and  $Y=(y_1,\cdots ,y_k)$.
  Observe that $\chi_I+\chi_J=\chi_X+\chi_Y$
  implies $\{i_a,j_a\}=\{x_a,y_a\}$ for every $a\in [k]$. Since the pairs are different it follows that $X$ and $Y$ cross.
  For each pair $a,b\in [k]$ we then have two possibilities:
  \begin{itemize}
    \item $\big\{ \{x_a,x_b\}, \{y_a,y_b\}\big\}= \big\{ \{i_a,i_b\}, \{j_a,j_b\}\big\}$. We then say that $(X,Y)$ is consistent 
      with $(I,J)$ on the coordinates $a$ and $b$.
    \item $\big\{ \{x_a,x_b\}, \{y_a,y_b\}\big\} =\big\{ \{i_a,j_b\}, \{j_a,i_b\}\big\}$. We then say that $(X,Y)$ is inconsistent 
      with $(I,J)$ on the coordinates $a$ and $b$.
  \end{itemize}
  Observe that the difference $w(X)+w(Y) - w(I)-w(J)$ equals
  \[
    \sum_{a,b} \alpha_{b-a}(i_aj_b +j_ai_b -i_ai_b -j_aj_b) = - \sum_{a,b} \alpha_{b-a}(i_a-j_a)(i_b-j_b),
  \]
  where the sum runs over all inconsistent pairs of coordinates with $1\le a<b\le k$. Observe also that, by the choice of parameters 
  $\alpha_{b-a}$ the sign of this expression depends only on the inconsistent pairs that minimize $b-a$.
  We claim that all such ``minimal distance inconsistent pairs'' have the property that $i_\ell = j_\ell$ for all $a < \ell < b$. 
  This follows from the fact that if $a<c<b$ and $i_c\ne j_c$ then $(a,b)$ being inconsistent implies that one of $(a,b)$ and $(b,c)$ is also inconsistent.

  Then, the fact that~$I$ and~$J$ are not crossing implies that, for all such $a$ and $b$, we have that
  \[
    i_a<j_a<j_b<i_b
    \quad\text{or}\quad
    j_a<i_a<i_b<j_b.
  \]
  In any case, $(i_a-j_a)(i_b-j_b)<0$, so that $w(X)+w(Y) > w(I)+w(J)$.
\end{proof}

The regularity assertion of the following corollary was proved in~\cite[Theorem~8.1]{PPS2010}.

\begin{corollary}
  \label{coro:noncrossing-regular}
  $\NC_{k,n}$ is the regular triangulation of $\oO_{k,n}$ induced by the weight vector $w$.
\end{corollary}

\begin{proof}
  This follows from Lemmas~\ref{lemma:regular-easy} and~\ref{lemma:noncrossing-regular}.
\end{proof}

\begin{remark}
  The same ideas show that the nonnesting complex is the regular triangulation of $\oO_{k,n}$ produced by the opposite weight vector $-w$. 
  The only difference in the proof is that at the end, since $(I,J)$ is now the nonnesting pair, we have that
  \[
    i_a<j_a<i_b<j_b
    \quad\text{or}\quad
    j_a<i_a<j_b<i_b,
  \]
  so that $(i_a-j_a)(i_b-j_b)>0$ and $w(X)+w(Y) < w(I)+w(J)$, as needed.

  This means that $\NN_{k,n}$ and $\NC_{k,n}$ are in a sense ``opposite'' regular triangulations, although this should not be taken too literally. 
  What we claim for this particular~$w$ and its opposite $-w$ may not be true for other weight vectors $w$ producing the triangulation $\NC_{k,n}$. 
  Anyway, since $\NN_{k,n}$ is the \defn{pulling triangulation} of $\oO_{k,n}$ with respect to any of a family of orderings of the vertices 
  (any ordering compatible with comparability of filters), this raises the question whether $\NC_{k,n}$ is the \defn{pushing triangulation} 
  for the same orderings. See~\cite{DeloeraRambauSantos} for more on pushing and pulling triangulations.
\end{remark}

%%%%%%%%%%%%%%%%%%%%%%

\subsection{Codimension one faces of $\NC_{k,n}$, and the Gra{\ss}mann-Tamari order}
\label{sec:shelling}

Here we prove Theorem~\ref{thm:welldefinedTamari}; that is, that the Gra{\ss}mann-Tamari order is well-defined, 
and that any of its linear extensions is a shelling order for $\NC_{k,n}$.

The key idea is the use a central orientation of the dual graph, as introduced in 
Section~\ref{sec:central_orientations}, by explicitly describing the hyperplane 
containing interior codimension one faces in the triangulation of $\oO_{k,n}$ by $\NC_{k,n}$.

\medskip

First we define the bending vector $\bend_{(I,X)} \in \RR^{P_{k,n}}$ of a segment 
$X=(i_{a_1},\dots,i_{a_2})$ (where $1\le a_1<a_2\le k$) of an element $I=(i_1,\dots,i_k) \in V_{k,n}$. 
We assume $i_{a_2} < a_2+(n-k)$.
That is,~$X$ does not meet the east boundary of the grid.
This technical condition is related to the convention that, when marking the last occurrence of each inter in 
a row of a noncrossing table, we omit the mark for the maximal integer $a+n-k$ in row~$a$; compare Remark~\ref{rem:bars}.

The definition heavily relies on looking at~$I$ as a monotone path in the dual grid of $P_{k,n}$ as explained in Figure~\ref{fig:poset}, 
and $X$ as a connected subpath in it, starting and ending with a vertical step. See Figure~\ref{fig:bending_functional} where~$X$ is represented 
as a thick subpath in a longer monotone path. Along~$X$ we have marked certain parts with a $\ominus$-dot or a $\oplus$-dot. Namely when
traversing the path from top to bottom: 

\begin{itemize}
  \item the first and the last step in~$X$ (both vertical) are marked with a~$\ominus$-dot and with a $\oplus$-dot, respectively, and 
  \item the corners ~$X$ turns left or right are also marked with a~$\oplus$-dot and with a $\ominus$-dot, respectively.
\end{itemize}

Observe that the dots alternate between~$\oplus$ and~$\ominus$, and that there is the same number of both along~$X$.

\begin{figure}
  \centering
  \begin{tikzpicture}[scale=.55]
    \draw[step=1,black,thin,dotted] (0,0) grid (8,10);

    \draw[line width=1.5pt, black, ->] (0,10) -- (1,10);

    \draw[line width=1.5pt, black] (0,10) -- (2,10) -- (2,6) -- (5,6) -- (5,3) -- (6,3) -- (6,1) -- (8,1) -- (8,0);
    \draw[line width=3pt, red] (2,9) -- (2,6) -- (5,6) -- (5,3) -- (6,3) -- (6,1);

    \node[circle, inner sep=.2pt, fill=white, draw=black] at (2,8.5) {\tiny$-$};
    \node[circle, inner sep=.2pt, fill=white, draw=black] at (2,6  ) {\tiny$+$};
    \node[circle, inner sep=.2pt, fill=white, draw=black] at (5,6  ) {\tiny$-$};
    \node[circle, inner sep=.2pt, fill=white, draw=black] at (5,3  ) {\tiny$+$};
    \node[circle, inner sep=.2pt, fill=white, draw=black] at (6,3  ) {\tiny$-$};
    \node[circle, inner sep=.2pt, fill=white, draw=black] at (6,1.5) {\tiny$+$};
  \end{tikzpicture}
  \qquad
  \qquad
  \begin{tikzpicture}[scale=.55]
    \draw[step=1,black,thin,dotted] (0,0) grid (8,10);

    \draw[line width=1.5pt, black, ->] (0,10) -- (1,10);

    \draw[line width=1.5pt, black] (0,10) -- (2,10) -- (2,6) -- (5,6) -- (5,3) -- (6,3) -- (6,1) -- (8,1) -- (8,0);
    \draw[line width=3pt, red] (2,9) -- (2,6) -- (5,6) -- (5,3) -- (6,3) -- (6,1);

    \node at (2.5,8.5) {\small$-$};
    \node at (1.5,8.5) {\small$+$};
    \node at (2.5,6.5) {\small$+$};
    \node at (1.5,5.5) {\small$-$};
    \node at (5.5,6.5) {\small$-$};
    \node at (4.5,5.5) {\small$+$};
    \node at (5.5,3.5) {\small$+$};
    \node at (4.5,2.5) {\small$-$};
    \node at (6.5,3.5) {\small$-$};
    \node at (5.5,2.5) {\small$+$};
    \node at (6.5,1.5) {\small$+$};
    \node at (5.5,1.5) {\small$-$};
  \end{tikzpicture}
  \caption{A segment of a path in the dual grid of~$P_{10,18}$ and the corresponding bending vector in $\RR^{P_{k,n}}$.}
  \label{fig:bending_functional}
\end{figure}

The \defn{bending vector}~$\bend_{(I,X)}$ is the vector in $\{-1,0,+1\}^{P_{k,n}}$ obtained as indicated on the right side of Figure~\ref{fig:bending_functional}.
We start with the zero vector $\{ 0\}^{P_{k,n}}$.
Every time~$X$ bends to the right or to the left, we add a $(+1,-1)$-pair to the squares north-east and south-west to the corner, with the $-1$ 
added to the square inside and the $+1$ added to square outside the corner. Equivalently, the northeast square gets a $+1$ added in left turns ($\oplus$-corners) 
and a $-1$ added in right turns or ($\ominus$ corners) and vice versa for the southwest square.
Additionally, we add $+1$ to the square west of the first step and $-1$ to the square east of the first step and the other way round for the last step. 
Observe that if the first (or last) vertical step in~$X$ directly ends in a corner, then there is a cancellation of a $+1$ and a $-1$, see Example~\ref{ex:bendingvector} below.

\medskip

Let now $J = (j_1,\ldots,j_k) \in V_{k,n}$ be another vector, thought of as another monotone path in the dual grid of~$P_{k,n}$.
Then the scalar product of the bending vector~$\bend_{(I,X)}$ and the characteristic vector~$\chi_J$ is, by construction, given by the number of $\oplus$-dots minus the 
number of $\ominus$-dots of the path~$X$ that~$J$ goes through.
Indeed, a $(+1,-1)$-pair in the bending vector $\bend_{(I,X)}$ contributes to the scalar product if and only if $J$ goes through the corresponding dot, and with the stated sign.
In particular, the scalar product depends only on the intersection of the paths corresponding to~$X$ and to~$J$.
The next lemma explains how to compute the scalar product by summing over the contributions from the connected components of this intersection. Observe that a connected component may be a single point.

\begin{lemma}
\label{lem:contribution}
  Let ~$I,J \in V_{k,n}$ be noncrossing and let $X = (i_{a_1},\ldots, i_{a_2})$ be a segment in~$I$. 
  Considering the connected components of the intersection of the paths corresponding to~$X$ and to~$J$, we have that
  \begin{enumerate}
    \item a component contributes $-1$ to~$\langle \bend_{(I,X)}, \chi_J\rangle$ if and only if this component starts with 
       $j_{a_1} = i_{a_1}$ and this component of~$J$ leaves~$X$ to the left, \label{eq:component-1}

    \item a component contributes $1$ to~$\langle \bend_{(I,X)}, \chi_J\rangle$ if and only if this component ends with 
      $j_{a_2} = i_{a_2}$ and this component of~$J$ enters~$X$ from the right, \label{eq:component1}

    \item all other components contribute zero to~$\langle \bend_{(I,X)}, \chi_J\rangle$. \label{eq:component0}
  \end{enumerate}
\end{lemma}

\begin{proof}
  First, consider a component that passes neither through the initial $\ominus$-dot nor through the final $\oplus$-dot.
  Since~$I$ and~$J$ are noncrossing,~$J$ enters and exits~$X$ on opposite sides.
  It therefore passes through equally many $\oplus$- and $\ominus$-dots, so its contribution is zero.
  Next, consider a component that passes through the initial $\ominus$-dot.
  Equivalently, this component starts with~$j_{a_1} = i_{a_1}$.
  Then there are two possibilities.
  Either this component of~$J$ leaves~$X$ to the left, in which case it contains one more $\ominus$-dot than $\oplus$-dots, 
  i.e., it contributes $-1$ to~$\langle \bend_{(I,X)}, \chi_J\rangle$.
  Or this component of~$J$ leaves~$X$ to the right and then this component contains as many $\ominus$-dots as it contains $\oplus$-dots, 
  i.e., it does not contribute to~$\langle \bend_{(I,X)}, \chi_J\rangle$.
  The argument for a component that passes through the final $\oplus$-dot, is analogous.
\end{proof}

Let~$F$ be an interior face of codimension one.
Represent~$F$ by a table with two marked positions in a certain column~$I = (i_1,\ldots,i_k)$ and exactly one marked position 
in the others, as in the proofs of Proposition~\ref{prop:furtherproperties} and of Corollary~\ref{cor:pseudomanifold}.
Let $a_1<a_2$ be the rows containing the two marked positions in~$I$, and let~$F_1$ and~$F_2$ be the two maximal faces 
obtained from~$F$ by pushing the marked position in the $a_1$\textsuperscript{th} row and in the $a_2$\textsuperscript{th} row, respectively.
We then have the following lemma.

\begin{lemma}
  \label{lemma:ridge_hyperplane}
  The vector orthogonal to the hyperplane containing~$F$ is the bending vector~$\bend_{(I,X)}$ where~$X = (i_{a_1},\ldots,i_{a_2})$ is 
  the segment of $I$ starting with the first marked position and ending with the second (and last) marked position.

  Moreover,~$F_2$ is on the same side of that hyperplane as~$\bend_{(I,X)}$, while~$F_1$ is on the opposite side.
\end{lemma}

\begin{proof}
  Let~$J$ be another column in the table for~$F$.
  Since~$I$ and~$J$ are noncrossing we can use Lemma~\ref{lem:contribution} to compute $\langle \bend_{(I,X)}, \chi_J\rangle$.
  The result is zero since the intersection of the paths corresponding to~$X$ and to~$J$ cannot contain a component as described in Lemma~\ref{lem:contribution}\eqref{eq:component-1} or~\eqref{eq:component1}, as a direct consequence of the description of the marked positions  given at the end of Remark~\ref{rem:bars}.
  Observe here that if there was a component as described in Lemma~\ref{lem:contribution}\eqref{eq:component-1}, then $i_{a_1} = j_{a_1}$, and $(i_{a_1+1},\ldots,i_k)$ would be lexicographically smaller than $(j_{a_1+1},\ldots,j_k)$, contradicting that description.
  Similarly, if there was a component as described in Lemma~\ref{lem:contribution}\eqref{eq:component1}, then $i_{a_2} = j_{a_2}$, and $(i_1,\ldots,i_{a_2-1})$ is larger than $(j_1,\ldots,j_{a_2-1})$ in reverse lexicographic order, again contradicting that description.

  Moreover, the same argument shows that if~$J_1$ and~$J_2$ are the new elements in~$F_1$ and in~$F_2$ that were obtained by the pushing procedure, then the scalar product of~$\bend_{(I,X)}$ and~$F_1$ is negative, while the scalar product of~$\bend_{(I,X)}$ and~$F_2$ is positive.
\end{proof}

\begin{example}
\label{ex:bendingvector}
  Consider the following summing tableau.
  \begin{center}  
    \begin{tikzpicture}[node distance=0 cm,outer sep = 0pt]
      \tikzstyle{bsq}=[rectangle, draw,opacity=1,fill opacity=1, minimum width=.6cm, minimum height=.6cm]
      {\node[bsq] (11) at (1, 1) {5};}
      {\node[bsq] (12) [right = of 11] {8};}
      {\node[bsq] (13) [right = of 12] {9};}
      {\node[bsq] (14) [right = of 13] {11};}
      {\node[bsq] (21) [below = of 11] {3};}
      {\node[bsq] (22) [right = of 21] {4};}
      {\node[bsq] (23) [right = of 22] {8};}
      {\node[bsq] (24) [right = of 23] {9};}
      {\node[bsq] (31) [below = of 21] {1};}
      {\node[bsq] (32) [right = of 31] {3};}
      {\node[bsq] (33) [right = of 32] {5};}
      {\node[bsq] (34) [right = of 33] {7};}
      \node (00) [left = of 21] {$T=$};
    \end{tikzpicture}
  \end{center}
  Its associated noncrossing table, corresponding to a codimension one face in $\NC_{3,7}$, is given as follows.
  \smallskip
  \begin{center}
    \begin{tikzpicture}[node distance=0 cm,outer sep = 0pt]
      \tikzstyle{bsq}=[minimum width=.6cm, minimum height=.6cm]
      \tikzstyle{bsq2}=[rectangle,draw=black,opacity=.25,fill opacity=1]
      \tikzstyle{cor}=[anchor=north west,inner sep=1pt]

      \node[bsq,bsq2] (11) at (.3, 1.5)     {1};
      \node[bsq,bsq2] (12) [right = of 11]  {1};
      \node[bsq,bsq2] (13) [right = of 12]  {1};
      \node[bsq,bsq2] (14) [right = of 13]  {1};
      \node[bsq,bsq2] (15) [right = of 14]  {1};
      \node[bsq,bsq2] (16) [right = of 15]  {2};
      \node[bsq,bsq2] (17) [right = of 16]  {2};
      \node[bsq,bsq2] (18) [right = of 17]  {2};
      \node[bsq,bsq2] (19) [right = of 18]  {3};
      \node[bsq,bsq2] (110) [right = of 19]  {4};
      \node[bsq,bsq2] (111) [right = of 110]  {4};

      \node[bsq,bsq2] (21) at (.3, .9)      {2};
      \node[bsq,bsq2] (22) [right = of 21]  {2};
      \node[bsq,bsq2] (23) [right = of 22]  {2};
      \node[bsq,bsq2] (24) [right = of 23]  {4};
      \node[bsq,bsq2] (25) [right = of 24]  {6};
      \node[bsq,bsq2] (26) [right = of 25]  {3};
      \node[bsq,bsq2] (27) [right = of 26]  {4};
      \node[bsq,bsq2] (28) [right = of 27]  {4};
      \node[bsq,bsq2] (29) [right = of 28]  {4};
      \node[bsq,bsq2] (210) [right = of 29]  {5};
      \node[bsq,bsq2] (211) [right = of 210]  {6};

      \node[bsq,bsq2] (31) at (.3,.3)       {3};
      \node[bsq,bsq2] (32) [right = of 31]  {4};
      \node[bsq,bsq2] (33) [right = of 32]  {7};
      \node[bsq,bsq2] (34) [right = of 33]  {7};
      \node[bsq,bsq2] (35) [right = of 34]  {7};
      \node[bsq,bsq2] (36) [right = of 35]  {4};
      \node[bsq,bsq2] (37) [right = of 36]  {5};
      \node[bsq,bsq2] (38) [right = of 37]  {6};
      \node[bsq,bsq2] (39) [right = of 38]  {5};
      \node[bsq,bsq2] (310) [right = of 39]  {6};
      \node[bsq,bsq2] (311) [right = of 310]  {7};

      \draw[ thin] (2.7, 1.5) circle(0.20);
      \draw[ thin] (4.5, 1.5) circle(0.20);
      \draw[ thin] (5.1, 1.5) circle(0.20); % removed double
      \draw[ thin] (6.3, 1.5) circle(0.20); % removed double

      \draw[ thin] (1.5, 0.9) circle(0.20); 
      \draw[ thin] (2.1, 0.9) circle(0.20);
      \draw[ thin] (3.3, 0.9) circle(0.20);
      \draw[ thin] (5.7, 0.9) circle(0.20);
    
      \draw[ thin] (0.3, 0.3) circle(0.20);
      \draw[ thin] (0.9, 0.3) circle(0.20); % removed double
      \draw[ thin] (3.9, 0.3) circle(0.20);
      \draw[ thin] (4.5, 0.3) circle(0.20);

      \node  (00) [left = of r2]  {$F = $};
    \end{tikzpicture}
  \end{center}
  \smallskip
  
  The doubly marked column is $I=(2,4,6)$, with $X=(2,4,6)$ as well. But observe that, as monotone paths,~$X$ is a strict subpath of~$I$ since the initial and final horizontal steps in~$I$ are not part of~$X$.
  The bending vector of~$X$ equals
  \begin{center}  
    \begin{tikzpicture}[node distance=0 cm,outer sep = 0pt]
      \tikzstyle{bsq}=[rectangle, draw,opacity=1,fill opacity=1, minimum width=.6cm, minimum height=.6cm]
      {\node[bsq] (11) at (1, 1) {\bf1};}
      {\node[bsq] (12) [right = of 11] { 0};}
      {\node[bsq] (13) [right = of 12] {\bf-1};}
      {\node[bsq] (14) [right = of 13] {0};}
      {\node[bsq] (21) [below = of 11] {\bf-1};}
      {\node[bsq] (22) [right = of 21] {\bf1};}
      {\node[bsq] (23) [right = of 22] {\bf1};}
      {\node[bsq] (24) [right = of 23] {\bf-1};}
      {\node[bsq] (31) [below = of 21] {0};}
      {\node[bsq] (32) [right = of 31] {\bf-1};}
      {\node[bsq] (33) [right = of 32] { 0};}
      {\node[bsq] (34) [right = of 33] {\bf1};}
      \node (00) [left = of 21] {$\bend_{(I,X)}=$};
    \end{tikzpicture}
  \end{center}
  Observe that two of the zeroes are obtained by canceling a $+1$ and a $-1$ in the definition of 
  the bending vector).

  This is indeed orthogonal to the $11$ tableaux corresponding to the columns of $F$, which are given as follows.

  \newcommand{\tableau}[3]
  {
  %  \begin{center}
    \begin{tikzpicture}[node distance=0 cm,outer sep = 0pt]
      \tikzstyle{bsq}=[rectangle, draw,opacity=1,fill opacity=1, minimum width=.6cm, minimum height=.6cm]
      \tikzstyle{wsq}=[rectangle, draw,opacity=0,fill opacity=1, minimum width=.6cm, minimum height=.6cm]
      \ifthenelse{#1<2}{\node[bsq] (11) at (1, 1) {\bf 1};}{\node[bsq] (11) at (1, 1) {\bf 0};}
      \ifthenelse{#1<3}{\node[bsq] (12) [right = of 11] {1};}{\node[bsq] (12) [right = of 11] {0};}
      \ifthenelse{#1<4}{\node[bsq] (13) [right = of 12] {\bf 1};}{\node[bsq] (13) [right = of 12] {\bf 0};}
      \ifthenelse{#1<5}{\node[bsq] (14) [right = of 13] { 1};}{\node[bsq] (14) [right = of 13] {0};}
      \ifthenelse{#2<3}{\node[bsq] (21) [below = of 11] {\bf1};}{\node[bsq] (21) [below = of 11] {\bf0};}
      \ifthenelse{#2<4}{\node[bsq] (22) [right = of 21] {\bf1};}{\node[bsq] (22) [right = of 21] {\bf0};}
      \ifthenelse{#2<5}{\node[bsq] (23) [right = of 22] {\bf1};}{\node[bsq] (23) [right = of 22] {\bf0};}
      \ifthenelse{#2<6}{\node[bsq] (24) [right = of 23] {\bf1};}{\node[bsq] (24) [right = of 23] {\bf0};}
      \ifthenelse{#3<4}{\node[bsq] (31) [below = of 21] { 1};}{\node[bsq] (31) [below = of 21] {0};}
      \ifthenelse{#3<5}{\node[bsq] (32) [right = of 31] {\bf1};}{\node[bsq] (32) [right = of 31] {\bf0};}
      \ifthenelse{#3<6}{\node[bsq] (33) [right = of 32] {1};}{\node[bsq] (33) [right = of 32] {0};}
      \ifthenelse{#3<7}{\node[bsq] (34) [right = of 33] {\bf1};}{\node[bsq] (34) [right = of 33] {\bf0};}
      \node (10) [left = of 11] {#1};
      \node (20) [left = of 21] {#2};
      \node (30) [left = of 31] {#3};
    %    \node      (00) [left = of 21]  {$T =\ $};
    \end{tikzpicture}
  %  \end{center}
  }

  \begin{center}
    \begin{tabular}{ccc}
      \tableau123 &
      \tableau124 &
      \tableau127 \\
      \\
      \tableau147 &
      \tableau167 &
      \tableau234 \\
      \\
      \tableau245 &
      \tableau246 &
      \tableau345 \\
      \\
      \tableau456 &
      \tableau467 & 
      \\
    \end{tabular}
  \end{center}

  Let us also compute the maximal faces $F_1$ and $F_2$ obtained pushing the $2$ and the $6$ in the column corresponding to~$I$ in~$F$.
 
  \smallskip
  \begin{center}
    \begin{tikzpicture}[node distance=0 cm,outer sep = 0pt]
     \tikzstyle{bsq}=[minimum width=.6cm, minimum height=.6cm]
     \tikzstyle{bsq2}=[rectangle,draw=black,opacity=.25,fill opacity=1]
     \tikzstyle{cor}=[anchor=north west,inner sep=1pt]

      \node[bsq,bsq2] (11) at (.3, 1.5)     {1};
      \node[bsq,bsq2] (12) [right = of 11]  {1};
      \node[bsq,bsq2] (13) [right = of 12]  {1};
      \node[bsq,bsq2] (14) [right = of 13]  {1};
      \node[bsq,bsq2] (15) [right = of 14]  {1};
      \node[bsq,bsq2] (16) [right = of 15]  {2};
      \node[bsq,bsq2] (17) [right = of 16]  {2};
      \node[bsq,bsq2] (18) [right = of 17]  {2};
      \node[bsq,bsq2] (1n) [right = of 18]  {\bf2};
      \node[bsq,bsq2] (19) [right = of 1n]  {3};
      \node[bsq,bsq2] (110) [right = of 19]  {4};
      \node[bsq,bsq2] (111) [right = of 110]  {4};

      \node[bsq,bsq2] (21) at (.3, .9)      {2};
      \node[bsq,bsq2] (22) [right = of 21]  {2};
      \node[bsq,bsq2] (23) [right = of 22]  {2};
      \node[bsq,bsq2] (24) [right = of 23]  {4};
      \node[bsq,bsq2] (25) [right = of 24]  {6};
      \node[bsq,bsq2] (26) [right = of 25]  {3};
      \node[bsq,bsq2] (27) [right = of 26]  {4};
      \node[bsq,bsq2] (28) [right = of 27]  {4};
      \node[bsq,bsq2] (2n) [right = of 28]  {\bf 4};
      \node[bsq,bsq2] (29) [right = of 2n]  {4};
      \node[bsq,bsq2] (210) [right = of 29]  {5};
      \node[bsq,bsq2] (211) [right = of 210]  {6};

      \node[bsq,bsq2] (31) at (.3,.3)       {3};
      \node[bsq,bsq2] (32) [right = of 31]  {4};
      \node[bsq,bsq2] (33) [right = of 32]  {7};
      \node[bsq,bsq2] (34) [right = of 33]  {7};
      \node[bsq,bsq2] (35) [right = of 34]  {7};
      \node[bsq,bsq2] (36) [right = of 35]  {4};
      \node[bsq,bsq2] (37) [right = of 36]  {5};
      \node[bsq,bsq2] (38) [right = of 37]  {6};
      \node[bsq,bsq2] (3n) [right = of 38]  {\bf 7};
      \node[bsq,bsq2] (39) [right = of 3n]  {5};
      \node[bsq,bsq2] (310) [right = of 39]  {6};
      \node[bsq,bsq2] (311) [right = of 310]  {7};

      \draw[ thin] (2.7, 1.5) circle(0.20);
      \draw[ thin] (5.1, 1.5) circle(0.20);
      \draw[ thin] (5.7, 1.5) circle(0.20); % removed double
      \draw[ thin] (6.9, 1.5) circle(0.20); % removed double

      \draw[ thin] (1.5, 0.9) circle(0.20); 
      \draw[ thin] (2.1, 0.9) circle(0.20);
      \draw[ thin] (3.3, 0.9) circle(0.20);
      \draw[ thin] (6.3, 0.9) circle(0.20);
    
      \draw[ thin] (0.3, 0.3) circle(0.20);
      \draw[ thin] (0.9, 0.3) circle(0.20); % removed double
      \draw[ thin] (3.9, 0.3) circle(0.20);
      \draw[ thin] (4.5, 0.3) circle(0.20);

      \node  (00) [left = of r2]  {$F_1 = $};

    \end{tikzpicture}
  \end{center}
  \smallskip
  and

  \smallskip
  \begin{center}
    \begin{tikzpicture}[node distance=0 cm,outer sep = 0pt]
      \tikzstyle{bsq}=[minimum width=.6cm, minimum height=.6cm]
      \tikzstyle{bsq2}=[rectangle,draw=black,opacity=.25,fill opacity=1]
      \tikzstyle{cor}=[anchor=north west,inner sep=1pt]

      \node[bsq,bsq2] (11) at (.3, 1.5)     {1};
      \node[bsq,bsq2] (12) [right = of 11]  {1};
      \node[bsq,bsq2] (13) [right = of 12]  {1};
      \node[bsq,bsq2] (1n) [right = of 13]  {\bf1};
      \node[bsq,bsq2] (14) [right = of 1n]  {1};
      \node[bsq,bsq2] (15) [right = of 14]  {1};
      \node[bsq,bsq2] (16) [right = of 15]  {2};
      \node[bsq,bsq2] (17) [right = of 16]  {2};
      \node[bsq,bsq2] (18) [right = of 17]  {2};
      \node[bsq,bsq2] (19) [right = of 18]  {3};
      \node[bsq,bsq2] (110) [right = of 19]  {4};
      \node[bsq,bsq2] (111) [right = of 110]  {4};

      \node[bsq,bsq2] (21) at (.3, .9)      {2};
      \node[bsq,bsq2] (22) [right = of 21]  {2};
      \node[bsq,bsq2] (23) [right = of 22]  {2};
      \node[bsq,bsq2] (2n) [right = of 23]  {\bf4};
      \node[bsq,bsq2] (24) [right = of 2n]  {4};
      \node[bsq,bsq2] (25) [right = of 24]  {6};
      \node[bsq,bsq2] (26) [right = of 25]  {3};
      \node[bsq,bsq2] (27) [right = of 26]  {4};
      \node[bsq,bsq2] (28) [right = of 27]  {4};
      \node[bsq,bsq2] (29) [right = of 28]  {4};
      \node[bsq,bsq2] (210) [right = of 29]  {5};
      \node[bsq,bsq2] (211) [right = of 210]  {6};

      \node[bsq,bsq2] (31) at (.3,.3)       {3};
      \node[bsq,bsq2] (32) [right = of 31]  {4};
      \node[bsq,bsq2] (33) [right = of 32]  {7};
      \node[bsq,bsq2] (3n) [right = of 33]  {\bf6};
      \node[bsq,bsq2] (34) [right = of 3n]  {7};
      \node[bsq,bsq2] (35) [right = of 34]  {7};
      \node[bsq,bsq2] (36) [right = of 35]  {4};
      \node[bsq,bsq2] (37) [right = of 36]  {5};
      \node[bsq,bsq2] (38) [right = of 37]  {6};
      \node[bsq,bsq2] (39) [right = of 38]  {5};
      \node[bsq,bsq2] (310) [right = of 39]  {6};
      \node[bsq,bsq2] (311) [right = of 310]  {7};

      \draw[ thin] (3.3, 1.5) circle(0.20);
      \draw[ thin] (5.1, 1.5) circle(0.20);
      \draw[ thin] (5.7, 1.5) circle(0.20); % removed double
      \draw[ thin] (6.9, 1.5) circle(0.20); % removed double

      \draw[ thin] (1.5, 0.9) circle(0.20); 
      \draw[ thin] (2.7, 0.9) circle(0.20);
      \draw[ thin] (3.9, 0.9) circle(0.20);
      \draw[ thin] (6.3, 0.9) circle(0.20);
    
      \draw[ thin] (0.3, 0.3) circle(0.20);
      \draw[ thin] (0.9, 0.3) circle(0.20); % removed double
      \draw[ thin] (2.1, 0.3) circle(0.20);
      \draw[ thin] (4.5, 0.3) circle(0.20);

      \node  (00) [left = of r2]  {$F_2 = $};
     \end{tikzpicture}
  \end{center}
  \smallskip
  
  The new elements~$J_1$ and~$J_2$ and their characteristic vectors~$\chi_{J_1}$ and~$\chi_{J_2}$ in these maximal faces are
 
  \[
    J_1= \parbox{3cm}{\tableau247} 
    \qquad 
    J_2= \parbox{3cm}{\tableau 146}.
  \]
 \smallskip
 
  As predicted by Lemma~\ref{lemma:ridge_hyperplane},~$\chi_{J_1}$ has a negative scalar product with $\bend(I,X)$ and~$\chi_{J_2}$ has a positive scalar product.
\end{example}

In order to prove Theorem~\ref{thm:welldefinedTamari} we need to use one more property of bending vectors:

\begin{lemma}
  \label{lemma:positive}
   Let $o\in \RR^{P_{k,n-k}}$ be any vector such that
   \[
     o_{a_1,b_1} + o_{a_2,b_2} > o_{a_1,b_2} + o_{a_2,b_1}, \quad \mbox{~for all~} 1\le a_1 < a_2 \le k, \ 1 \le b_1 < b_2 \le n-k.
   \]
   Then, for every $I \in V_{k,n}$ and every segment $X$ in~$I$ we have
   \[
     \langle \bend(I,X), o \rangle >0.
   \]
   In particular, this holds for the vector $o$ defined as $o_{a,b}=ab$ for every $a,b\in P_{k,n-k}$.
\end{lemma}

\begin{proof}
  The following additivity property follows trivially form the definition of bending vectors: if $X=(i_{a_1}, \dots, i_{a_2})$ is a 
  segment, and we decompose it into two parts $X_1=(i_{a_1}, \dots, i_{a})$ and $X_2=(i_{a}, \dots, i_{a_2})$ via a certain 
  $a_1 < a < a_2$, then
  \[
    \bend(I,X) = \bend(I,X_1) + \bend(I,X_2).
  \]
  Via this property, we only need to prove the lemma for segments with only two entries.

  So, let $X=(i_a, i_{a+1})$ be such a segment. Its bending vector has exactly four nonzero entries. It has $+1$ in positions
  $(a, i_a-a)$ and $(a+1, i_{a+1} -a)$, and it has $-1$ in $(a+1, i_a-a)$ and $(a, i_{a+1} -a)$. 
  By choice of $o$,  $\langle \bend(I,X), o \rangle >0$.
\end{proof}

Putting together Lemmas~\ref{lemma:central-shelling},~\ref{lemma:ridge_hyperplane} and~\ref{lemma:positive} we can now easily 
prove Theorem~\ref{thm:welldefinedTamari}.

\begin{proof}[Proof of Theorem~\ref{thm:welldefinedTamari}.]
  By Lemmas~\ref{lemma:ridge_hyperplane} and~\ref{lemma:positive}, the Gra{\ss}mann-Tamari orientation in the dual graph of $\NC_{k,n}$ coincides 
  with the central orientation induced by any $o \in \RR^{P_{k,n-k}}$ satisfying the assumptions of Lemma \ref{lemma:positive}. 
  By Lemma~\ref{lemma:central-shelling} this orientation is acyclic and compatible with a shelling order of 
  the maximal faces in $\NC_{k,n}$.
\end{proof}

%%%%%%%%%%%%%%%%%%%%%%

\subsection{The nonnesting and noncrossing triangulations of the cube}
\label{sec:cube}
We saw in Section~\ref{sub:cubicalfaces} that every order polytope has some special faces that are cubes, namely the minimal face containing two given vertices.
Here we describe these faces for~$\oO_{k,n}$ and study the triangulations of them induced by~$\NN_{k,n}$ and~$\NC_{k,n}$.

Let~$\chi_I$ and~$\chi_J$ be two vertices of $\oO_{k,n}$. 
Remember from Section~\ref{sub:cubicalfaces} that the minimal face of $\oO_{k,n}$ containing~$\chi_I$ and~$\chi_J$ has the following description.
Let~$P_1,\ldots,P_d$ be the connected components of the poset given by the symmetric difference of the filters corresponding to $I$ and $J$.
Then~$\fF(I,J)$ is (affinely equivalent to) a cube of dimension~$d$ whose vertices are given as
\[
  \chi_{I\cap J} + \alpha_1 \chi_{P_1} +\cdots + \alpha_d \chi_{P_d}
\]
for the~$2^d$ choices of a vector $\alpha = (\alpha_1,\ldots,\alpha_d) \in\{0,1\}^d$.
In particular,~$\chi_I$ and~$\chi_J$ themselves correspond to certain antipodal vectors $\alpha_I,\alpha_J\in \{0,1\}^d$.

Our next observation is that the components $P_1,\ldots, P_d$ come with a natural linear order (and we consider them labelled according to that order).
Indeed, if we think of~$I$ and~$J$ as monotone paths in the $k\times(n-k)$ grid, 
the components are the regions that arise between the two paths. Since the paths are monotone, 
we consider those regions ordered from left to right (or, equivalently, from top to bottom).

\begin{example}
  Consider the two vectors
  \begin{align*}
    I &= (1,2,3,5,11,12,13,14,15,21,24),\\
    J &= (3,4,5,6, 7, 9,12,16,18,19,20)
  \end{align*}
  for $k=11, n = 24$ with paths shown in Figure~\ref{fig:cube}, compare also Figure~\ref{fig:poset}.
  The poset given by symmetric difference has the shown four connected components labelled~$1$ through~$4$ from left to right.
  Thus, $\fF(I,J)$ is combinatorially a $4$-dimensional cube.
  
  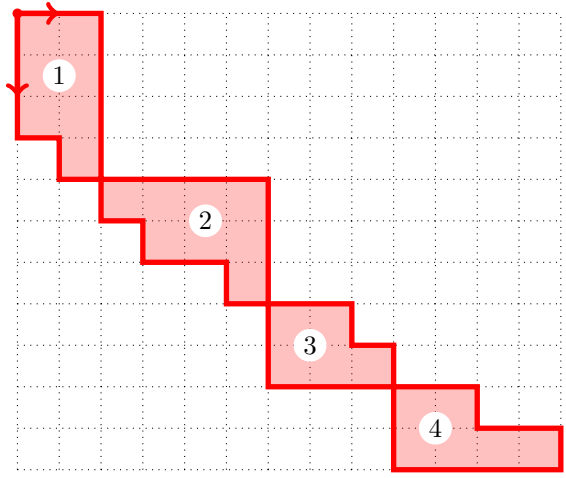
\begin{figure}
  \centering
    \begin{tikzpicture}[scale=.55]
      \draw[step=1,black,thin,dotted] (0,4) grid (13,15);

      \draw[line width=2pt, red, ->] (0,15) -- (0,13);
      \draw[line width=2pt, red, ->] (0,15) -- (1,15);

      \draw[fill, fill opacity=0.25, line width=2pt, red] (0,15) -- (0,12) -- (1,12) --(1,11) -- (5,11) -- (6,11) -- (6,6) -- (11,6) -- (11,5) -- (13,5) -- (13,4) -- (9,4) -- (9,7) -- (8,7) -- (8,8) -- (5,8) --(5,9) -- (3,9) -- (3,10) -- (2,10) -- (2,15) -- (0,15);

      \draw[color=red,fill=red] (0,15) circle (3pt);
      \node[circle, inner sep=1.5pt, fill=white] at (1.0,13.5) {$1$};
      \node[circle, inner sep=1.5pt, fill=white] at (4.5,10.0) {$2$};
      \node[circle, inner sep=1.5pt, fill=white] at (7.0, 7.0) {$3$};
      \node[circle, inner sep=1.5pt, fill=white] at (10 , 5.0) {$4$};
    \end{tikzpicture}
    \caption{Two paths in the dual grid of~$P_{11,24}$ and the corresponding four components.}
    \label{fig:cube}
  \end{figure}
\end{example}

With this point of view, the following lemma is straightforward.

\begin{lemma}
\label{lem:crossingnestingdiagonals}
  Let~$\chi_I$ and~$\chi_J$ be two vertices of $\oO_{k,n}$, let $d$ be the number of connected components of the symmetric
  difference of the filters corresponding to $I$ and $J$, and 
  let $\alpha_I, \alpha_J\in \{0,1\}^d$ be the 0/1-vectors identifying~$\chi_I$ and~$\chi_J$ as vertices of $\fF(I,J)$. Then
  \begin{itemize}
    \item $I,J$ are nonnesting if and only if $\{\alpha_I,\alpha_J\}=\big\{(0,\ldots,0),(1,\ldots,1)\big\}$.
    \item $I,J$ are noncrossing if and only if $\{\alpha_I,\alpha_J\}=\big\{(0,1,0,\ldots),(1,0,1,\ldots)\big\}$.
  \end{itemize}
\end{lemma}

Lemma~\ref{lem:crossingnestingdiagonals} has the following consequence.
\begin{corollary}
  With the same notation, let $\chi_X$ and $\chi_Y$ be two vertices of~$\fF(I,J)$, and let $\alpha_X, \alpha_Y\in \{0,1\}^d$ be the 0/1-vectors identifying them as vertices of~$\fF(I,J)$.
  Then
  \begin{itemize}
    \item $X,Y$ are nonnesting if and only if one of~$\alpha_X$ and~$\alpha_Y$ is coordinatewise smaller than the other.
    \item $X,Y$ are noncrossing if and only if,~$\alpha_X$ and~$\alpha_Y$ alternate between~$0$ and~$1$ in the coordinates in which they differ.
  \end{itemize}
\end{corollary}

This shows that if we restrict the triangulations induced by $\NN_{k,n}$ and $\NC_{k,n}$ to the cube $\fF(I,J)$ they coincide with the following 
well known triangulations of this cube.

\begin{itemize}
  \item $\NN_{k,n}$ induces the standard triangulation of the cube, understood as the order polytope of an antichain. 
    That is, it is obtained by slicing the cube along all hyperplanes of the form $x_i=x_j$.

  \item $\NC_{k,n}$ induces a flag triangulation whose edges are the 0/1-vectors that alternate relative to one another. It can also be described as 
    the triangulation obtained slicing the cube by all the hyperplanes of the form
    \[
       x_i+\cdots + x_j =z
    \]
    for every pair of coordinates $1\le i < j \le d$ and for every $z\in [j-i]$.
    We call it the \defn{noncrossing triangulation} of the cube.
\end{itemize}

\begin{remark}
  The triangulation of the cube induced by $\NC_{k,n}$ was first constructed by R.~Stanley~\cite{Stanley} and then 
  (as a triangulation of each of the \defn{hypersimplices}) by B.~Sturmfels~\cite{Sturmfels1996}. 
  T.~Lam and A.~Postnikov~\cite{LamPostnikov} showed the two constructions to coincide.

  Both the triangulation induced by $\NN_{k,n}$ and $\NC_{k,n}$ are images of the \emph{dicing triangulation} of an \emph{alcoved polytope} for the root system~$A_d$.
  Alternatively, one can say they are (after a linear transformation) Delaunay triangulations of a fundamental parallelepiped in the lattice~$A_d^*$.
  Dicing triangulations of alcoved polytopes in type~$A$ are always regular, unimodular and flag~\cite{LamPostnikov}.

  Note that the nonnesting triangulation of $\oO_{k,n}$ is also induced by hyperplane cuts, but the noncrossing triangulation is not.
  This can be seen, for example, in $\NC_{2,5}$, see Figure~\ref{fig:52} on page~\pageref{fig:52}.
\end{remark}

\begin{remark}
  The (dual graph) diameters of the standard and the noncrossing triangulations of the cube $\fF(I,J)$ are easy to compute 
  from the fact that they are obtained by hyperplane cuts: the distance between two given 
  maximal faces equals the number of cutting hyperplanes that separate them. 

  In the standard triangulation every maximal face is separated from its opposite one by all the $\binom{d}{2}$ 
  hyperplanes, so the diameter equals $\binom{d}{2}$.
  For the noncrossing triangulation a slightly more complicated argument gives that 
  the diameter is $\binom{d+1}{3}$.
  In both bounds, $d\le \min\{k,n-k\}$ is the dimension of the cubical face $\fF(I,J)$.
\end{remark}

\subsection{$\oO_{k,n}$ as a Cayley polytope}

Let $\Delta_{\ell}=\conv\{v_1,\ldots, v_\ell\}$ be an $(\ell-1)$-dimensional unimodular simplex and let $Q_1,\ldots,Q_\ell$ be lattice polytopes in $\RR^k$.
We do not require the individual~$Q_i$'s to be full dimensional, but we require it for their 
Minkowski sum. The \defn{Cayley sum} or \defn{Cayley embedding} of the~$Q_i$'s is the $(k+\ell)$-dimensional polytope
\[
  \cC(Q_1,\ldots,Q_k):=\conv\big\{\bigcup_{i=1}^k Q_i\times \{v_i\}\big\} \subset \RR^{k+\ell}.
\]
We show that for each of the $k$ rows (or for each of the $n-k$ columns)
of the poset $P_{k,n}$ we can derive a representation of $\oO_{k,n}$ as
a Cayley sum. Let $a\in [k]$ be fixed, and for each 
$b\in [0,\ldots,n-k]$ let $V_{k,n}^{a,b}$ be the set of vectors 
from $V_{k,n}$ that have an $a+b$ in their $a$\textsuperscript{th} entry. 
In terms of tableaux, the vectors in~$V_{k,n}^{a,b}$ correspond to those tableaux that have a~$0$ in position$(a,b)$ and a~$1$ in position $(a,b+1)$.
Denote by~$\oO_{k,n}^{a,b}$ the convex hull of $V_{k,n}^{a,b}$.
We then have the following lemma.
\begin{lemma}
  \label{lem:cayley}
  For every $a\in [k]$,
  \[
    \oO_{k,n} = \cC(\oO_{k,n}^{a,0}, \ldots, \oO_{k,n}^{a,n-k}).
  \]
\end{lemma}

This has enumerative consequences for the numbers of nonnesting (i.e., standard Young) and noncrossing tableaux.
\begin{definition}
  Let $Q_1,\ldots,Q_\ell$ be an $\ell$-tuple of polytopes in $\RR^k$. 
  For each $m=(m_1,\ldots,m_\ell)\in \mathbb{N}^\ell$ with 
  $\sum m_i=k$  (equivalently, for each monomial of degree $k$ in 
  $\RR[x_1,\ldots,x_\ell]$) call the coefficient of 
  $\xx^m$ in the homogeneous  polynomial
  \[
     \operatorname{vol}(x_1Q_1 + \cdots + x_\ell Q_\ell).
  \]
  the \defn{$m$-mixed volume} of $Q_1,\ldots,Q_\ell$.
  Here the volume is meant normalized to the lattice. 
  That is, unimodular simplices are considered to have volume $1$.
\end{definition}

We will need the following consequence of the 
Cayley trick as given in~\cite[Theorem~9.2.18]{DeloeraRambauSantos}.

\begin{lemma}
  Let $Q_1,\ldots,Q_\ell \subseteq \RR^k$ be an $\ell$-tuple of polytopes 
  and let $\Delta$ be a unimodular triangulation of $\cC(Q_1,\ldots,Q_\ell)$.
  Then, for each  tuple $m=(m_1,\ldots,m_\ell) \in \mathbb{N}^\ell$ of 
  sum $k$, the $m$-mixed volume of the tuple equals the number of maximal 
  faces of $\Delta$ that have exactly $m_b+1$ vertices in each fiber 
  $\{v_b\}\times Q_b$.
\end{lemma}

Applied to the representation of $\oO_{k,n}$ as a Cayley sum from 
Lemma \ref{lem:cayley} and taking into account that both $\NN_{k,n}$ and 
$\NC_{k,n}$ are unimodular triangulations of $\oO_{k,n}$, 
the previous lemma has the following corollary.

\begin{corollary}
  Fix an $a\in [k]$ and let $t=(t_1,\ldots, t_{n-k})$ be a 
  vector with $a\le t_1<t_2<\dots<t_{n-k} \le k(n-k) +a -k$. 
  From $t$ we derive a partition $m=(m_0,\ldots, m_{n-k})$ of 
  $k(n-k)-(n-k)-(n-k)$ by setting $m_b=t_{b+1}-t_b-1$, with the 
  conventions $t_0=0$ and $t_{n-k+1}=k(n-k)+1$.

  Then the following numbers coincide:
  \begin{itemize}
    \item The number of maximal faces of $\NN_{k,n}$ whose summing tableaux 
      have their $a$\textsuperscript{th} row equal to~$t$.
    \item The number of maximal faces of $\NC_{k,n}$ whose summing tableaux 
      have their $a$\textsuperscript{th} row equal to~$t$.
    \item The $m$-mixed volume of the tuple 
      $\oO_{k,n}^{a,0}, \ldots, \oO_{k,n}^{a,n-k}$.
  \end{itemize}
\end{corollary}

Note that the inequalities $a\le t_1<t_2<\dots<t_{n-k} \le k(n-k) +a -k$
in the corollary are necessary and sufficient for $t$ to appear as the 
$a$\textsuperscript{th} row in some standard tableau. 

\begin{example}
  Consider the case $k=2$ and $n=5$.
  The nonnesting and noncrossing complexes each have five maximal faces, corresponding to the following tableaux:
  \newcommand{\tabtwothree}[6]
  {
  \begin{tikzpicture}[node distance=0 cm,outer sep = 0pt]
    \tikzstyle{bsq}=[rectangle, draw, minimum width=.6cm, minimum height=.6cm]
    \node[bsq] (11) at (1, 1)       {#1};
    \node[bsq] (12) [right = of 11] {#2};
    \node[bsq] (13) [right = of 12] {#3};
    \node[bsq] (21) [below = of 11] {#4};
    \node[bsq] (22) [right = of 21] {#5};
    \node[bsq] (23) [right = of 22] {#6};
%  \node      (00) [left = of 21]  {$T =\ $};
  \end{tikzpicture}
  }

  \[
    \NN_{2,5}=
    \left\{
       \begin{array}{ccccc}
         \tabtwothree123456 &
         \tabtwothree124356 &
         \tabtwothree125346 &
         \tabtwothree134256 &
         \tabtwothree135246
       \end{array}
    \right\}
  \]
  \[
    \NC_{2,5}=\left\{
    \begin{array}{ccccc}
      \tabtwothree123246 &
      \tabtwothree124256 &
      \tabtwothree125356 &
      \tabtwothree134346 &
      \tabtwothree135456
    \end{array}
    \right\}
  \]
  As can be seen, the multisets of $a$-rows are the same in both complexes.
  \begin{align*}
    \text{First rows} &= \{123,124,125,134,135\}, \\
    \text{Second rows} &= \{246,256,356,346,456\}, \\
  \end{align*}
  Of course, symmetry under exchange of $k$ and $n-k$ implies that the same happens for columns:
  \begin{align*}
    \text{First columns} &= \{12,12,13,13,14\}, \\
    \text{Second columns} &= \{24,25,25,34,35\}, \\
    \text{Third columns} &= \{36,46,46,56,56\}.
  \end{align*}

\end{example}

\section{Relation to the weak separation complex}
\label{sec:separated}

\subsection{The weak separation complex}

B.~Leclerc and A.~Zelevinsky~\cite{LZ1998} introduced the complex of \emph{weakly separated subsets} of $[n]$ and showed that 
its faces are the sets of pairwise quasi-commuting quantum {P}l\"ucker coordinates in a $q$-deformation of the coordinate ring of 
the flag variety resp. the Gra\ss mannian.
A geometric version of their definition is that two subsets $X$ and $Y$ of $[n]$ are weakly separated if, when considered as subsets 
of vertices in an $n$-gon,  the convex hulls of $X\setminus Y$ and $Y\setminus X$ are disjoint. 
If we restrict to $X$ and $Y$ of fixed size~$k$ then we are in the Gra\ss mann situation and the the following complex was studied 
by J.~Scott~\cite{Sco2005,Sco2006} as a subcomplex of the Leclerc-Zelevinsky complex.

\begin{definition}
  Let $I$ and $J$ be two vectors in $V_{k,n}$. We say that $I$ and $J$ are \defn{weakly separated} if, considered as subsets of 
  vertices of an $n$-gon, the convex hulls of $I\setminus J$ and $J\setminus I$ do not meet.
  We denote by~$\WS_{k,n}$ the simplicial complex of subsets of $V_{k,n}$ whose elements are pairwise weakly separated.
\end{definition}

J.~Scott conjectured that~$\WS_{k,n}$ is pure of dimension $k(n-k)$ and that it is strongly connected (that is, its dual graph is connected). 
Both conjectures were shown to hold by S.~Oh, A.~Postnikov and D.~Speyer~\cite{OPS2011}, for the first one see also 
V.~I.~Danilov, A.~V.~Karzanov, and G.~A.~Koshevoy~\cite[Prop.~5.9]{DKK2010}.

It is not hard to see that~$\WS_{k,n}$ is a subcomplex of $\NC_{k,n}$ and it is trivial to observe that $\WS_{k,n}$ is invariant under cyclic 
(or, more strongly, dihedral) symmetry. Our next results combines these two properties and states that $\WS_{k,n}$ is the ``cyclic part'' of $\NC_{k,n}$. 
In the following statement, we denote by $I^{+i}$ the cyclic shift of $I\in V_{k,n}$ by the amount $i$, i.e., the image of $I$ under the 
map $x\to x+i$ considered as remainders $1,\ldots,n$ modulo~$n$.

\begin{proposition}
\label{prop:cyclicsymmetry}
  The weak separation complex~$\WS_{k,n}$ is equal to the intersection of all cyclic shifts of the noncrossing complex~$\NC_{k,n}$:
  \begin{enumerate}
    \item If $I$ and $J$ are weakly separated, then they are noncrossing. \label{eq:separabletocrossing}
    \item  If $I^{+i}$ and $J^{+i}$ are noncrossing for every $i\in [n]$, then $I$ and $J$ are weakly separated.\label{eq:crossingtoseparable}
  \end{enumerate}
\end{proposition}
\begin{proof}
  Assume that $1 \leq a < b \leq k$ are indices such that for $(i_1,\ldots, i_k)$ and $(j_1,\ldots, j_k)$ in $V_{k,n}$ we have $i_\ell = j_\ell$ for $a < \ell < b$ and $(i_a,i_b)$ and $(j_a,j_b)$ cross.
  Then we must have $i_a,i_b \in I\setminus J$ and $j_a,j_b \in J \setminus I$.
  Thus the convex hulls of $I\setminus J$ and $J \setminus I$ intersect nontrivially and hence~$I$ and~$J$ cannot be weakly separated.
  This completes the proof of~\eqref{eq:separabletocrossing}.

  Next, recall that we have seen in Proposition~\ref{prop:intro}\eqref{eq:symmetricdifference} that the crossing or noncrossing of $I$ and~$J$ only depends on $I\triangle J$, as does weak separability by definition.
  Hence, there is no loss of generality in assuming that~$I$ and~$J$ are complementary to one another.
  Think of~$I$ as a cyclic sequence of pluses and minuses, indicating horizontal and vertical steps when you look at~$I$ as a monotone path in the $k\times (n-k)$ grid (the complementarity assumption, of course, implies $k=n-k$), compare~\figref{fig:poset} in the introduction.
  Then~$I$ and~$J$ are opposite sequences.
  Observe that~$I$ and its complement~$J$ are weakly separated if and only if~$I$ changes signs (considered cyclically) exactly twice.
  In other words,~$I$ and~$J$ are weakly separated if and only if, after some cyclic shift, it consists of $n/2$ pluses followed by $n/2$ minuses.
  Suppose that~$I$ is not of that form, and it changes sign at least four times.
  Let~$a$ and~$b$ be the lengths of the first two maximal constant subsequences.
  We then are in one of the following two situations.
  \begin{itemize}
    \item If $a > b$, shifting the sequence by $a-b$ we get that~$I$ starts with two constant subsequences of the same length $b$.
    This implies that~$I$ crosses its complement.

    \item If $a < b$, shifting the sequence by $2a$ we get that $I$ finishes with two constant subsequences of the same length $a$.
    This implies as well that~$I$ crosses its complement.
  \end{itemize}
  Statement~\eqref{eq:crossingtoseparable} follows.
\end{proof}

The fact that $\WS_{k,n}$ is a subcomplex of $\NC_{k,n}$ was already observed in~\cite[Lem. 2.10]{PPS2010}.
Proposition~\ref{prop:cyclicsymmetry} and the results from~\cite{OPS2011} immediately imply the following corollary.

\begin{corollary}
  $\WS_{k,n}$ is a full dimensional pure flag subcomplex of $\NC_{k,n}$.
\end{corollary}

\subsection{A conjecture on the topology of $\WS_{k,n}$}

The (cyclic) intervals considered in Proposition~\ref{prop:furtherproperties}\eqref{eq:join} were shown to be vertices in all maximal 
faces of~$\NC_{k,n}$. They are as well vertices in all maximal faces of $\WS_{k,n}$.
Thus it makes sense to look at the \emph{reduced} weak separation complex $\WSred_{k,n}$, which is a subcomplex of $\NCred_{k,n}$.
Observe that both complexes coincide for $k=2$.
This follows for example from Proposition~\ref{prop:cyclicsymmetry} and the fact that $\NC_{2,n}$ is cyclic symmetric.
Hence, $\WSred_{2,n}$ is an $(n-4)$-sphere, the dual associahedron. 
The topology of $\WSred_{k,n}$ and its generalizations motivated by the work of B.~Leclerc and A.~Zelevinsky~\cite{LZ1998} has been scrutinized in~\cite{HessHirsch2011,HessHirsch2013}.
In the preliminary version~\cite{HessHirsch2011} they make detailed conjectures about the topology of the various complexes based on computer experiments.
In particular, they observe that for every~$k$ and~$n$ the complex $\WSred_{k,n}$ appears to have the same homology as an $(n-4)$-sphere.
We state their conjecture in this case and remark that in the first version of the present paper, we had independently come to the same conjecture.
We learned about the preliminary version~\cite{HessHirsch2011} from Vic Reiner after the first version of the present paper was released on the arxiv.

%Observe that, surprisingly, this appears to be independent of~$k$.
%We pose this as a conjecture.
\begin{conjecture}
  \label{conj:ws-sphere}
  For every $2\le k\le n-2$ the reduced complex $\WSred_{k,n}$ of weakly separated $k$-subsets of $[n]$ is homotopy equivalent to the $(n-4)$-sphere.
\end{conjecture}

\bibliographystyle{amsalpha} %amsplain, amsalpha
\bibliography{../SantosStumpWelker}

\end{document}